\UseRawInputEncoding
\documentclass[two-sides,10pt]{amsart}

\usepackage{indentfirst,latexsym,bm}
\usepackage{amsfonts}
\usepackage{amssymb}
\usepackage{amsmath}
\usepackage{amsbsy}
\usepackage{dsfont}
\usepackage{amsthm}
\usepackage{hyperref}
\usepackage{amscd}
\usepackage[all]{xy}
\usepackage{chemarrow}
\usepackage{multirow}
\usepackage[titletoc]{appendix}

\begin{document}
\title[Pointed Hopf algebras of dimension 16 in characteristic $2$]{On non-connected   pointed Hopf algebras of dimension 16 in characteristic $2$}
\author{Rongchuan Xiong}
\address{Department of Mathematics, Changzhou University, Changzhou 213164, China}
\email{rcxiong@foxmail.com}

\date{}
\maketitle

\newcommand{\tabincell}[2]{\begin{tabular}{@{}#1@{}}#2\end{tabular}}
\newtheorem{question}{Question}
\newtheorem{defi}{Definition}[section]
\newtheorem{conj}{Conjecture}
\newtheorem{thm}[defi]{Theorem}
\newtheorem{lem}[defi]{Lemma}
\newtheorem{pro}[defi]{Proposition}
\newtheorem{cor}[defi]{Corollary}
\newtheorem{rmk}[defi]{Remark}
\newtheorem{example}{Example}[section]

\theoremstyle{plain}
\newcounter{maint}
\renewcommand{\themaint}{\Alph{maint}}
\newtheorem{mainthm}[maint]{Theorem}

\newcommand{\AdL}{\text{ad}_L\,}
\newcommand{\AdR}{\text{ad}_R\,}
\newcommand{\K}{\mathds{k}}
\newcommand{\A}{\mathcal{A}}
\newcommand{\C}{\mathcal{C}}
\newcommand{\M}{\mathcal{M}}
\newcommand{\E}{\mathcal{E}}
\newcommand{\D}{\mathcal{D}}
\newcommand{\G}{\mathbf{G}}
\newcommand{\Z}{\mathbb{Z}}
\newcommand{\I}{\mathbb{I}}
\newcommand{\J}{\mathcal{J}}
\newcommand{\bJ}{\mathbb{J}}
\newcommand{\BN}{\mathcal{B}}
\newcommand{\Lam}{\lambda}
\newcommand{\Ome}{\omega}
\newcommand{\HYD}{{}^{H}_{H}\mathcal{YD}}
\newcommand{\KYD}{{}^{\widetilde{H}}_{\widetilde{H}}\mathcal{YD}}
\newcommand{\As}{^\ast}
\newcommand{\N}{\mathds{N}}
\newcommand{\Pp}{\mathcal{P}}
\newcommand{\HH}{\widetilde{H}}
\newcommand{\KK}{\mathcal{K}}
\newcommand\ad{\operatorname{ad}}
\newcommand\Ob{\operatorname{Ob}}
\newcommand{\Alg}{\Hom_{\text{alg}}}
\newcommand\Aut{\operatorname{Aut}}
\newcommand{\AuH}{\Aut_{\text{Hopf}}}
\newcommand\coker{\operatorname{coker}}
\newcommand\car{\operatorname{char}}
\newcommand\Der{\operatorname{Der}}
\newcommand\diag{\operatorname{diag}}
\newcommand\End{\operatorname{End}}
\newcommand\mult{\operatorname{mult}}
\newcommand\id{\operatorname{id}}
\newcommand\Char{\operatorname{char}}
\newcommand\gr{\operatorname{gr}}
\newcommand\GK{\operatorname{GKdim}}
\newcommand{\Hom}{\operatorname{Hom}}
\newcommand\ord{\operatorname{ord}}
\newcommand\rk{\operatorname{rk}}
\newcommand\Soc{\operatorname{soc}}
\newcommand\lgot{\operatorname{l}}
\newcommand\Top{\operatorname{top}}
\newcommand\supp{\operatorname{supp}}
\newcommand{\bp}{\mathbf{p}}
\newcommand{\bq}{\mathbf{q}}
\newcommand\Sb{\mathbb S}
\newcommand\cR{\mathcal{R}}
\newcommand\cH{\mathcal{H}}

\begin{abstract}
Let $\mathds{k}$ be an algebraically closed field. We give a complete  classification of non-connected pointed Hopf algebras of dimension $16$ with $\operatorname{char}\mathds{k}=2$ that are generated by group-like elements and skew-primitive elements. It turns out that there are infinitely many classes (up to isomorphism) of
pointed Hopf algebras of dimension 16. In particular, we obtain infinitely many new examples of non-commutative   non-cocommutative finite-dimensional pointed Hopf algebras.

\end{abstract}
\section{Introduction}
Let $\mathds{k}$ be an algebraically closed field of positive characteristic. It is a difficult question to classify Hopf algebras over $\K$ of a given dimension. Indeed, the complete classifications have been done only for prime dimensions  (see \cite{NW18}).  One may obtain partial classification results by  determining Hopf algebras with some properties. To date, pointed ones  are the class best classified.

Let $p,q,r$ be distinct prime numbers and $\Char\K=p$. G. Henderson    classified cocommutative connected Hopf algebras  of dimension less than or equal to $p^3$ \cite{Hen95};    X. Wang classified connected Hopf algebras   of dimension $p^2$ \cite{W1} and  pointed ones with L. Wang \cite{WW};  V. C. Nguyen, L. Wang and X. Wang determined connected Hopf algebras of dimension $p^3$  \cite{NWW1,NWW2}; Nguyen-Wang \cite{NW} studied the classification of non-connected pointed Hopf algebras of dimension $p^3$ and classified coradically graded ones; motivated by \cite{SO, NW}, the author   gave a complete classification of pointed Hopf algebras of dimension $pq$, $pqr$, $p^2q$, $2q^2$, $4p$ and pointed Hopf algebras of dimension $pq^2$ whose diagrams are Nichols algebras \cite{X17}.  It should be mentioned that   S. Scherotzke classified finite-dimensional pointed Hopf algebras   whose infinitesimal braidings are one-dimensional and the diagrams are Nichols algebras \cite{S}; N. Hu, X. Wang and Z. Tong  constructed many examples of pointed Hopf algebras of dimension $p^n$ for some $n\in\N$ via quantizations of the restricted universal enveloping algebras of the restricted modular simple Lie algebras of Cartan type, see \cite{HW07,HW11,THW15,TH16}; C. Cibils, A. Lauve and S. Witherspoon constructed several examples of finite-dimensional pointed Hopf algebras whose diagrams are Nichols algebras of Jordan type \cite{CLW}; The authors in  \cite{AAH19,ABFF,AP}  constructed some examples of finite-dimensional coradically graded pointed Hopf algebras whose diagram are Nichols algebras of non-diagonal type. Until now, it is still an open question to give a complete classification of non-connected pointed Hopf algebras of dimension $p^3$ or pointed ones of dimension $pq^2$ whose diagrams are not Nichols algebras for odd prime numbers $p,q$.

In this paper, we study the isomorphism classification of  non-connected pointed Hopf algebras with $\Char\K=2$ of dimension $16$  that are generated by group-like elements and skew-primitive elements. Indeed,  S. Caenepeel, S. D\u{a}sc\u{a}lescu and S. Raianu  \cite{CDR00} classified all pointed complex Hopf algebras of dimension $16$. It turns out that there are finitely many isomorphism classes and all are generated by group-like elements and skew-primitive elements.

The strategy follows the ideas in \cite{AS98b}, that is, the so-called  Lifting Method. Let $H$ be a finite dimensional Hopf algebra such that the coradical $H_0$ is a Hopf subalgebra, then $\gr H$, the graded coalgebra of $H$ associated to the coradical filtration, is a Hopf algebra with projection onto the coradical $H_0$. Then there is a connected graded braided Hopf algebra $R=\oplus_{n=0}^{\infty}R(n)$ in ${}^{H_0}_{H_0}\mathcal{YD}$ such that $\gr H\cong R\sharp H_0$.  We call $R$ and $R(1)$ the \emph{diagram} and \emph{infinitesimal braiding} of $H$, respectively.   Furthermore,  the diagram $R$ is coradically graded and the subalgebra generated  by $V$ is the so-called Nichols algebra $\BN(V)$ over $V:=R(1)$, which plays a key role in the classification of pointed complex Hopf algebras. In particular,   pointed Hopf algebras are generated by group-like elements and skew-primitive elements if and only if the diagrams are Nichols algebras.

By means of the Lifting Method \cite{AS98b}, we classify all non-connected Hopf algebras of dimension $16$ with $\Char\K=2$ whose diagrams are Nichols algebras. See Theorem \ref{thm:16-diagram-Nichols-algebra} for the classification results. Contrary to the case of characteristic zero, there exist infinitely many isomorphism classes, which provides a counterexample to Kaplansky's 10-th conjecture, and  there are infinitely many classes of pointed Hopf algebras of dimension $16$ with non-abelian coradicals.

Besides, we also classify pointed Hopf algebras of dimension $p^4$ with some properties, see e.g. Theorem \ref{thm:p4-x1y0z0}.  In particular, we obtain infinitely many new examples of non-commutative   non-cocommutative finite-dimensional pointed Hopf algebras.

The paper is organized as below: In section  \ref{secPre}, we introduce necessary notations and  materials that we will need to study pointed Hopf algebras in positive characteristic. In section \ref{sec:p4}, we study  pointed Hopf algebras   of dimension $p^4$ with some properties. In section \ref{sec:16}, we classify non-connected pointed Hopf algebras of dimension $16$ whose diagrams are Nichols algebras. The classification of  pointed ones whose diagrams are not Nichols algebras is much more difficult and requires different techniques, such as   the Hochschild cohomology of coalgebras   (see e.g. \cite{SO,NW,WZZ}). We treat them in a subsequent work.

\section{Preliminaries}\label{secPre}
\subsection*{Conventions}
We work over an algebraically closed field $\K$ of positive characteristic.  Denote by $\Char\K$ the characteristic of $\K$,  by $\N$ the set of natural numbers, and by $\Z_n$ the cyclic group of order $n$.  $\K^{\times}=\K-\{0\}$.   Given~$n\geq k\geq 0$,    $\I_{k,n}=\{k,k+1,\ldots,n\}$. Let $C$ be a coalgebra. Then the set $\G(C):=\{c\in C\mid \Delta(c)=c\otimes c,\ \epsilon(c)=1\}$ is called the set of \emph{group-like} elements of $C$. For any $g,h\in\G(C)$, the set  $\Pp_{g,h}(C):=\{c\in C\mid \Delta(c)=c\otimes g+h\otimes c\} $ is called the space of $(g,h)$-\emph{skew primitive elements} of $C$. In particular, the linear space $\Pp(C):=\Pp_{1,1}(C)$ is called the set of \emph{primitive elements}. Unless otherwise stated, ``pointed" refers to ``nontrivial pointed" in our context.

Our references for Hopf algebra theory are \cite{R11}.

\subsection{Yetter-Drinfeld modules and bonsonizations}
Let $H$ be a Hopf algebra with bijective antipode and ${}^{H}_{H}\mathcal{YD}$  the category of left Yetter-Drinfeld modules over $H$. As well-known, ${}^{H}_{H}\mathcal{YD}$ is braided monoidal with the braiding $c_{V,W}$ for $V,W\in {}^{H}_{H}\mathcal{YD}$ given by
\begin{align}\label{equbraidingYDcat}
c_{V,W}:V\otimes W\mapsto W\otimes V,\ v\otimes w\mapsto v_{(-1)}\cdot w\otimes v_{(0)},\ \forall\,v\in V, w\in W.
\end{align}
In particular, $c:=c_{V,V}$ is a linear isomorphism satisfying the
braid equation $(c\otimes\text{id})(\text{id}\otimes c)(c\otimes\text{id})=(\text{id}\otimes c)(c\otimes\text{id})(\text{id}\otimes c)$, that is, $(V,c)$ is a braided vector space.

\begin{rmk}\label{rmk:dimV=1}
Let~$V\in\HYD$~such that~$\dim V=1$. Let $\{v\}$ be a basis of  $V$. By definition, there is an algebra map~$\chi:H\rightarrow\K$~and~$g\in\G(H)$~satisfying
\begin{align*}
h_{(1)}\chi(h_{(2)})g=gh_{(2)}\chi(h_{(1)}),
\end{align*}
such that~$\delta(v)=g\otimes v$, $h\cdot v=\chi(h)v$. Moreover, $g$~lies in the center of $\G(H)$.
\end{rmk}

\begin{rmk}\cite[Remark 1.5]{AS02}
Suppose that $H=\K[G]$, where $G$ is a group. We write ${}_G^G\mathcal{YD}$ for the category of Yetter-Drinfeld modules over $\K[G]$. Let $V\in{}_G^G\mathcal{YD}$. Then $V$ as a $G$-comodule is just a $G$-graded vector space $V:=\oplus_{g\in G}V_g$, where $V_g:=\{v\in V\mid \delta(v)=g\otimes v\}$.

Assume in addition that the action of $G$ is diagonalizable, that is, $V=\oplus_{\chi\in \widehat{G}}V^{\chi}$, where $V^{\chi}:=\{v\in V\mid g\cdot v=\chi(g)v,\;\forall g\in G\}$.  Then
\begin{align*}
V=\oplus_{g\in G,\chi\in \widehat{G}}V_{g}^{\chi}, \text{ where }V_{g}^{\chi}=V_g\cap V^{\chi}.
\end{align*}
\end{rmk}

Let $G$ be a finite group. For any $g\in G$, we denote by $\mathcal{O}_g$ the conjugacy class of $g$, by $C_G(g)$ the isotropy subgroup of $g$ and by $\mathcal{O}(G)$ be the set of conjugacy classes of $G$. For any  $\Omega\in\mathcal{O}(G)$, fix $g_{\Omega}\in\Omega$, then $G=\sqcup_{\Omega\in\mathcal{O}(G)}\mathcal{O}_{g_{\Omega}}$ is a decomposition of conjugacy classes of $G$. Let $\psi:\K[C_G(g_{\Omega})]\rightarrow \End(V)$ be a representation of  $\K[C_G(g_{\Omega})]$, denoted by $(V,\psi)$. Then the induced module $M(g_{\Omega},\psi):=\K[G]\otimes_{\K[C_G(g_{\Omega})]}V$ can be an object in ${}_G^G\mathcal{YD}$ by
\begin{gather*}
h\cdot (g\otimes v)=hg\otimes v, \quad \delta(g\otimes v)=gg_{\Omega}g^{-1}\otimes (g\otimes v),\quad h,g\in G, v\in V.
\end{gather*}
In particular, $\dim M(g_{\Omega},\psi)=[G,C_G(g_{\Omega})]\times \dim V$. Furthermore, indecomposable objects in ${}_G^G\mathcal{YD}$ are indexed by the pairs  $(V,\psi)$, see e.g. \cite{M,Witherspoon}.
 \begin{thm}\cite{M,Witherspoon}\label{thm:indecomposable-object-YD-over-groups}
$M(g_{\Omega},\psi)$ is an indecomposable object in ${}_G^G\mathcal{YD}$ if and only if $(V,\psi)$ is an indecomposable $\K[C_G(g_{\Omega})]$-module. Furthermore, any indecomposable object in ${}_G^G\mathcal{YD}$ is isomorphic to $M(g_{\Omega},\psi)$ for some $\Omega\in\mathcal{O}(G)$ and indecomposable $\K[C_G(g_{\Omega})]$-module $(V,\psi)$.
\end{thm}

Let $\Z_{p^s}:=\langle g\rangle$ and $\Char\K=p$. Then the $p^s$  non-isomorphic indecomposable $\Z_{p^s}$-modules consist of $r$-dimensional modules $V_r=\K\{v_1,v_2,\cdots,v_r\}$ for $r\in\I_{1,p^s}$,   whose   module structure given by
\begin{gather*}
g\cdot v_1=v_1,\quad g\cdot v_{m}=v_{m}+v_{m-1},\quad 1<m\leq r.
\end{gather*}

The following well-known result follows directly by Theorem \ref{thm:indecomposable-object-YD-over-groups}. See e.g.\cite{DC} for details.
 \begin{pro}\label{pro:indecomposable-object-cyclic-p-group}
Let $\Z_{p^s}:=\langle g\rangle$ and $\Char\K=p$. The  indecomposable objects in ${}_{\Z_{p^s}}^{\Z_{p^s}}\mathcal{YD}$ consist of $r$-dimensional objects $M_{i,r}:=M(g^i,V_r)=\K\{v_1,v_2,\cdots,v_r\}$ for $r\in\I_{1,p^s}$, $i\in\I_{0,p^s-1}$, whose Yetter-Drinfeld module structure is given by
\begin{gather*}
g\cdot v_1=v_1,\quad g\cdot v_{m}=v_{m}+v_{m-1},\quad 1<m\leq r;\quad \delta(v_n)=g^i\otimes v_n,\quad n\in\I_{1,r}.
\end{gather*}
\end{pro}

We consider Hopf algebras in ${}^{H}_{H}\mathcal{YD}$. For any finite-dimensional graded Hopf algebra in ${}^{H}_{H}\mathcal{YD}$, it satisfies the Poincar\'{e} duality:
\begin{pro}\cite[Proposition\,3.2.2]{AG99}\label{pro-P-duality}
Let $R=\oplus_{n=0}^N R(i)$ be a graded Hopf algebra in ${}^{H}_{H}\mathcal{YD}$, and suppose that $R(N)\neq 0$. Then $\dim R(i)=\dim R(N-i)$ for any $0\leq i<N$.
\end{pro}

 Let $R$ be a braided Hopf algebra in ${}^{H}_{H}\mathcal{YD}$. We write $\Delta_R(r)=r^{(1)}\otimes r^{(2)}$ for the comultiplication to avoid confusions. The \emph{bosonization or Radford biproduct} $R\sharp H$ of $R$ by $H$ is a   Hopf algebra over $\K$  defined as follows:
$R\sharp H=R\otimes H$ as a vector space, and the multiplication and comultiplication are given by the smash product and smash coproduct, respectively:
\begin{align*}
(r\sharp g)(s\sharp h)=r(g_{(1)}\cdot s)\sharp g_{(2)}h,\quad
\Delta(r\sharp g) =r^{(1)}\sharp (r^{(2)})_{(-1)}g_{(1)}\otimes (r^{(2)})_{(0)}\sharp g_{(2)}.
\end{align*}
See \cite{R11} for  details.

\subsection{Braided vector spaces and Nichols algebras}
We follow \cite{AS02} to introduce the definition of Nichols algebras.

Let  $(V, c)$ be a braided vector space. Then the tensor algebra $T(V)=\oplus_{n\geq 0}T^n(V):=\oplus_{n\geq 0}V^{\otimes n}$ admits a connected braided Hopf algebra structure in the usual way.

Let $\mathbb B_n$ be the braid group  presented by generators $(\tau_j)_{j \in \I_{1,n-1}}$
with the defining relations
\begin{align}\label{eq:braid-rel}
\tau_i\tau_j = \tau_j\tau_i,\quad \tau_i  \tau_{i+1}\tau_i = \tau_{i+1}\tau_i\tau_{i+1}, \quad\text{ for } i\in\I_{1,n-2},~j\neq i+1.
\end{align}

Then there exists naturally the representation $\varrho_n$ of
$\mathbb B_n$ on $T^n(V)$ for $n \geq 2$ given by $$\varrho_n: \tau_j \mapsto c_j:=\id_{V^{\otimes (j-1)}} \otimes c \otimes \id_{V^{\otimes (n - j-1)}}.$$

Let  $M_n: \mathbb{S}_n \rightarrow \mathbb B_n$ be the (set-theoretical)
Matsumoto section, that preserves the length and satisfies $M_n(s_j) = \tau_j$.
Then the \emph{quantum symmetrizer} $\Omega_n:V^{\otimes n}\rightarrow V^{\otimes n}$  is defined by
\begin{align*}
 \Omega_n = \sum_{\tau \in \mathbb{S}_n} \varrho_n (M_n(\tau)).
\end{align*}

\begin{defi}\label{defi-Nicholsalgebra}
Let $(V, c)$ be a braided vector space. The Nichols algebra $\BN(V)$ is  defined by
\begin{align}
\BN(V) =  T(V) / \J(V),\quad \text{where }\J(V) = \oplus_{n\geq 2}\J^n(V) \text{ and }\J^n(V) = \ker \Omega_n.
\end{align}
\end{defi}
Indeed, $\J(V)$ coincides with the largest homogeneous ideal of $T(V)$ generated by elements of degree bigger than $2$ that is also a coideal. Moreover, $\BN(V)=\oplus_{n\geq 0}\BN^n(V)$ is a connected $\N$-graded Hopf algebra.

\begin{example}
A braided vector space $(V, c)$ of rank $m$ is said to be of diagonal type, if there exists a basis $\{x_i\}_{i\in\I_{1,m}}$ such that $c(x_i\otimes x_j)=q_{ij}x_j\otimes x_i$ for $q_{i,j}\in\K^{\times}$.     Rank 2 and 3 Nichols algebras of   diagonal type with finite PBW-generators were classified in \cite{WH,W}.
\end{example}

\begin{example}\label{ex-1}
A braided vector space $(V,c)$ of rank $m>1$ is said to be of Jordan type, denoted by $\mathcal{V}(s,m)$, if there exists a basis $\{x_i\}_{i\in\I_{1,m}}$ such that
\begin{gather*}
c(x_i\otimes x_1)=sx_1\otimes x_i,\quad \text{and}\quad c(x_i\otimes x_j)=(sx_j+x_{j-1})\otimes x_i,\quad i\in\I_{1,m},j\in\I_{2,m}.
\end{gather*}

Let $V_i=\K\{x_1,x_2,\cdots,x_i\}$ for $i\in\I_{1,m}$. It is clear that $0\subset V_1\subset V_2\cdots \subset V_{m}=\mathcal{V}(s,m)$ is a flag of   $\mathcal{V}(s,m)$. Then $\gr\mathcal{V}(s,m)$ is of diagonal type with braiding $(q_{i,j})_{i,j\in\I_{1,m}}$, $q_{i,j}=s$ for all $i,j\in\I_{1,m}$.
In particular, if $\Char\K=p$, then $\dim\BN(\mathcal{V}(1,m))\geq \dim\BN(\gr\mathcal{V}(1,m))=p^m$. See e.g. \cite[3.4]{AAH19} for more details.
\end{example}

\begin{pro}\cite{AS02}
A $\N$-graded Hopf algebra $R=\oplus_{n\geq 0} R(n)$ in ${}^{H}_{H}\mathcal{YD}$ is a Nichols algebra if and only if
\begin{itemize}
\item[(1)] $R(0)\cong\K$,\quad
 $(2)$ $\Pp(R)=R(1)$,\quad
 $(3)$ $R$~ is generated  as an algebra by~$R(1)$.
\end{itemize}
\end{pro}

Recall that an object in the category of Yetter-Drinfeld modules is a braided vector space.
\begin{pro}\cite[Theorem 5.7]{T00}\label{pro-Nichols-YD-Realization}
Let $(V,c)$ be a rigid braided vector space. Then  $\BN(V)$ can be realized as a braided Hopf algebra in $\HYD$ for some Hopf algebra $H$.
\end{pro}

\begin{rmk}
 By Definition \ref{defi-Nicholsalgebra} and Proposition \ref{pro-Nichols-YD-Realization},
 $\BN(V)$ depends only on $(V, c_{V,V})$ and  the same braided vector space  can be realized in $\HYD$ in many ways and for many $H$'s.
\end{rmk}

\subsection{Useful results in positive characteristic}
We introduce some important skills in positive characteristic.
Set $(\AdL x)(y) := [x, y]$ and $(x)(\AdR y)=[x, y]$.
\begin{pro}\cite[p. 186-187]{J}\label{proJ}
Let $A$ be any associative algebra over a field of characteristic $p$. For any $a, b\in A$,
\begin{gather*}
(\AdL a)^p(b)=[a^p,b],\quad
(\AdL a)^{p-1}(b)=\sum_{i=0}^{p-1}a^{i}ba^{p-1-i};\\
(a)(\AdR b)^{p}=[a,b^p],\quad
(a)(\AdR b)^{p-1}=\sum_{i=0}^{p-1}b^{p-1-i}ab^i.
\end{gather*}
Furthermore,
\begin{align*}
(a + b)^p=a^p+b^p+\sum_{i=1}^{p-1}s_i(a,b),
\end{align*}
where $is_i(a,b)$ is the coefficient of $\lambda^{i-1}$ in $(a)(\AdR \lambda a+b)^{p-1}$, $\lambda$ an indeterminate.
\end{pro}

\begin{lem}\label{pqlem1}\cite[Lemma\,5.1(1)]{NW}
Assume that  $\Char\K=p>0$. Let $A$ be an associative algebra over $\K$ with generators $g$ and $x$, subject to  the relations $g^{p^n}=1, gx-xg=  g(1-g)$. Then
\begin{description}
  \item[(1)] $(g)(\AdR x)^{p-1}=g-g^p$, $[g,x^p]=(g)(\AdR x)^{p}=[g,x]$.
  \item[(2)] $(\AdL x)^{p-1}(g)=g-g^p$, $[x^p, g]=(\AdL x)^p(g)=[x,g]$.
\end{description}
\end{lem}

\begin{lem}\label{pqlem2}\cite[p.423]{NW}
Let $\Char\K=p>0$ and $\mu\in\I_{1,p-1}$. Let $A$ be an   associative algebra generated by $g$, $x$, $y$. Assume that the  relations
\begin{gather*}
g^{p}=1, \quad gx-xg=\lambda_1(g-g^2),\quad gy-yg=\lambda_2(g-g^{\mu+1}),\\x^p-\lambda_1x=0,\quad y^p-\lambda_2y=0,\quad xy-yx+\mu\lambda_1 y-\lambda_2x=\lambda_3(1-g^{\mu+1}),
\end{gather*}
hold in $A$ for some $\lambda_1,\lambda_2\in\I_{0,1},\lambda_3\in\K$. Then
\begin{description}
  \item[(1)] $(x)(\AdR y)^n=\lambda_2^{n-1}(x)(\AdR y)-\lambda_3\sum_{i=0}^{n-2}\lambda_2^i(g^{\mu+1})(\AdR y)^{n-1-i}$ for $n>0$. In particular,    $(x)(\AdR y)^p=\lambda_2^{p-1}(x)(\AdR y)$.
  \item[(2)] $(\AdL x)^n(y)=(-\mu\lambda_1)^{n-1}(\AdL x)(y)-\lambda_3\sum_{i=0}^{n-2}(-\mu\lambda_1)^i(\AdL x)^{n-1-i}(g^{\mu+1})$. In particular,   $(\AdL x)^p(y)=(-\mu\lambda_1)^{p-1}(\AdL x)(y)$.
  \end{description}
\end{lem}

Now we introduce the following proposition, which is useful to determine when a coalgebra   map is one-one.
\begin{pro}\cite[Proposition 4.3.3]{R11}\label{pro:R11-4.3.3}
Let $C, D$ be coalgebras over $\K$ and $f: C \rightarrow D$ is a coalgebra map. Assume that $C$ is pointed. Then the following are equivalent:
\begin{itemize}
  \item[(a)] $f$~is one-one.
  \item[(b)] For any~$g,h\in\G(C)$, $f|_{\Pp_{g,h}(C)}$~is one-one.
  \item[(c)] $f|_{C_1}$~is one-one.
\end{itemize}
\end{pro}

\section{On pointed Hopf algebras of dimension $p^4$}\label{sec:p4}
Let $p$ be a prime number and $\Char\K=p$. We study pointed Hopf algebras of dimension $p^4$ with some properties, which will be used to obtain our main results.  In particular, we obtain some classification results of pointed Hopf algebras of dimension $p^4$ with some properties. We mention that N. Andruskiewitsch and H. J. Schneider classified pointed complex Hopf algebra of $p^4$ for an odd prime $p$ \cite{AS00b}; S. Caenepeel, S. D\u{a}sc\u{a}lescu and S. Raianu classified all pointed complex Hopf algebras of dimension $16$ \cite{CDR00}; and the Hopf subalgebra of dimension $p^3$   have already appeared in \cite{NW}.

\begin{lem}\label{lem:cyclic-groups-dimV=2}
Let $\Char\K=p$, $\Z_{p^s}:=\langle g\rangle$ and $V$ be an object  in ${}_{\Z_{p^s}}^{\Z_{p^s}}\mathcal{YD}$ such that $\dim\BN(V)=p^2$. Then $\dim V=2$. Furthermore,
\begin{itemize}
  \item If $\BN(V)$ is of diagonal type, then $V\cong M_{i,1}\oplus M_{j,1}$ for $i,j\in\I_{0,p^s-1}$  or $M_{k,2}$ for $p\mid k\in\I_{0,p^s-1}$ and hence $\BN(V)\cong\K[x,y]/(x^p,y^p)$.
  \item If $\BN(V)$ is not of diagonal type, then $p>2$, $V\cong M_{i,2}$ for $p\nmid i\in\I_{1,p^s-1}$  and hence
   $\BN(V)\cong \K\langle x,y\rangle/(x^p,y^p,yx-xy+\frac{1}{2}x^2)$.
\end{itemize}
\end{lem}
\begin{proof}
Assume that $\dim V=1$. Then by Proposition \ref{pro:indecomposable-object-cyclic-p-group}, $V\cong M_{i,1}$ for $i\in\I_{0,p^s-1}$  and hence  $\dim\BN(V)=p$, a contradiction.

Assume that $\dim V=2$. Then by Proposition \ref{pro:indecomposable-object-cyclic-p-group}, $V\cong M_{i,1}\oplus M_{j,1}$ for $i,j\in\I_{0,p^s-1}$  or $M_{k,2}$ for $k\in\I_{0,p^s-1}$.

If $V\cong M_{i,1}\oplus M_{j,1}$ for $i,j\in\I_{0,p^s-1}$, then $V$ is of diagonal type with trivial braiding, which implies that  $\BN(V)\cong\K[x,y]/(x^p,y^p)$.

If $V\cong M_{k,2}:=\K\{x,y\}$ for $k\in\I_{0,p^s-1}$, then the braiding of $V$ is
\begin{align*}
  c(\left[\begin{array}{ccc} x\\y \end{array}\right]\otimes\left[\begin{array}{ccc} x~y \end{array}\right])=
   \left[\begin{array}{ccc}
    x\otimes x    & (y+kx)\otimes x  \\
   x\otimes y   & (y+kx)\otimes y
   \end{array}\right].
  \end{align*}

If $p\mid k$, then $V$ is of diagonal type with trivial braiding and hence $\BN(V)\cong\K[x,y]/(x^p,y^p)$.

If $p\nmid k$, then $V$ is of Jordan type and hence by \cite[Theorem 3.1 and 3.5]{CLW}, $p>2$ and $\BN(V)\cong \K\langle x,y\rangle/(x^p,y^p,yx-xy+\frac{1}{2}x^2)$.

Assume that $\dim V>2$.  Then by Proposition \ref{pro:indecomposable-object-cyclic-p-group}, $V\cong \oplus_{i=1}^rM_{n_i,m_i}$ with $n_i\in\I_{0,p^s-1}$, $m_i>0$ and $\sum_{i=0}^rm_i=\dim V$.  If there exists some $i\in\I_{1,r}$ such that $m_i\geq 2$, then there must be a two-dimensional subobject $W$ of $V$, which is isomorphic to    $M_{n_i,2}$; otherwise $W\cong M_{i,1}\oplus M_{j,1}$ for $i,j\in\I_{0,p^s-1}$. Hence from the preceding discussions, $\dim\BN(V)>\dim\BN(W)\geq p^2$. Consequently, the assertions hold.
\end{proof}
\begin{rmk}\label{rmk-11-1--1}
Let $G$ be a finite group and $V\in{}_G^G\mathcal{YD}$ with $\dim V=2$. Then  by \cite[Proposition 3.3]{NW}, $V$ is either of diagonal type or of Jordan type. In particular, if $V$ is of Jordan type such that $\dim\BN(V)=p^2$, then $\Char\K=p>2$.
\end{rmk}

\begin{lem}\label{lem:cyclic-groups-dimV=3}
Let $\Char\K=p$, $\Z_p:=\langle g\rangle$ and $V$ be a decomposable object  in ${}_{\Z_p}^{\Z_p}\mathcal{YD}$ such that $\dim\BN(V)=p^3$. Then $\dim V=3$. Furthermore,
\begin{itemize}
  \item If $\BN(V)$ is of diagonal type, then $V\cong M_{i,1}\oplus M_{j,1}\oplus M_{k,1}$ for $i,j,k\in\I_{0,p-1}$  or $M_{0,2}\oplus M_{0,1}$ and hence $\BN(V)\cong\K[x,y,z]/(x^p,y^p,z^p)$.
  \item If $\BN(V)$ is not of diagonal type, then $p>2$, $V\cong M_{i,2}\oplus M_{0,1}$ for $i\in\I_{1,p-1}$ and hence
   $\BN(V)\cong \K\langle x,y,z\rangle/(x^p,y^p,z^p,yx-xy+\frac{1}{2}x^2,[x,z],[y,z])$.
\end{itemize}
\end{lem}
\begin{proof}
We first claim that $\dim V\geq 3$. Indeed,  if $\dim V=1$, then $V\cong M_{i,1}$ and hence $\dim\BN(V)=p$; if $\dim V=2$, then $V\cong M_{i,1}\oplus M_{j,1}$  or $M_{i,2}$ for $i,j\in\I_{0,p-1}$ and hence $V$ is of diagonal type or of Jordan type. By \cite[Proposition 3.7]{NW}, $\dim\BN(V)=p^2$ or $16$. Consequently, the claim follows.

Assume that $\dim V=3$. Then by assumption,  $V\cong M_{i,1}\oplus M_{j,1}\oplus M_{k,1}$  or $M_{i,2}\oplus M_{j,1}$ for $i,j,k\in\I_{0,p-1}$.

If $V\cong M_{i,1}\oplus M_{j,1}\oplus M_{k,1}$  for $i,j,k\in\I_{0,p-1}$, then  $V$ is of diagonal type with trivial braiding and hence $\BN(V)\cong\K[x,y,z]/(x^p,y^p,z^p)$.

If $V\cong M_{i,2}\oplus M_{j,1}:=\K\{x,y\}\oplus\K\{z\}$ for $i,j\in\I_{0,p-1}$, then the braiding of $V$ is
 \begin{align*}
  c(\left[\begin{array}{ccc} x\\y\\z\end{array}\right]\otimes\left[\begin{array}{ccc} x~y~z\end{array}\right])=
   \left[\begin{array}{ccc}
    x\otimes x    & (y+ix)\otimes x & z \otimes x\\
   x\otimes y   & (y+ix)\otimes y & z\otimes y\\
   x\otimes z, & (y+jx)\otimes z& z\otimes z
         \end{array}\right].
  \end{align*}

If   $i=0=j$, then $V$ has trivial braiding  and hence $\BN(V)\cong\K[x,y,z]/(x^p,y^p,z^p)$.

If $i=0$ and $j\neq 0$, then $V$ is not of diagonal type, which also appeared in \cite[7.1]{AAH19}. We claim that $\dim\BN(V)>p^3$. Indeed, if $p>2$, then by \cite[Proposition 7.1]{AAH19}, $\dim\BN(V)=2^pp^2$; if $p=2$, then the proof following the same lines. Indeed, it is easy to show that $\{x^iy^j[z,x]^kz^k\}_{i,j,k,l\in\I_{0,1}}$ is linearly independent in $\BN(V)$.

If $i\neq 0$ and $j\neq 0$, then without loss of generality, we assume that $i=1$. In this case, $V$ is  not of diagonal type, which also appeared in \cite[4.1]{AAH19}. If $p=2$, then by \cite[Theorem 3.1]{CLW}, $\dim\BN(V)>16$, a contradiction. If $p>2$, then by \cite[Proposition 4.10.]{AAH19}, $\dim\BN(V)>p^3$, a contradiction.

Assume that $\dim V>3$. Then  by Proposition \ref{pro:indecomposable-object-cyclic-p-group}, there must be a three-dimensional subobject $W$ of $V$, which is isomorphic to $M_{i,3}$,  $M_{i,2}\oplus M_{j,1}$ or $M_{i,1}\oplus M_{j,1}\oplus M_{k,1}$   for $i,j,k\in\I_{0,p-1}$. Hence from the preceding discussions and Example \ref{ex-1},   we have $\dim\BN(V)>\dim\BN(W)\geq p^3$.

Consequently, $\dim V=3$; if $V$ is not of diagonal type, then $p>2$ and $V\cong M_{i,2}\oplus M_{0,1}$ for $i\in\I_{1,p-1}$.  Clearly, $c^2=\id$ if and only if $j=0$. Hence by \cite[Theorem 2.2]{G}, $\BN(V)\cong\BN(M_{i,2})\otimes\BN(M_{0,1})$.
\end{proof}
\begin{rmk}
If $p=2$, then by Proposition \ref{pro:indecomposable-object-cyclic-p-group}, the objects of dimension greater than 2 in ${}_{\Z_2}^{\Z_2}\mathcal{YD}$ must be decomposable in ${}_{\Z_2}^{\Z_2}\mathcal{YD}$.
\end{rmk}

\begin{lem}\label{lem:p4-x1y1z1}
Let $p$ be a prime number and $\Char\K=p$. Let $H$ be a pointed Hopf algebra over $\K$ of dimension $p^4$. Assume that $\gr H=\K[g,x,y,z]/(g^p-1,x^p,y^p,z^p)$ with $g\in\G(H)$ and $x,y,z\in\Pp_{1,g}(H)$. Then the defining relations of $H$ are
\begin{gather*}
g^p=1,\quad gx-xg=\lambda_1g(1-g),\quad  gy-yg=\lambda_2g(1-g),\quad gz-zg=\lambda_3g(1-g),\\
x^p-\lambda_1x=0,\quad y^p-\lambda_2y=0,\quad z^p-\lambda_3z=0,\quad
xy-yx-\lambda_2x+\lambda_1y=\lambda_4(1-g^2),\\ xz-zx-\lambda_3x+\lambda_1z=\lambda_5(1-g^2),\quad  yz-zy-\lambda_3y+\lambda_2z=\lambda_6(1-g^2).
\end{gather*}
for some $\lambda_1,\lambda_2,\lambda_3\in\I_{0,1},~\lambda_4,\lambda_5,\lambda_6\in\K$ satisfying the conditions
\begin{align*}
\lambda_2\lambda_5=\lambda_3\lambda_4+\lambda_1\lambda_6.
\end{align*}
\end{lem}
\begin{proof}

It follows by a direct computation that
\begin{align*}
\Delta(gx-xg)=(gx-xg)\otimes g+g^2\otimes (gx-xg)\Rightarrow gx-xg\in\Pp_{g,g^2}(H)\cap H_0.
\end{align*}
Hence  $gx-xg=\lambda_1g(1-g)$ for some $\lambda_1\in\K$. By rescaling $x$, we can take $\lambda_1\in\I_{0,1}$.  Then by Proposition \ref{proJ} and Lemma \ref{pqlem1},
\begin{align*}
\Delta(x^p)=(x\otimes 1+g\otimes x)^p=x^p\otimes 1+1\otimes x^p +\lambda_1(g-1)\otimes x,
\end{align*}
which implies that $x^p-\lambda_1x\in\Pp(H)$. Since $\Pp(H)=0$, it follows that $x^p-\lambda_1x=0$ in $H$. Similarly, we have
\begin{align*}
gy-yg=\lambda_2g(1-g),\quad y^p-\lambda_2y=0,\quad\lambda_2\in\I_{0,1};\\
gz-zg=\lambda_3g(1-g),\quad z^p-\lambda_3z=0,\quad\lambda_3\in\I_{0,1}.
\end{align*}
Then a direct computation shows that
\begin{align*}
\Delta(xy-yx)=(xy-yx)\otimes 1+\lambda_2g(1-g)\otimes x-\lambda_1g(1-g)\otimes y+g^2\otimes (xy-yx),
\end{align*}
which implies that  $xy-yx-\lambda_2x+\lambda_1y\in\Pp_{1,g^2}(H)$. Since $\Pp_{1,g^2}(H)=\K\{1-g^2\}$, it follows that $xy-yx-\lambda_2x+\lambda_1y=\lambda_4(1-g^2)$ for some $\lambda_4\in\K$. Similarly, we have
\begin{align*}
xz-zx-\lambda_3x+\lambda_1z=\lambda_5(1-g^2),\quad yz-zy-\lambda_3y+\lambda_2z=\lambda_6(1-g^2),
\end{align*}
 for some  $\lambda_5,\lambda_6\in\K$.

 Applying the Diamond Lemma \cite{B} to show that $\dim H=p^4$, it suffices to show that the following ambiguities
 \begin{gather*}
a^pb=a^{p-1}(ab), \quad a(b^p)=(ab)b^{p-1},\quad b<a,~\text{and}~a,b\in\{g,x,y,z\},\\
(ab)c=a(bc),\quad c<b<a~\text{and}~a,b,c\in\{g,x,y,z\},
 \end{gather*}
are resolvable.

By Lemma \ref{pqlem1}, $[g, x^p]=(g)(\AdR x)^p=\lambda_1^{p-1}[g,x]$ and $[g^p, x]=pg^{p-1}[g,x]=0$. Then a direct computation shows that the ambiguities $(g^{p})x=g^{p-1}(gx)$ and $g (x^p)=(gx)x^{p-1}$ are resolvable. Similarly, $(g^{p})a=g^{p-1}(ga)$ and $g (a^p)=(ga)a^{p-1}$ are resolvable for $a\in\{y,z\}$.

By Lemma \ref{pqlem2}, $[x,y^p]=(x)(\AdR y)^p=\lambda_2^{p-1}[x,y]$ and $[x^p,y]=(\AdL x)^p(y)=(-\lambda_1)^{p-1}[x,y]$. Then a direct computation shows that the ambiguity $(x^{p})y=x^{p-1}(xy)$ and $g (x^p)=(gx)x^{p-1}$ are resolvable. Similarly, $a^pb=a^{p-1}(ab)$ and $a(b^p)=(ab)b^{p-1}$, for $b<a$, $a,b\in\{x,y,z\}$.

Now we claim that the ambiguity $g(xy)=(gx)y$ is resolvable. Indeed
\begin{align*}
g(xy)&=g(yx+\lambda_2x-\lambda_1y+\lambda_4(1-g^2))=(gy)x+\lambda_2gx-\lambda_1gy+\lambda_4g(1-g^2)\\
&=y(gx)+\lambda_2gx+\lambda_2xg-\lambda_2g^2x+\lambda_4g(1-g^2)-\lambda_1yg\\
&=yxg-\lambda_1yg^2+2\lambda_2xg+\lambda_1\lambda_2(g-g^2)-\lambda_2xg^2-2\lambda_1\lambda_2g^2(1-g)+\lambda_4g(1-g^2)\\
&=xyg+\lambda_2x(g-g^2)+\lambda_1yg+\lambda_1\lambda_2(g-g^2)-\lambda_1yg^2-\lambda_1\lambda_2g^2(1-g)\\
&=x(gy)+\lambda_1gy-\lambda_1g^2y=(gx)y.
\end{align*}
Similarly, $g(xz)=(gx)z$ and $g(yz)=(gy)z$ are resolvable.

We claim that the ambiguity $x(yz)=(xy)z$ imposes   $\lambda_2\lambda_5=\lambda_3\lambda_4+\lambda_1\lambda_6$. Indeed,
\begin{align*}
(xy)z&=[yx+\lambda_2x-\lambda_1y+\lambda_4(1-g^2)]z=y(xz)+\lambda_2xz-\lambda_1yz+\lambda_4z-\lambda_4g^2z\\
&=(yz)x+\lambda_3yx-2\lambda_1yz+\lambda_2xz+\lambda_5y(1-g^2)+\lambda_4z(1-g^2)-2\lambda_3\lambda_4g^2(1-g)\\
&=zyx+2\lambda_3yx+\lambda_2[x,z]-2\lambda_1yz+\lambda_5y(1-g^2)+\lambda_6x(1-g^2)+\lambda_4z(1-g^2)\\&\quad
   -2\lambda_1\lambda_6g^2(1-g)-2\lambda_3\lambda_4g^2(1-g)\\
&=zyx+2\lambda_3yx+\lambda_2\lambda_3x+\lambda_1\lambda_2z-2\lambda_1zy-2\lambda_1\lambda_3y+\lambda_5y(1-g^2)\\&\quad
+\lambda_6x(1-g^2)+\lambda_4z(1-g^2)+\lambda_2\lambda_5(1-g^2)-2\lambda_1\lambda_6(1-g^2)\\&\quad
-2\lambda_1\lambda_6g^2(1-g)-2\lambda_3\lambda_4g^2(1-g);
\end{align*}
\begin{align*}
x(yz)&=x(zy+\lambda_3y-\lambda_2z+\lambda_6(1-g^2))=(xz)y+\lambda_3xy-\lambda_2xz+\lambda_6x(1-g^2)\\
&=z(xy)+2\lambda_3xy-\lambda_1zy+\lambda_5(1-g^2)y-\lambda_2xz+\lambda_6x(1-g^2)\\
&=zyx-\lambda_2[x,z]-2\lambda_1zy+2\lambda_3xy+\lambda_4z(1-g^2)+\lambda_6x(1-g^2)\\&\quad+\lambda_5y(1-g^2)-2\lambda_2\lambda_5g^2(1-g)\\
&=zyx+2\lambda_3yx+\lambda_2\lambda_3x+\lambda_1\lambda_2z-2\lambda_1zy-2\lambda_1\lambda_3y+\lambda_5y(1-g^2)\\&\quad
+\lambda_6x(1-g^2)+\lambda_4z(1-g^2)+2\lambda_3\lambda_4(1-g^2)
-2\lambda_2\lambda_5g^2(1-g)- \lambda_2\lambda_5 (1-g^2).
\end{align*}
\end{proof}
\begin{rmk}
The Hopf subalgebras of $H$   in Lemma \ref{lem:p4-x1y1z1} generated by $g,x,y$  appeared in \cite{NW} as examples of pointed Hopf algebras over $\K$ of dimension $p^3$.
\end{rmk}
\begin{thm}\label{thm:p4-x1y0z0}
Let $p$ be a prime number and $\Char\K=p$. Let $H$ be a  pointed Hopf algebra over $\K$ of dimension $p^4$. Assume that $\gr H=\K[g,x,y,z]/(g^p-1,x^p,y^p,z^p)$ with $g\in\G(H)$, $x\in\Pp_{1,g}$ and $y,z\in\Pp(H)$. Then $H$ is isomorphic to one of the following Hopf algebras:
\begin{description}
  \item[(1)] $H_1(\lambda):=\K\langle g,x,y,z\rangle/(g^p-1,[g,x],[g,y],[g,z],[x,y]-\lambda x,[x,z],[y,z]-z,x^p,y^p-y,z^p)$,
  \item[(2)] $\K\langle g,x,y,z\rangle/(g^p-1,[g,x],[g,y],[g,z],[x,y]-x,[x,z]-(1-g),[y,z]-z,x^p,y^p-y,z^p)$,
  \item[(3)] $\K\langle g,x\rangle/(g^p-1,gx-xg-g(1-g),x^p-x)\otimes \K\langle y,z\rangle/(y^p-y,z^p,[y,z]-z)$,

  \item[(4)] $\K[ g,x]/(g^p-1, x^p)\otimes \K[y,z]/(y^p-y,z^p-z)$,
  \item[(5)]  $H_2(\lambda):=\K[ g,x,y,z]/(g^p-1, x^p-y-\lambda z,y^p-y,z^p-z)$,
  \item[(6)]  $\K\langle g,x,y,z\rangle/(g^p-1,[g,x],[g,y],[g,z],[x,y]-x,[x,z],[y,z],x^p,y^p-y,z^p-z)$,
  \item[(7)]  $\K\langle g,x,y,z\rangle/(g^p-1,[g,x],[g,y],[g,z],[x,y]-x,[x,z],[y,z],x^p- z,y^p-y,z^p-z)$,
  \item[(8)]  $H_3(\lambda,\gamma):=\K\langle g,x,y,z\rangle/(g^p-1,[g,x]-g(1-g),[g,y],[g,z],[x,y],[x,z],[y,z],x^p-x-\lambda y-\gamma z,y^p-y,z^p-z)$,
  \item[(9)] $\K[g,x]/(g^p-1,x^p)\otimes\K[y,z]/(y^p-y,z^p)$,
  \item[(10)] $\K[g,x,y,z]/(g^p-1,x^p-z,y^p-y,z^p)$,
  \item[(11)] $\K[g,x,y,z]/(g^p-1,x^p-y,y^p-y,z^p)$,
  \item[(12)] $\K[g,x,y,z]/(g^p-1,x^p-y-z,y^p-y,z^p)$,
  \item[(13)] $\K\langle g,x,y,z\rangle/(g^p-1,[g,x],[g,y],[g,z],[x,y],[x,z]-(1-g),[y,z],x^p, y^p-y,z^p)$,
  \item[(14)] $\K\langle g,x,y,z\rangle/(g^p-1,[g,x],[g,y],[g,z],[x,y],[x,z]-(1-g),[y,z],x^p- y, y^p-y,z^p)$,
  \item[(15)] $\K\langle g,x,y,z\rangle/(g^p-1,[g,x],[g,y],[g,z],[x,y]-x,[x,z],[y,z],x^p, y^p-y,z^p)$,
  \item[(16)] $\K\langle g,x,y,z\rangle/(g^p-1,[g,x],[g,y],[g,z],[x,y]-x,[x,z],[y,z],x^p-z, y^p-y,z^p)$,
  \item[(17)]  $H_4(\lambda,i):=\K\langle g,x,y,z\rangle/(g^p-1,[g,x]-g(1-g),[g,y],[g,z],[x,y],[x,z],[y,z],x^p-x-\lambda y-iz,y^p-y,z^p)$, for $i\in\I_{0,1}$,
  \item[(18)]  $H_5(\lambda):=\K\langle g,x,y,z\rangle/(g^p-1,[g,x]-g(1-g),[g,y],[g,z],[x,y],[x,z]-(1-g),[y,z],x^p-x-\lambda y, y^p-y,z^p)$,

  \item[(19)] $\K[g,x]/(g^p-1,x^p)\otimes\K[y,z]/(y^p-z,z^p)$,
  \item[(20)] $\K[g,x,y,z]/(g^p-1,x^p-z,y^p-z,z^p)$,
  \item[(21)] $\K[g,x,y,z]/(g^p-1,x^p-y,y^p-z,z^p)$,
  \item[(22)] $\K\langle g,x,y,z\rangle/(g^p-1,[g,x],[g,y],[g,z],[x,y]-(1-g),[y,z],[x,z],x^p,y^p-z,z^p)$,
  \item[(23)] $\K\langle g,x,y,z\rangle/(g^p-1,[g,x],[g,y],[g,z],[x,y]-(1-g),[y,z],[x,z],x^p- z,y^p-z,z^p)$,
  \item[(24)] $\K\langle g,x\rangle/(g^p-1,gx-xg-g(1-g),x^p-x)\otimes\K[y,z]/(y^p-z,z^p)$,
  \item[(25)] $\K\langle g,x,y,z\rangle/(g^p-1,gx-xg-g(1-g),[g,y],[g,z],[x,y],[x,z],[y,z],x^p-x-z,y^p-z,z^p)$,
  \item[(26)] $\K\langle g,x,y,z\rangle/(g^p-1,gx-xg-g(1-g),[g,y],[g,z],[x,y],[x,z],[y,z],x^p-x-y,y^p-z,z^p)$,
  \item[(27)]  $H_6(\lambda):=\K\langle g,x,y,z\rangle/(g^p-1,gx-xg-g(1-g),[g,y],[g,z],[x,y]-(1-g),[x,z],[y,z],x^p-x-\lambda z,y^p-z,z^p)$,

  \item[(28)] $\K[g,x]/(g^p-1,x^p)\otimes\K[y,z]/(y^p,z^p)$,
  \item[(29)] $\K\langle g,x\rangle/(g^p-1,gx-xg-g(1-g),x^p-x)\otimes\K[y,z]/(y^p,z^p)$,
  \item[(30)] $\K[g,x,y,z]/(g^p-1,x^p-y, y^p,z^p)$,
  \item[(31)] $\K\langle g,x,y,z\rangle/(g^p-1,gx-xg=g(1-g),[g,y],[g,z],[x,y],[x,z],[y,z],x^p-x-y,y^p,z^p)$,
  \item[(32)] $\K\langle g,x,y,z\rangle/(g^p-1,[g,x],[g,y],[g,z],[x,y]-(1-g),[x,z],[y,z],x^p,y^p,z^p)$,
  \item[(33)] $\K\langle g,x,y,z\rangle/(g^p-1,[g,x],[g,y],[g,z],[x,y]-(1-g),[x,z],[y,z],x^p-z,y^p,z^p)$,
  \item[(34)] $\K\langle g,x,y,z\rangle/(g^p-1,gx-xg=g(1-g),[g,y],[g,z],[x,y]-(1-g),[x,z],[y,z],x^p-x,y^p,z^p)$,
  \item[(35)] $\K\langle g,x,y,z\rangle/(g^p-1,gx-xg=g(1-g),[g,y],[g,z],[x,y]-(1-g),[x,z],[y,z],x^p-x-z,y^p,z^p)$,
\end{description}
Furthermore, for $\lambda,\gamma\in\K$,
\begin{itemize}
\item $H_1(\lambda)\cong H_1(\gamma)$, if and only if, $\lambda=\gamma$;
\item $H_2(\lambda)\cong H_2(\gamma)$, if and only if, there exist $\alpha_1,\alpha_2,\beta_1,\beta_2\in\K$ satisfying $\alpha_i^p-\alpha_i=0=\beta_i^p-\beta_i$ for $i\in\I_{1,2}$ such that $(\alpha_1+\beta_1\lambda)\gamma=(\alpha_2+\beta_2\lambda)$ and $\alpha_1\beta_2-\alpha_2\beta_1\neq 0$;
\item $H_3(\lambda,\gamma)\cong H_3(\mu,\nu)$, if and only if, there exist $\alpha_i,\beta_i\in\K$  satisfying $\alpha_i^p-\alpha_i=0=\beta_i^p-\beta_i$ for $i\in\I_{0,1}$ such that $\alpha_1\beta_2-\alpha_2\beta_1\neq 0$ and $\lambda\alpha_1+\gamma\beta_1=\mu$,  $\lambda\alpha_2+\gamma\beta_2=\nu$;
\item $H_4(\lambda,i)\cong H_4(\gamma,j)$, if and only if, there is $\alpha\neq 0\in\K$ satsifying $\alpha^p=\alpha$ such that $\lambda\alpha=\gamma$ and $i=j$;
\item $H_5(\lambda)\cong H_5(\gamma)$, if and only if, there is $\alpha\neq 0\in\K$ satsifying $\alpha^p=\alpha$ such that $\lambda\alpha=\gamma$;
\item $H_6(\lambda)= H_6(\gamma)$, if and only if, $\lambda=\gamma$.
\end{itemize}
\end{thm}
\begin{proof}
Similar to the proof of Lemma \ref{lem:p4-x1y1z1}, we have
\begin{gather*}
gx-xg=\lambda_1g(1-g),\quad gy-yg=0,\quad gz-zg=0,\\
x^p-\lambda_1x\in\Pp(H),\quad y^p\in\Pp(H),\quad z^p\in\Pp(H),\\
xy-yx\in\Pp_{1,g}(H),\quad xz-zx\in\Pp_{1,g}(H),\quad yz-zy\in\Pp(H).
\end{gather*}
for some $\lambda_1\in\I_{0,1}$.
Since $\Pp(H)=\K\{y,z\}$ and $\Pp_{1,g}(H)=\K\{x,1-g\}$, it follows that
\begin{gather*}
x^p-\lambda_1x=\mu_1y+\mu_2z,\quad y^p=\mu_3y+\mu_4z,\quad z^p=\mu_5y+\mu_6z,\\
xy-yx=\nu_1x+\nu_2(1-g),\quad xz-zx=\nu_3x+\nu_4(1-g),\quad yz-zy=\nu_5y+\nu_6z,
\end{gather*}
for some $\mu_1,\cdots,\mu_6,\nu_1,\cdots,\nu_6\in\K$.

It follows by  Lemmas \ref{pqlem1}--\ref{pqlem2} that
\begin{align*}
 [g,x^p]&=(g)(\AdR x)^p=[g,x],\quad g^px=xg^p,\\
 [x^p,y]&=(\AdL x)^p(y)=-\nu_2(\AdL x)^{p-1}(g)=\nu_2\lambda_1(1-g),  \\
 [x^p,z]&=(\AdL x)^p(z)=-\nu_4(\AdL x)^{p-1}(g)=\nu_4\lambda_1(1-g),\\
 [x,y^p]&=(x)(\AdR y)^p=\nu_1(x)(\AdR y)^{p-1}=\nu_1^{p-1}[x,y]=\nu_1^px+\nu_1^{p-1}\nu_2(1-g),\\
 [x,z^p]&=(x)(\AdR z)^p=\nu_3(x)(\AdR z)^{p-1}=\nu_3^{p-1}[x,z]=\nu_3^px+\nu_3^{p-1}\nu_4(1-g),\\
 [y^p,z]&=(\AdL y)^p(z)=\nu_6(\AdL y)^{p-1}(z)=\nu_6^{p-1}[y,z]=\nu_6^{p-1}\nu_5y+\nu_6^pz,\\
 [y,z^p]&=(y)(\AdR z)^p=\nu_5(y)(\AdR z)^{p-1}=\nu_5^{p-1}[y,z]=\nu_5^py+\nu_5^{p-1}\nu_6z.
\end{align*}
Then the verification of $(a^p)b=a^{p-1}(ab)$ for $a,b\in\{g,x,y,z\}$ and $(gx)y=g(xy), g(xz)=(gx)z$ amounts to the conditions
\begin{gather}
\lambda_1\nu_1=\mu_2\nu_5=\mu_2\nu_6=0,\quad \lambda_1\nu_3=\mu_1\nu_5=\mu_1\nu_6=0, \label{eq_q4-1}\\
\mu_3\nu_1+\mu_4\nu_3=\nu_1^p,~ \mu_3\nu_2+\mu_4\nu_4=\nu_1^{p-1}\nu_2,~\mu_5\nu_1+\mu_6\nu_3=\nu_3^p,~ \mu_5\nu_2+\mu_6\nu_4=\nu_3^{p-1}\nu_4,\label{eq_q4-2} \\
\mu_3\nu_5=\nu_6^{p-1}\nu_5,\quad \mu_3\nu_6=\nu_6^{p},\quad \mu_6\nu_5=\nu_5^{p},\quad \mu_6\nu_6=\nu_5^{p-1}\nu_6, \label{eq_q4-3}\\
\mu_1\nu_1+\mu_2\nu_3=0=\mu_1\nu_2+\mu_2\nu_4,\quad \mu_4\nu_5=\mu_4\nu_6=\mu_5\nu_5=\mu_5\nu_6=0.\label{eq_q4-4}
\end{gather}
Finally, the verification of $(xy)z=x(yz)$ amounts to the conditions
\begin{gather}
\nu_1\nu_5=\nu_3\nu_6=0,\quad \nu_2\nu_3+\nu_2\nu_5+\nu_4\nu_6=\nu_1\nu_4.\label{eq_q4-5}
\end{gather}
By the Diamond lemma, $\dim H=p^4$.

Let $L$ be the subalgebra of $H$ generated by $y,z$. It is clear that $L$ is a Hopf subalgebra of $H$.  Indeed, $L\cong U(\Pp(H))$, where $U(\Pp(H))$ is the restricted universal enveloping algebra of $\Pp(H)$. By \cite[Proposition A.3]{W1}, $L$ is isomorphic to one of the following Hopf algebras
\begin{enumerate}
  \item[(a)] $\K\langle y,z\rangle/(y^p-y,z^p,[y,z]-z)$,
  \item[(b)] $\K[y,z]/(y^p-y,z^p-z)$,
  \item[(c)] $\K[y,z]/(y^p-y,z^p)$,
  \item[(d)] $\K[y,z]/(y^p-z,z^p)$,
  \item[(e)] $\K[y,z]/(y^p,z^p)$.
\end{enumerate}

\textbf{Case (a).} Assume that $L$ is isomorphic to the Hopf algebra described in $(a)$, that is, there is an isomorphism of Hopf algebras $\phi$ from $L$ to the Hopf algebra described in $(a)$. Then  there are some $\beta_1,\beta_2,\gamma_1,\gamma_2\in\K$ such that
\begin{align*}
\phi(y)=\beta_1y^{\prime}+\beta_2z^{\prime},\quad \phi(z)=\gamma_1y^{\prime}+\gamma_2z^{\prime},\quad\beta_1\gamma_2-\beta_2\gamma_1\neq 0.
\end{align*}
Applying $\phi$ to the relation $yz-zy=\nu_5y+\nu_6z$ in $L$, we have
\begin{align*}
\left(\begin{array}{cc}
   \nu_5 & \nu_6
 \end{array}\right)
\left( \begin{array}{cc}
   \beta_1 & \beta_2\\
   \gamma_1& \gamma_2
 \end{array}\right)
 \neq
 \left( \begin{array}{cc}
   0 & 0
 \end{array}\right)
 \Rightarrow (\nu_5,\nu_6)\neq (0,0).
\end{align*}

Then from Equations \eqref{eq_q4-1} and \eqref{eq_q4-4}, we have $\mu_1=\mu_2=0=\mu_4=\mu_5$.

If $\nu_5\nu_6=0$, then by swapping $y$ and $z$, we may assume that $\nu_5= 0$ and $\nu_6\neq 0$. Then Equations \eqref{eq_q4-3} and \eqref{eq_q4-5} yield $\mu_3=\nu_6^{p-1}\neq 0$, $\mu_6=0=\nu_3$ and $\nu_4\nu_6=\nu_1\nu_4$. Therefore, we can take $\mu_3=1=\nu_6$ via the linear translation $y\mapsto\nu_6^{-1}y$.

If $\nu_5\nu_6\neq 0$, then from Equations \eqref{eq_q4-3} and \eqref{eq_q4-5}, we have $\mu_3=\nu_6^{p-1}\neq0$, $\mu_6=\nu_5^{p-1}\neq 0$, $\nu_1=\nu_3=0$ and $\nu_2\nu_5+\nu_4\nu_6=0$. Furthermore, Equations \eqref{eq_q4-2} yield $\nu_2=0=\nu_4$. Therefore, we can take $\mu_3=1=\nu_6$ and $\mu_6=0=\nu_5$ via the linear translation $y\mapsto \nu_6^{-1}y, z\mapsto \nu_6^{-1}y+\nu_5^{-1}z$.

From the above discussion, without loss of generality, we can take $\mu_3-1=0=\mu_4=\mu_5=\mu_6$ and $\nu_5=0=\nu_6-1$. Then from Equations \eqref{eq_q4-1}--\eqref{eq_q4-5},  $\mu_1=0=\mu_2$, $\lambda_1\nu_1=0=\nu_3$, $\nu_1^p=\nu_1$, $\nu_4=\nu_1\nu_4$, $\nu_2=\nu_1^{p-1}\nu_2$ and we can take $\nu_4\in\I_{0,1}$ by rescaling $z$.

If $\lambda_1=0=\nu_4$,   then we can take $\nu_2=0$. Indeed, if  $\nu_1=0$, then $\nu_2=0$, otherwise we can take $\nu_2=0$ via the linear translation $x\mapsto x+a(1-g)$ satisfying $\nu_1a=\nu_2$. Hence $H\cong H_1(\nu_1)$   described in $(1)$.
If $\lambda_1=0=\nu_4-1$, then $\nu_1=1$ and we can take $\nu_2=0$ via the linear translation $x\mapsto x+\nu_2(1-g)$,  which gives one class of $H$ described in $(2)$.

If $\lambda_1=1$, then $\nu_1=0=\nu_2=\nu_4$, which gives one class of $H$   described in $(3)$.

\textbf{Claim:} $H_1(\lambda)\cong H_1(\gamma)$ for $\lambda,\gamma\in\K$, if and only if, $\lambda=\gamma$.

Write $g^{\prime},x^{\prime},y^{\prime},z^{\prime}$ to distinguish the generators of $H_{1}(\gamma)$. Observe that $\Pp_{1, g^{\prime} }(H_1(\gamma))=\K\{x^{\prime}\}\oplus \K\{1- g^{\prime} \}$ and $\Pp(H_1(\gamma))=\K\{y^{\prime},z^{\prime}\}$. Suppose that $\phi: H_{1}(\lambda)\rightarrow H_{1}(\gamma)$ for $\lambda,\gamma\in\K$ is a Hopf algebra isomorphism.    Then by Proposition \ref{pro:R11-4.3.3}, $\phi(\Pp(H_{1}(\lambda)))=\Pp(H_{1}(\lambda))$, $\phi(\Pp_{1,g}(H_{1}(\lambda)))=\Pp_{1,g^{\prime}}(H_{1}(\lambda))$. Therefore,
\begin{align}\label{eq:iso-1-0-0}
 \phi(g)=g^{\prime}, \quad \phi(x)=\alpha_1x^{\prime}+\alpha_2(1-g^{\prime}),\quad\phi(y)=\beta_1y^{\prime}+\beta_2z^{\prime},\quad \phi(z)=\gamma_1y^{\prime}+\gamma_2z^{\prime},
 \end{align}
 for some $\alpha_i,\beta_i,\gamma_i\in\K$ and $i\in\I_{1,2}$. Applying $\phi$ to the relation $[y,z]-z=0$,   we have
\begin{align*}
\gamma_1=0,\quad (\beta_1-1)\gamma_2=0\quad \Rightarrow \quad \beta_1=1.
\end{align*}
Then applying $\phi$ to the relation $[x,y]-\lambda x=0$, we have
\begin{align*}
\lambda=\gamma.
\end{align*}
 Conversely, it is easy to see that   $ H_{1}(\lambda)\cong H_{1}(\gamma)$ if  $\lambda=\gamma$. Consequently, the claim follows.

 Similarly, we can also show that the Hopf algebras described in $(1)$--$(3)$ are pairwise isomorphic.  Indeed, direct computations show that there are no elements $\alpha_i,\beta_i,\gamma_i\in\K$ for $i\in\I_{1,2}$ such that the morphism \eqref{eq:iso-1-0-0} is an isomorphism.

\textbf{Case (b).} Assume that $L$ is isomorphic to the Hopf algebra described in $(b)$, that is, there is an isomorphism of Hopf algebras $\phi$ from $L$ to the Hopf algebra described in $(b)$. Then  there are some $\beta_1,\beta_2,\gamma_1,\gamma_2\in\K$ such that
\begin{align*}
\phi(y)=\beta_1y^{\prime}+\beta_2z^{\prime},\quad \phi(z)=\gamma_1y^{\prime}+\gamma_2z^{\prime},\quad\beta_1\gamma_2-\beta_2\gamma_1\neq 0.
\end{align*}
Applying $\phi$ to the relation $yz-zy=\nu_5y+\nu_6z$ in $L$, we have
\begin{align*}
\left(\begin{array}{cc}
   \nu_5 & \nu_6
 \end{array}\right)
\left( \begin{array}{cc}
   \beta_1 & \beta_2\\
   \gamma_1& \gamma_2
 \end{array}\right)
 = 0
 \Rightarrow \nu_5=0=\nu_6.
\end{align*}
Applying $\phi$ to the relations $y^p=\mu_3y+\mu_4z$ and $y^p=\mu_5y+\mu_6z$, we have
\begin{align*}
\left(\begin{array}{cc}
 \mu_3 & \mu_4 \\
   \mu_5 & \mu_6
\end{array}\right)
\left(\begin{array}{cc}
  \beta_1 & \beta_2\\
  \gamma_1& \gamma_2
\end{array}\right)
=
 \left( \begin{array}{cc}
   \beta_1^p & \beta_2^p\\
   \gamma_1^p& \gamma_2^p
 \end{array}\right).
\end{align*}
Since $\beta_1\gamma_2-\beta_2\gamma_1\neq 0$,  by Proposition \ref{proJ}, $\beta_1^p\gamma_2^p-\beta_2^p\gamma_1^p\neq 0$ and hence $\mu_3\mu_6-\mu_4\mu_5\neq0$.

Denoted by $H(\lambda_1,\nu_1,\cdots,\nu_4,\mu_1,\cdots,\mu_6)$ the Hopf algebra $H$ in this case for $\lambda_1\in\I_{0,1}$, $\mu_1,\cdots,\mu_6$, $\nu_1,\cdots,\nu_4\in\K$ satisfying the relations \eqref{eq_q4-1}--\eqref{eq_q4-5}. Then there is an isomorphism of Hopf algebras:
\begin{align*}
\psi: H(\lambda_1,\nu_1,\cdots,\nu_4,\mu_1,\cdots,\mu_6)\mapsto H(\lambda_1,\nu_1^{\prime},\cdots,\nu_4^{\prime},\mu_1^{\prime},\mu_2^{\prime},1,0,0,1),
\end{align*}
given by
\begin{align*}
\psi(g)=g^{\prime},\quad\psi(x)=x^{\prime},\quad \psi(y)=\beta_1y^{\prime}+\beta_2z^{\prime},\quad \psi(z)=\gamma_1y^{\prime}+\gamma_2z^{\prime},
\end{align*}
where
 \begin{align*}
 \left(\begin{array}{cc}
   \mu_1 & \mu_2
 \end{array}\right)
\left( \begin{array}{cc}
   \beta_1 & \beta_2\\
   \gamma_1& \gamma_2
 \end{array}\right)=
 \left(\begin{array}{cc}
   \mu_1^{\prime} & \mu_2^{\prime}
 \end{array}\right),\\
 \left(\begin{array}{cc}
 \beta_1 & \beta_2\\
  \gamma_1& \gamma_2
\end{array}\right)
\left(\begin{array}{cc}
  \nu_1^{\prime} & \nu_2^{\prime}\\
   \nu_3^{\prime}& \nu_4^{\prime}
\end{array}\right)
=
 \left( \begin{array}{cc}
   \nu_1 & \nu_2\\
   \nu_3& \nu_4
 \end{array}\right).
 \end{align*}

From the above discussion, without loss of generality, we can assume that $\mu_3=1$, $\mu_4=\mu_5=0$, $\mu_6=1$ and $\nu_5=0=\nu_6$. Then $\lambda_1\nu_1=0=\lambda_1\nu_3$,   $\nu_1^p=\nu_1$, $\nu_3^p=\nu_3$, $\nu_4=\nu_3^{p-1}\nu_4$,  $\nu_2=\nu_1^{p-1}\nu_2$, $\mu_1\nu_1+\mu_2\nu_3=0=\mu_1\nu_2+\mu_2\nu_4$ and $\nu_2\nu_3=\nu_1\nu_4$. Hence we can take $\nu_1,\nu_3\in\{0,1\}$ by rescaling $y,z$.

If $\lambda_1=0$ and $\nu_1=0=\nu_3$, then $\nu_2=0=\nu_4$ and we can take $\mu_1\in\I_{0,1}$ or $\mu_2\in\I_{0,1}$ by rescaling $x$. If $\mu_1=0=\mu_2$, then $H$ is isomorphic to the Hopf algebra described in $(4)$.
If $\mu_1=1$, then $H\cong H_2(\mu_2)$ described in $(5)$.
If $\mu_1=0$ and $\mu_2\neq 0$, then by rescaling $x$, we have $\mu_2=1$, and hence by swapping $x$ and $y$, $H\cong H_2(0)$.

If $\lambda_1=0$ and $\nu_1-1=0=\nu_3$, then $\nu_4=0=\mu_1$, and we can take $\nu_2=0$ via the linear translation $x\mapsto x+\nu_2(1-g)$. Moreover, we can take $\mu_2\in\I_{0,1}$ by rescaling $x$, which gives two classes of $H$ described in $(6)$--$(7)$.

If $\lambda_1=0$ and $\nu_1=0=\nu_3-1$, then it can be reduced to the case $\lambda_1=0$ and $\nu_1-1=0=\nu_3$ by swapping $x$ and $y$.

If $\lambda_1=0$ and $\nu_1=1=\nu_3$, then $\mu_1+\mu_2=0=\nu_2-\nu_4$ and hence we can take $\nu_2=0=\nu_4$ via the linear translations $x\mapsto x+\nu_2(1-g)$. Therefore, it can be reduced to the case $\lambda_1=0$ and $\nu_1-1=0=\nu_3$  via the linear translation $z\mapsto z-y$.

If $\lambda_1=1$, then $\nu_1=0=\nu_3$ and hence $\nu_2=0=\nu_4$. Therefore  $H\cong H_3(\mu_1,\mu_2)$  described in $(8)$.

Similar to the proof of Case $(a)$,   $H_2(\lambda)\cong H_2(\gamma)$ for $\lambda,\gamma\in\K$, if and only if, there exist $\alpha_1,\alpha_2,\beta_1,\beta_2\in\K$ satisfying $\alpha_i^p-\alpha_i=0=\beta_i^p-\beta_i$ for $i\in\I_{1,2}$ such that $(\alpha_1+\beta_1\lambda)\gamma=(\alpha_2+\beta_2\lambda)$ and $\alpha_1\beta_2-\alpha_2\beta_1\neq 0$.
$H_3(\lambda,\gamma)\cong H_3(\mu,\nu)$ if and only if, there exist $\alpha_i,\beta_i\in\K$ satisfying $\alpha_i^p-\alpha_i=0=\beta_i^p-\beta_i$ for $i\in\I_{0,1}$ such that $\alpha_1\beta_2-\alpha_2\beta_1\neq 0$ and $\lambda\alpha_1+\gamma\beta_1=\mu$,  $\lambda\alpha_2+\gamma\beta_2=\nu$. The Hopf algebras from the different items are pairwise non-isomorphic.

\textbf{Case (c).} Assume that $L$ is isomorphic to the Hopf algebra described in $(c)$. Then similar to the proof of Case (b), we have $\nu_5=0=\nu_6$. Denoted by $H(\lambda_1,\nu_1,\cdots,\nu_4,\mu_1,\cdots,\mu_6)$ the Hopf algebra $H$ in this case. Then there is an isomorphism of Hopf algebras:
\begin{align*}
\psi: H(\lambda_1,\nu_1,\cdots,\nu_4,\mu_1,\cdots,\mu_6)\mapsto H(\lambda_1,\nu_1^{\prime},\cdots,\nu_4^{\prime},\mu_1^{\prime},\mu_2^{\prime},1,0,0,0),
\end{align*}
given by
\begin{align*}
\psi(g)=g^{\prime},\quad\psi(x)=x^{\prime},\quad \psi(y)=\beta_1y^{\prime}+\beta_2z^{\prime},\quad \psi(z)=\gamma_1y^{\prime}+\gamma_2z^{\prime},
\end{align*}
where $\beta_1\gamma_2-\beta_2\gamma_1\neq0$,
 \begin{gather*}
 \left(\begin{array}{cc}
   \mu_1 & \mu_2
 \end{array}\right)
\left( \begin{array}{cc}
   \beta_1 & \beta_2\\
   \gamma_1& \gamma_2
 \end{array}\right)=
 \left(\begin{array}{cc}
   \mu_1^{\prime} & \mu_2^{\prime}
 \end{array}\right),\\
 \left(\begin{array}{cc}
 \beta_1 & \beta_2\\
  \gamma_1& \gamma_2
\end{array}\right)
\left(\begin{array}{cc}
  \nu_1^{\prime} & \nu_2^{\prime}\\
   \nu_3^{\prime}& \nu_4^{\prime}
\end{array}\right)
=
 \left( \begin{array}{cc}
   \nu_1 & \nu_2\\
   \nu_3& \nu_4
 \end{array}\right),\\
 \left(\begin{array}{cc}
 \mu_3 & \mu_4 \\
   \mu_5 & \mu_6
\end{array}\right)
\left(\begin{array}{cc}
  \beta_1 & \beta_2\\
  \gamma_1& \gamma_2
\end{array}\right)
=
 \left( \begin{array}{cc}
   \beta_1^p & \beta_2^p\\
   \gamma_1^p& \gamma_2^p
 \end{array}\right)
 \left(\begin{array}{cc}
 1 & 0 \\
   0 & 0
\end{array}\right).
 \end{gather*}

Therefore,  we can assume that $\mu_3-1=0=\mu_4=\mu_5=\mu_6=\nu_5=\nu_6$. Then $\lambda_1\nu_1=0=\nu_3$, $\nu_1=\nu_1^p$, $\nu_2=\nu_1^{p-1}\nu_2$, $\mu_1\nu_2+\mu_2\nu_4=0=\mu_1\nu_1=\nu_1\nu_4$ and we can take $\nu_1\in\I_{0,1}$ by rescaling $y$.

If $\lambda_1=0=\nu_1$, then $\nu_2=0=\mu_2\nu_4$ and we can take $\nu_4\in\I_{0,1}$ by rescaling $z$. If $\nu_4=0$, then we can take $\mu_1,\mu_2\in\I_{0,1}$ by rescaling $x,z$, which gives four classes of $H$ described in $(9)$--$(12)$.
If $\nu_4=1$, then $\mu_2=0$ and we can take $\mu_1\in\I_{0,1}$ by rescaling $x,z$. Indeed, if $\mu_1\neq 0$, then we can take $\mu_1=1$ via $x\mapsto ax,z\mapsto a^{-1}z$ satisfying $a^p=\mu_1$. Therefore $H$ is isomorphic to one of the Hopf algebras in $(13)$--$(14)$.

If $\lambda_1=0=\nu_1-1$, then $\mu_1=0=\nu_4$ and we can take $\nu_2=0$ via the linear translation $x\mapsto x+\nu_2(1-g)$. Hence we can take $\mu_2\in\I_{0,1}$ by rescaling $x$, which gives two classes of $H$ described in $(15)$--$(16)$.

If $\lambda_1=1$, then $\nu_1=0=\nu_2=\mu_2\nu_4$ and we can take $\nu_4\in\I_{0,1}$ by rescaling $z$. If $\nu_4=0$, then we can take $\mu_2\in\I_{0,1}$ by rescaling $z$ and hence $H\cong H_4(\mu_1,\mu_2)$   described in $(17)$.
If $\nu_4=1$, then $\mu_2=0$ and hence $H\cong H_5(\mu_1)$ described in $(18)$.

Similar to the proof of Case $(a)$,   $H_4(\lambda,i)\cong H_4(\gamma,j)$ if and only if there is $\alpha\neq 0\in\K$ satsifying $\alpha^p=\alpha$ such that $\lambda\alpha=\gamma$ and $i=j$. $H_5(\lambda)\cong H_5(\gamma)$ if and only if there is $\alpha\neq 0\in\K$ satsifying $\alpha^p=\alpha$ such that $\lambda\alpha=\gamma$. The Hopf algebras from different items are pairwise non-isomorphic.

\textbf{Case (d).} Assume that $L$ is isomorphic to the Hopf algebra described in $(d)$. Similar to the proof of Case (b), without loss of generality, we can assume that $\mu_3=0=\mu_4-1=\mu_5=\mu_6=\nu_5=\nu_6$. Then $\nu_1=\nu_3=\nu_4=\mu_1\nu_2=0$.

If $\lambda_1=0$, then $\nu_2\in\I_{0,1}$ by rescaling $x$. If $\nu_2=0$, then we can take $\mu_1\in\I_{0,1}$ by rescaling $x$. If $\mu_1=0$,  then we can take $\mu_2\in\I_{0,1}$. If $\mu_1=1$, then we can take $\mu_2=0$ via the linear translation $y\mapsto y+\mu_2z$. Therefore, we obtain three classes of $H$ described in $(19)$--$(21)$.
If $\nu_2=1$, then $\mu_1=0$ and we can take $\mu_2\in\I_{0,1}$, which gives two classes of $H$ described in $(22)$--$(23)$. Indeed, if $\mu_2\neq 0$, then we can take $\mu_2=1$ via $x\mapsto ax,y\mapsto a^{-1}y,z\mapsto a^{-p}z$ satisfying $a^{-2p}=\mu_2$.

If $\lambda_1-1=0=\nu_2$, then we can take $\mu_1\in\I_{0,1}$ by rescaling $y,z$. If $\mu_1=0$, then we can take $\mu_2\in\I_{0,1}$ by rescaling $y,z$. If $\mu_1=1$, then we can take $\mu_2=0$ via the linear translation $y\mapsto y+\mu_2z$. Therefore, we obtain three classes of $H$ described in $(24)$--$(26)$.

If $\lambda_1=1$ and $\nu_2\neq 0$, then $\mu_1=0$ and we can take $\nu_2=1$ by rescaling $y,z$. Therefore, $H\cong H_7(\mu_2)$ described in $(27)$.

Similar to the proof of Case $(a)$, $H_7(\lambda)= H_7(\gamma)$, if and only if, $\lambda=\gamma$. The Hopf algebras from different items are pairwise non-isomorphic.

\textbf{Case (e).} Assume that $L$ is isomorphic to the Hopf algebra described in $(e)$. Then there is an isomorphism of Hopf algebras $\phi$ from $L$ to the Hopf algebra described in $(e)$ given by
\begin{align*}
\phi(y)=\beta_1y^{\prime}+\beta_2z^{\prime},\quad \phi(z)=\gamma_1y^{\prime}+\gamma_2z^{\prime},\quad\beta_1\gamma_2-\beta_2\gamma_1\neq 0.
\end{align*}
Applying $\phi$ to the relation $yz-zy=\nu_5y+\nu_6z$ in $L$, we have
\begin{align*}
\left(\begin{array}{cc}
   \nu_5 & \nu_6
 \end{array}\right)
\left( \begin{array}{cc}
   \beta_1 & \beta_2\\
   \gamma_1& \gamma_2
 \end{array}\right)
 =0
 \Rightarrow \nu_5=0=\nu_6.
\end{align*}
From the relations $y^p=\mu_3y+\mu_4z$ and $z^p=\mu_5y+\mu_6z$, we have
\begin{align*}
\left(\begin{array}{cc}
   \mu_3 & \mu_4\\
   \mu_5 & \mu_6
 \end{array}\right)
\left( \begin{array}{cc}
    \beta_1 & \beta_2\\
   \gamma_1& \gamma_2
 \end{array}\right)
 =
\left( \begin{array}{cc}
    0 & 0\\
   0& 0
 \end{array}\right)
\end{align*}
Therefore, we have $\mu_3=\mu_4=\mu_5=\mu_6=\nu_5=\nu_6=0$. Then
$\nu_1=0=\nu_3$, $\mu_1\nu_2+\mu_2\nu_4=0$ and we can take $\nu_2,\nu_4\in\I_{0,1}$ by rescaling $y,z$.

If $\nu_2=0=\nu_4$ and $\mu_1=0=\mu_2$, then $H$ is isomorphic to one of the Hopf algebras described in $(28)$--$(29)$.

If $\nu_2=0=\nu_4$ and $\mu_1\neq 0$ or $\mu_2\neq 0$, then $H$ is isomorphic to one of the Hopf algebras described in $(30)$--$(31)$.
Indeed, if $\mu_1\neq 0$, then we can take $\mu_1=1$ and $\mu_2=0$ via the linear translation $y\mapsto \mu_1y+\mu_2z$, $z\mapsto z$; if $\mu_2\neq 0$, then we can take $\mu_1=1$ and $\mu_2=0$ via the linear translation $y\mapsto \mu_1y+\mu_2z$, $z\mapsto y$;

If $\nu_2-1=0=\nu_4$, then $\mu_1=0$ and $\mu_2\in\I_{0,1}$ by rescaling $z$, which gives four classes of $H$ described in $(32)$--$(35)$.

If $\nu_2=0=\nu_4-1$, then it can be reduced to the case $\nu_2-1=0=\nu_4$ by swapping $y$ and $z$.

If $\nu_2=1=\nu_4$, then $\mu_1+\mu_2=0$ and hence it can be reduced to the case $\nu_2-1=0=\nu_4$ via the linear translation $z\mapsto z-y$.

Similar to the proof of Case (a), the Hopf algebras from different items are pairwise non-isomorphic.
\end{proof}

\begin{rmk}
The Hopf subalgebras of $H$ in Theorem \ref{thm:p4-x1y0z0} generated by $g, x, y$ or by $g, y, z$
appeared in \cite{NW} as examples of pointed Hopf algebras over k of dimension $p^3$.

\end{rmk}

\begin{rmk}
In Theorem \ref{thm:p4-x1y0z0}, there are six infinite families of Hopf algebras of dimension $p^4$, which constitute new examples of Hopf algebras. Moreover, the Hopf algebras described in $(1)$--$(2)$, $(6)$--$(8)$, $(13)$--$(18)$, $(22)$--$(23)$, $(25)$--$(27)$, $(31)$--$(35)$ are not tensor product Hopf algebras and constitute new examples of non-commutative and non-cocommutative pointed Hopf algebras. In particular, up to isomorphism,  there are infinitely many Hopf algebras of dimension $p^4$ that are generated by group-like elements and skew-primitive elements.
\end{rmk}

\begin{lem}\label{lem:p4-x1y1z0}
Let $p$ be a prime number and $\Char\K=p$. Let $H$ be a  pointed Hopf algebra over $\K$ of dimension $p^4$. Assume that $\gr H=\K[g,x,y,z]/(g^p-1,x^p,y^p,z^p)$ with $g\in\G(H)$, $x,y\in\Pp_{1,g}(H)$ and $z\in\Pp(H)$. Then the defining relations of $H$ have the following form
\begin{gather*}
g^p=1,\quad gx-xg=\lambda_1g(1-g),\quad gy-yg=\lambda_2g(1-g),\quad gz-zg=0,\\
x^p-\lambda_1x=\lambda_3z,\quad y^p-\lambda_2y=\lambda_4z,\quad z^p=\lambda_5z,\\
xz-zx=\gamma_1x+\gamma_2y+\gamma_3(1-g),\quad yz-zy=\gamma_4x+\gamma_5y+\gamma_6(1-g),\\
xy-yx-\lambda_2x+\lambda_1y=\left\{
  \begin{array}{ll}
    \lambda_6z, & p=2, \\
    \lambda_7(1-g^2), & p>2.
  \end{array}
\right.
\end{gather*}
for some $\lambda_1,\cdots,\lambda_7, \gamma_1,\cdots,\gamma_6\in\K$.

Suppose that $p=2$. Then $dim H=16$ if and only if the parameters satisfy the following conditions:
\begin{gather}
\lambda_6\gamma_1=\lambda_3\gamma_4, \quad\lambda_6\gamma_2=\lambda_3\gamma_5,\quad \lambda_6\gamma_3=\lambda_3\gamma_6,\label{eq:x1y1z0-1}\\
\lambda_1\gamma_1=\lambda_2\gamma_2,\quad \lambda_6\gamma_2=0,\quad \lambda_1\gamma_4=\lambda_2\gamma_5,\quad \lambda_6\gamma_4=0,\label{eq:x1y1z0-2}\\
\lambda_6\gamma_4=\lambda_4\gamma_1,\quad \lambda_6\gamma_5=\lambda_4\gamma_2,\quad\lambda_6\gamma_6=\lambda_4\gamma_3,\label{eq:x1y1z0-3}\\
(\lambda_5-\gamma_1)\gamma_1+\gamma_2\gamma_4=(\lambda_5-\gamma_1)\gamma_2+\gamma_2\gamma_5=(\lambda_5-\gamma_1)\gamma_3+\gamma_2\gamma_6=0,\label{eq:x1y1z0-4}\\
(\lambda_5-\gamma_5)\gamma_4+\gamma_1\gamma_4=(\lambda_5-\gamma_5)\gamma_5+\gamma_2\gamma_4=(\lambda_5-\gamma_5)\gamma_6+\gamma_3\gamma_4=0,\label{eq:x1y1z0-5}\\
\lambda_3\gamma_1=\lambda_3\gamma_2=\lambda_3\gamma_3=0=\lambda_4\gamma_4=\lambda_4\gamma_5=\lambda_4\gamma_6,\label{eq:x1y1z0-6}\\
\lambda_6\gamma_1=\lambda_6\gamma_5.\label{eq:x1y1z0-7}
\end{gather}
\end{lem}
\begin{proof}

Similar to the proof of Lemma \ref{lem:p4-x1y1z1}, we have $gx-xg=\lambda_1g(1-g)$, $gy-yg=\lambda_2g(1-g)$ and $gz-zg=0$ in $H$ for some $\lambda_1,\lambda_2\in\I_{0,1}$. Moreover, $x^p-\lambda_1x,y^p-\lambda_2y,z^p\in\Pp(H)$, $xy-yx-\lambda_2x+\lambda_1y\in\Pp_{1,g^2}(H)$ and $xz-zx,yz-zy\in\Pp_{1,g}(H)$. Since $\Pp(H)=\K\{z\}$ and $\Pp_{1,g}(H)=\K\{1-g,x,y\}$, it follows that
\begin{gather*}
x^p-\lambda_1x=\lambda_3z,\quad y^p-\lambda_2y=\lambda_4z,\quad z^p=\lambda_5z,\\
xz-zx=\gamma_1x+\gamma_2y+\gamma_3(1-g),\quad yz-zy=\gamma_4x+\gamma_5y+\gamma_6(1-g),
\end{gather*}
for $\lambda_3,\lambda_4,\lambda_5,\gamma_1,\cdots,\gamma_6\in\K$.

If $g^2=1$, then $xy-yx-\lambda_2x+\lambda_1y\in\Pp(H)$ and hence $xy-yx-\lambda_2x+\lambda_1y=\lambda_6z$ for some $\lambda_6\in\K$; otherwise, $xy-yx-\lambda_2x+\lambda_1y=\lambda_7(1-g^2)$ for some $\lambda_7\in\K$.

Assume that $p=2$. Then it follows by a direct computation that
\begin{align*}
[x,[x,y]]-[x^2,y]=\lambda_6[x,z]-\lambda_3[z,y],\\
[x,[x,z]]-[x^2,z]=\gamma_2[x,y]+\gamma_3[g,x]-\lambda_1[x,z],\\
[[x,y],y]-[x,y^2]=\lambda_6[y,z]-\lambda_4[x,z],\\
[[x,z],z]-[x,z^2]=\gamma_1[x,z]+\gamma_2[y,z]-\lambda_5[x,z],\\
[y,[y,z]]-[y^2,z]=\gamma_4[x,y]+\gamma_6[g,y]-\lambda_2[y,z],\\
[[y,z],z]-[y,z^2]=\gamma_4[x,z]+\gamma_5[y,z]-\lambda_5[y,z].
\end{align*}
Then the verification of $(a^2)b=a(ab)$ and $a(b^2)=(ab)b$ for $a,b\in\{g,x,y,z\}$ amounts to the conditions \eqref{eq:x1y1z0-1}--\eqref{eq:x1y1z0-6}.
Then it follows by a direct computation that the ambiguities $(ab)c=a(bc)$ for $a,b,c\in\{g,x,y,z\}$ give the conditions \eqref{eq:x1y1z0-7}.
\end{proof}

\begin{rmk}
The Hopf subalgebras of $H$ in Lemma \ref{lem:p4-x1y1z0} generated by $g, x, y$ or by $g, x, z$
appeared in \cite{NW} as examples of pointed Hopf algebras over k of dimension $p^3$.
\end{rmk}

\begin{lem}\label{lem:p4-x1yu}
Let $p$ be a prime number and $\Char\K=p$. Let $H$ be a  pointed Hopf algebra over $\K$.   Assume that $\gr H=\K[g,h,x,y]/(g^p-1,h^{p^n}-1,x^p,y^p)$ with $g,h\in\G(H)$, $x\in\Pp_{1,g}(H)$ and $y\in\Pp_{1,g^{\mu}}(H)$ for $\mu\in\I_{0,p-1}$. If $\mu=0$, then  the defining relations of $H$ are
\begin{gather*}
g^p=1,\quad h^{p^n}=1,\quad gx-xg=\lambda_1g(1-g), \quad gy-yg=0,\\
 hx-xh=\lambda_3h(1-g),\quad hy-yh=0,\\
x^p-\lambda_1x=\mu_1y,\quad y^p=\mu_2y,\quad xy-yx=\mu_3x+\mu_4(1-g),
\end{gather*}
for $\lambda_1\in\I_{0,1},\lambda_3,\mu_1,\cdots,\mu_4\in\K$ with  conditions
\begin{align*}
\mu_1\mu_3=0=\mu_1\mu_4,\quad \mu_2\mu_3=\mu_3^p,\quad \mu_2\mu_4=\mu_3^{p-1}\mu_4,\quad \lambda_1\mu_3=0=\mu_3\lambda_3.
\end{align*}
If $\mu\neq 0$, then the defining relations are
\begin{gather*}
g^p=1,\quad h^{p^n}=1,\quad  gx-xg=\lambda_1g(1-g), \quad gy-yg=\lambda_2g(1-g^{\mu}),\\
hx-xh=\lambda_3h(1-g),\quad hy-yh=\lambda_4h(1-g^{\mu}),\\
x^p-\lambda_1x=0,\quad y^p- \lambda_2y=0,\quad xy-yx+\mu\lambda_1y-\lambda_2x=\lambda_5(1-g^{\mu+1}).
\end{gather*}
for $\lambda_1,\lambda_2\in\I_{0,1},\lambda_3,\cdots,\lambda_5\in\K$ satisfying the conditions
\begin{align*}
\lambda_1\lambda_4(1-g^{\mu+1})=0=\lambda_2\lambda_3(1-g^{\mu+1}).
\end{align*}
\end{lem}
\begin{proof}
By similar computations as before, we have
\begin{gather*}
gx-xg=\lambda_1g(1-g), \quad gy-yg=\lambda_2g(1-g^{\mu}),\\
hx-xh=\lambda_3h(1-g),\quad hy-yh=\lambda_4h(1-g^{\mu}),\\
x^p-\lambda_1x\in\Pp(H),\quad y^p-\mu^{p-1}\lambda_2y\in\Pp(H),\quad xy-yx+\mu\lambda_1y-\lambda_2x\in\Pp_{1,g^{\mu+1}}(H).
\end{gather*}
for some $\lambda_1,\lambda_2\in\I_{0,1}$, $\lambda_3,\lambda_4\in\K$.

If $\mu=0$, then $\Pp(H)=\K\{y\}$ and $\Pp_{1,g}(H)=\K\{1-g,x\}$. Hence
\begin{align*}
x^p-\lambda_1x=\mu_1y,\quad y^p=\mu_2y,\quad xy-yx=\mu_3x+\mu_4(1-g).
\end{align*}
for some $\mu_1,\cdots,\mu_4\in\K$. The verification of $(x^p)x=x(x^p)$ and $(y^p)y=y(y^p)$ amounts to the conditions
\begin{align*}
\mu_1\mu_3=0=\mu_1\mu_4.
\end{align*}
By induction, for any $n>1$, we have $(x)(\AdR y)^n=\mu_3(x)(\AdR y)^{n-1}$ and $(\AdL x)^n(y)=(-\mu_4)(\AdL x)^{n-1}(g)$. Then by Lemma \ref{pqlem1},
\begin{align*}
[x,y^p]&=\mu_2[x,y]=\mu_2\mu_3x+\mu_2\mu_4(1-g),\\
(x)(\AdR y)^p&=\mu_3(x)(\AdR y)^{p-1}=\mu_3^{p-1}[x,y]=\mu_3^{p}x+\mu_3^{p-1}\mu_4(1-g);\\
[x^p,y]&=\lambda_1[x,y]=\lambda_1\mu_3x+\lambda_1\mu_4(1-g),\\
(\AdL x)^p(y)&=-\mu_4(\AdL x)^{p-1}(g)=-\mu_4\lambda_1^{p-1}(g-1)=\mu_4\lambda_1^{p-1}(1-g).
\end{align*}
Hence by Proposition \ref{proJ}, $[x,y^p]=(x)(\AdR y)^p$ and $[x^p,y]=(\AdL x)^p(y)$, which implies that
\begin{align*}
\mu_2\mu_3=\mu_3^p,\quad \mu_2\mu_4=\mu_3^{p-1}\mu_4,\quad \lambda_1\mu_3=0.
\end{align*}
Finally, it follows by a direct computation that $a(xy)=(ax)y$ and $(gh)b=g(hb)$ for $a\in\{g,h\}, b\in\{x,y\}$ amounts to the conditions
\begin{align*}
\mu_3\lambda_3=0=\mu_3\lambda_1.
\end{align*}

If $\mu\neq 0$, then $\Pp(H)=0$ and $\Pp_{1,g^{\mu+1}}(H)=\K\{1-g^{\mu+1}\}$. By Fermat's little theorem, $\mu^{p-1}=1$. Hence
\begin{align*}
x^p-\lambda_1x=0,\quad y^p- \lambda_2y=0,\quad xy-yx+\mu\lambda_1y-\lambda_2x=\lambda_5(1-g^{\mu+1}).
\end{align*}
The verification of $(hx)y=h(xy)$ amounts to the conditions
\begin{align*}
\lambda_1\lambda_4(1-g^{\mu+1})=0=\lambda_2\lambda_3(1-g^{\mu+1}).
\end{align*}
Then using Lemmas \ref{pqlem1} and \ref{pqlem2}, it follows by a direct computation that the ambiguities $a^{p-1}(ab)=(a^p)b$, $(ab)b^{p-1}=a(b^p)$ for $a,b\in\{g,x,y\}$  and $g(xy)=(gx)y$ are resolvable. By the Diamond lemma, $\dim H=p^{3+n}$.
\end{proof}

\begin{rmk}
The Hopf subalgebras of $H$ in Lemma \ref{lem:p4-x1yu} generated by $g, h, y$ appeared
in \cite{NW} as examples of pointed Hopf algebras over $\K$ of dimension $p^3$.
\end{rmk}

\begin{lem}\label{lem:p4-xg1yhu}
Let $p$ be a prime number and $\Char\K=p$. Let $H$ be a  pointed Hopf algebra over $\K$ of dimension $p^4$. Assume that $\gr H=\K[g,h,x,y]/(g^p-1,h^p-1,x^p,y^p)$ with $g,h\in\G(H)$, $x\in\Pp_{1,g}(H)$ and $y\in\Pp_{1,h^{\mu}}(H)$ for $\mu\in\I_{1,p-1}$.
 Then the defining relations of $H$ have the following form
\begin{gather*}
gx-xg=\lambda_1g(1-g),\quad hx-xh=\lambda_2h(1-g), \quad x^p-\lambda_1x=0\\
gy-yg=\lambda_3g(1-h^{\mu}), \quad hy-yh=\lambda_4h(1-h^{\mu}),\quad y^p- \lambda_4y=0,\\
 xy-yx-\lambda_3x+\mu\lambda_2y=\lambda_4(1-gh^{\mu}).
\end{gather*}
for some $\lambda_1,\lambda_4\in\I_{0,1}$, $\lambda_2,\lambda_3,\lambda_4\in\K$.
\end{lem}
\begin{proof}

Observe that $\mu\neq 0$, then $\Pp(H)=0$ and $\Pp_{1,gh^{\mu}}(H)=\K\{1-gh^{\mu}\}$. By Fermat's little theorem $\mu^{p-1}=1$. By similar computations as before, we have
\begin{align*}
gx-xg&=\lambda_1g(1-g),\quad hx-xh=\lambda_2h(1-g), \quad x^p-\lambda_1x=0,\\
gy-yg&=\lambda_3g(1-h^{\mu}), \quad hy-yh=\lambda_4h(1-h^{\mu}),\quad y^p- \lambda_4y=0.
\end{align*}
for some $\lambda_1,\lambda_4\in\I_{0,1}$, $\lambda_2,\lambda_3\in\K$.  Now we determine $\Delta(xy-yx)$. Observe that $h^{\mu}x=xh^{\mu}+\lambda_2\mu h^{\mu}(1-g)$. Then
\begin{align*}
\Delta(xy-yx)&=(x\otimes 1+g\otimes x)(y\otimes 1+h^{\mu}\otimes y)-(y\otimes 1+h^{\mu}\otimes y)(x\otimes 1+g\otimes x)\\
&=(xy-yx)\otimes 1+(gy-yg)\otimes x-(h^{\mu}x-xh^{\mu})\otimes y+gh^{\mu}\otimes (xy-yx)\\
&=(xy-yx)\otimes 1+\lambda_3g(1-h^{\mu})\otimes x-\lambda_2\mu h^{\mu}(1-g)\otimes y+gh^{\mu}\otimes(xy-yx).
\end{align*}
One can check that $xy-yx-\lambda_3x+\mu\lambda_2y\in\Pp_{1,gh^{\mu}}(H)$, which implies that
\begin{align*}
  xy-yx-\lambda_3x+\mu\lambda_2y=\lambda_5(1-gh^{\mu}).
\end{align*}
for some $\lambda_5\in\K$.
\end{proof}

\begin{rmk}
The Hopf subalgebras of $H$ in Lemma \ref{lem:p4-xg1yhu} generated by $g, h, y$ appeared
in \cite{NW} as examples of pointed Hopf algebras over $\K$ of dimension $p^3$.
\end{rmk}
\section{Non-connected pointed Hopf algebras of dimension $16$ whose diagrams are Nichols algebras}\label{sec:16}
In this section, we assume that $\Char\K=2$ and give a complete classification of non-connected pointed Hopf algebras of dimension $16$ over $\K$ whose diagrams are Nichols algebras by means of the Lifting Method. The strategy can be divided into the following steps:
\begin{enumerate}
  \item Determine all possible coradicals $H_0$ such that $\dim H_0\mid 16$;
  \item Given $H_0$ as in the previous item, determine all possible Nichols algebras $\BN(V)$ in ${}_{H_0}^{H_0}\mathcal{YD}$ such that $\dim\BN(V)\sharp H_0=16$;
  \item Given $H_0$ and $\BN(V)$ as in previous items, determine isomorphism classes of $\BN(V)\sharp H_0$;
  \item Given $\BN(V)\sharp H_0$, compute all possible deformations $H$ of $\BN(V)\sharp H_0$ such
that $\dim H=\dim\BN(V)\sharp H_0$ and finally determine their isomorphism classes.
\end{enumerate}

\begin{lem}\label{lem:16-group-like}
Let $H$ be a pointed non-connected Hopf algebra over $\K$ of dimension $16$. Then $\G(H)$ is isomorphic to the Dihedral group $D_4$, the quaternions group $Q_8$, $\Z_8$, $\Z_4\times \Z_2$, $\Z_2\times \Z_2\times \Z_2$, $\Z_4$, $\Z_2\times \Z_2$ or $\Z_2$.
\end{lem}
\begin{proof}
By Nichols Zoeller theorem, $|\G(H)|$ must divide $16$. By the assumption, $|\G(H)|\in\{8,4,2\}$ and hence the lemma follows.
\end{proof}

Recall that $D_4:=\langle g,h\mid g^4=1,h^2=1,hg=g^3h\rangle$, $Q_8:=\langle g,h\mid g^4=1,hg=g^3h,g^2=h^2\rangle$. Now we give a complete classification of non-connected pointed Hopf algebras of dimension $16$ whose diagrams are Nichols algebras.
\begin{thm}\label{thm:16-diagram-Nichols-algebra}
Let $H$ be a non-trivial non-connected pointed  Hopf algebra over $\K$ of dimension $16$ whose diagram is a Nichols algebra. Then $H$ is isomorphic to one of the following Hopf algebras
\begin{description}
  \item[(1)] $\K[D_4]\otimes\K[x]/(x^2)$, with $x\in\Pp(H)$;
  \item[(2)] $\K[D_4]\otimes\K[x]/(x^2-x)$, with $x\in\Pp(H)$;

  \item[(3)] $\K\langle g,h,x\rangle/(g^4-1,h^2-1,hg-g^3h,[g,x],[h,x],x^2)$, with $g,~h\in\G(H)$ and $x\in\Pp_{1,g^2}(H)$;
  \item[(4)] $\K\langle g,h,x\rangle/(g^4-1,h^2-1,hg-g^3h,[g,x],[h,x]-h(1-g^2),x^2)$, with $g,~h\in\G(H)$ and $x\in\Pp_{1,g^2}(H)$;
  \item[(5)] $\HH_1(\lambda):=\K\langle g,h,x\rangle/(g^4-1,h^2-1,hg-g^3h,[g,x]-g(1-g^2),[h,x]-\lambda h(1-g^2),x^2)$, for $\lambda\in\K$,
with $g,~h\in\G(H)$ and $x\in\Pp_{1,g^2}(H)$; moreover,
\begin{itemize}
\item $\HH_1(\lambda)\cong\HH_1(\gamma)$ for $\lambda,\gamma\in\K$, if and only if, $\lambda=\gamma+i$ for some $i\in\I_{0,1}$;
\end{itemize}

  \item[(6)] $\K[Q_8]\otimes\K[x]/(x^2)$, with $x\in\Pp(H)$;
  \item[(7)] $\K[Q_8]\otimes\K[x]/(x^2-x)$, with $x\in\Pp(H)$;

  \item[(8)] $\K\langle g,h,x\rangle/(g^4-1,hg-g^3h,g^2-h^2,[g,x],[h,x],x^2)$, with $g,~h\in\G(H)$,  $x\in\Pp_{1,g^2}(H)$;
  \item[(9)] $\HH_2(\lambda):=\K\langle g,h,x\rangle/(g^4-1,hg-g^3h,g^2-h^2,[g,x]-g(1-g^2),[h,x]-\lambda h(1-g^2),x^2)$, for $\lambda \in\K$,
with $g,~h\in\G(H)$, $x\in\Pp_{1,g^2}(H)$; moreover,
\begin{itemize}
\item $\HH_2(\lambda)\cong\HH_2(\gamma)$ for $\lambda,\gamma\in\K $, if and only if, $\lambda=\gamma+i$ or $(\lambda-j)(\gamma-i)=1$ for some  $i,j\in\I_{0,1}$;
\end{itemize}

  \item[(10)] $\K[\Z_8]\otimes\K[x]/(x^2)$, with $x\in\Pp(H)$;
  \item[(11)] $\K[\Z_8]\otimes\K[x]/(x^2-x)$, with $x\in\Pp(H)$;

  \item[(12)] $\K[g,x]/(g^8-1,x^2)$, with $g\in\G(H)$, $x\in\Pp_{1,g^{\mu}}(H)$ for $\mu\in\{1,2,4\}$;
  \item[(13)] $\K\langle g,x\rangle/(g^8-1,[g,x]-g(1-g^{\mu}),x^2-\mu x)$,
  with $g\in\G(H)$, $x\in\Pp_{1,g^{\mu}}(H)$ for $\mu\in\{1,4\}$;

  \item[(14)] $\K[\Z_4\times \Z_2]\otimes\K[x]/(x^2)$, with $x\in\Pp(H)$;
  \item[(15)] $\K[\Z_4\times \Z_2]\otimes\K[x]/(x^2-x)$, with $x\in\Pp(H)$;

  \item[(16)] $\K[g,h,x]/(g^4-1,h^2-1, x^2)$, with $g,h\in\G(H)$, $x\in\Pp_{1,g^{\mu}}(H)$ for $\mu\in\{1,2\}$;
  \item[(17)] $\K\langle g,h,x\rangle/(g^4-1,h^2-1,[g,h], [g,x], [h,x]-h(1-g^{\mu}),x^2)$, with $g,h\in\G(H)$, $x\in\Pp_{1,g^{\mu}}(H)$ for $\mu\in\{1,2\}$;
  \item[(18)] $\HH_{3,\mu}(\lambda):=\K\langle g,h,x\rangle/(g^4-1,h^2-1,[g,h], [g,x]-g(1-g^{\mu}), [h,x]-\lambda h(1-g^{\mu}),x^2-\mu x)$ for   $\lambda\in\K$,
  with $g,h\in\G(H)$, $x\in\Pp_{1,g^{\mu}}(H)$ for $\mu\in\{1,2\}$;

  \item[(19)] $\K[g,h,x]/(g^4-1,h^2-1,x^2)$, with $g,h\in\G(H)$ and $x\in\Pp_{1,h}(H)$;
  \item[(20)] $\K\langle g,h,x\rangle/(g^4-1,h^2-1,[g,h], [g,x]- g(1-h), [h,x],x^2)$, with $g,h\in\G(H)$ and $x\in\Pp_{1,h}(H)$;
  \item[(21)] $\HH_{4}(\lambda):=\K\langle g,h,x\rangle/(g^4-1,h^2-1, [g,h], [g,x]-\lambda g(1-h), [h,x]-h(1-h),x^2-x)$ for $\lambda\in\K$,
with $g,h\in\G(H)$ and $x\in\Pp_{1,h}(H)$; moreover,
\begin{itemize}
\item $\HH_{3,1}(\lambda)\cong\HH_{3,1}(\gamma)$, if and only if, $\lambda=\gamma$;
\item $\HH_{3,2}(\lambda)\cong\HH_{3,2}(\gamma)$, if and only if, $\lambda=\gamma$ or $\lambda\gamma=\lambda+\gamma$;
\item $\HH_4(\lambda)\cong\HH_4(\gamma)$, if and only if, $\lambda=\gamma+i$ for $i\in\I_{0,1}$;

\end{itemize}

  \item[(22)] $\K[\Z_2\times \Z_2\times \Z_2]\otimes \K[x]/(x^2)$, with $x\in\Pp(H)$;
  \item[(23)] $\K[\Z_2\times \Z_2\times \Z_2]\otimes \K[x]/(x^2-x)$, with $x\in\Pp(H)$;

  \item[(24)] $\K[g,h,k,x]/(g^2-1,h^2-1,k^2-1,x^2)$, with $g,h,k\in\G(H)$, $x\in\Pp_{1,g}(H)$;
  \item[(25)] $\HH_5(\lambda):=\K\langle g,h,k,x\rangle/(g^2-1,h^2-1,k^2-1, [g,h],[g,k],[h,k], [g,x], [h,x]-h(1-g), [k,x]-\lambda k(1-g),x^2)$ for $\lambda\in\K$, with $g,h,k\in\G(H)$, $x\in\Pp_{1,g}(H)$;
  \item[(26)] $\HH_6(\lambda,\gamma):=\K\langle g,h,k,x\rangle/(g^2-1,h^2-1,k^2-1, [g,h], [g,k], [h,k], [g,x]-g(1-g), [h,x]-\lambda h(1-g), [k,x]-\gamma k(1-g),x^2-x)$ for $\lambda,\gamma\in\K$,
with $g,h,k\in\G(H)$, $x\in\Pp_{1,g}(H)$; moreover,
\begin{itemize}
\item $\HH_5(\lambda)\cong\HH_5(\gamma)$, if and only if,
\begin{align*}
\lambda\gamma=\lambda+\gamma,\quad\text{or }(1+\lambda)\gamma=1, \quad\text{or } \lambda=\gamma+i,\quad \text{or } 1+i\gamma=\lambda\gamma,\quad i\in\I_{0,1};
\end{align*}
\item $\HH_6(\lambda_1,\lambda_2)\cong\HH_6(\gamma_1,\gamma_2)$, if and only if, there exist $q,r,\nu,\iota\in\I_{0,1}$ such that
\begin{align*}
q\iota+r\nu=1,\quad q\gamma_1+r\gamma_2 =\lambda_1,\quad \nu\gamma_1+\iota\gamma_2=\lambda_2;
\end{align*}
\end{itemize}

  \item[(27)] $\K[\Z_4]\otimes\K[x,y]/(x^2,y^2)$, with $x,y\in\Pp(H)$;
  \item[(28)] $\K[\Z_4]\otimes\K[x,y]/(x^2-x,y^2)$, with $x,y\in\Pp(H)$;
  \item[(29)] $\K[\Z_4]\otimes\K[x,y]/(x^2-y,y^2)$, with $x,y\in\Pp(H)$;
  \item[(30)] $\K[\Z_4]\otimes\K[x,y]/(x^2-x,y^2-y)$, with $x,y\in\Pp(H)$;
  \item[(31)] $\K[\Z_4]\otimes\K\langle x,y\rangle/([x,y]-y,x^2-x,y^2)$, with $x,y\in\Pp(H)$;

  \item[(32)] $\K[g,x,y]/(g^4-1,x^2,y^2)$, with $g\in\G(H)$, $x\in\Pp(H)$ and $y\in\Pp_{1,g}(H)$;
  \item[(33)] $\K\langle g,x,y\rangle/(g^4-1,[g,x],[g,y],x^2,y^2,[x,y]-(1-g))$, with $g\in\G(H)$, $x\in\Pp(H)$ and $y\in\Pp_{1,g}(H)$;
  \item[(34)] $\K[g,x,y]/(g^4-1,x^2-x,y^2)$, with $g\in\G(H)$, $x\in\Pp(H)$ and $y\in\Pp_{1,g}(H)$;

  \item[(35)] $\K\langle g,x,y\rangle/(g^4-1,[g,x],[g,y],x^2-x,y^2,[x,y]-y)$, with $g\in\G(H)$, $x\in\Pp(H)$ and $y\in\Pp_{1,g}(H)$;
  \item[(36)] $\K\langle g,y\rangle/(g^4-1,[g,y]-g(1-g),y^2-y)\otimes\K[x]/(x^2)$, with $g\in\G(H)$, $x\in\Pp(H)$ and $y\in\Pp_{1,g}(H)$;
  \item[(37)] $\K\langle g,y\rangle/(g^4-1,[g,y]-g(1-g),y^2-y)\otimes\K[x]/(x^2-x)$, with $g\in\G(H)$, $x\in\Pp(H)$ and $y\in\Pp_{1,g}(H)$;

  \item[(38)] $\K[g,y]/(g^4-1,y^2)\otimes\K[x]/(x^2)$, with $g\in\G(H)$, $x\in\Pp(H)$ and $y\in\Pp_{1,g^2}(H)$;
  \item[(39)] $\K\langle g,y\rangle/(g^4-1,[g,y]-g(1-g^2),y^2)\otimes\K[x]/(x^2)$, with $g\in\G(H)$, $x\in\Pp(H)$ and $y\in\Pp_{1,g^2}(H)$;
  \item[(40)] $\K[g,x,y]/(g^4-1, x^2,y^2-x)$, with $g\in\G(H)$, $x\in\Pp(H)$ and $y\in\Pp_{1,g^2}(H)$;
  \item[(41)] $\K\langle g,x,y\rangle/(g^4-1,[g,x],[g,y]-g(1-g^2),x^2,y^2-x,[x,y])$, with $g\in\G(H)$, $x\in\Pp(H)$ and $y\in\Pp_{1,g^2}(H)$;
  \item[(42)] $\K\langle g,x,y\rangle/(g^4-1,[g,x],[g,y],x^2,y^2,[x,y]-(1-g^2))$, with $g\in\G(H)$, $x\in\Pp(H)$ and $y\in\Pp_{1,g^2}(H)$;
  \item[(43)] $\K\langle g,x,y\rangle/(g^4-1,[g,x],[g,y]-g(1-g^2),x^2,y^2,[x,y]-(1-g^2))$, with $g\in\G(H)$, $x\in\Pp(H)$ and $y\in\Pp_{1,g^2}(H)$;
  \item[(44)] $\K[g,y]/(g^4-1,y^2)\otimes\K[x]/(x^2-x)$, with $g\in\G(H)$, $x\in\Pp(H)$ and $y\in\Pp_{1,g^2}(H)$;
  \item[(45)] $\K[g,x,y]/(g^4-1,x^2-x,y^2-x)$, with $g\in\G(H)$, $x\in\Pp(H)$ and $y\in\Pp_{1,g^2}(H)$;
  \item[(46)] $\HH_{7}(\lambda):=\K\langle g,x,y\rangle/(g^4-1,[g,x],[g,y]-g(1-g^2),x^2-x,y^2-\lambda x,[x,y])$, with $g\in\G(H)$, $x\in\Pp(H)$ and $y\in\Pp_{1,g^2}(H)$;
  \item[(47)] $\K\langle g,x,y\rangle/(g^4-1,[g,x],[g,y],x^2-x,y^2,[x,y]-y)$, with $g\in\G(H)$, $x\in\Pp(H)$ and $y\in\Pp_{1,g^2}(H)$;
  \item[(48)] $\K\langle g,x,y\rangle/(g^4-1,[g,x],[g,y]-g(1-g^2),x^2-x,y^2,[x,y]-y)$,
with $g\in\G(H)$, $x\in\Pp(H)$ and $y\in\Pp_{1,g^2}(H)$;
\begin{itemize}
\item $\HH_7(\lambda)\cong\HH_7(\gamma)$, if and only if, $\lambda=\gamma$;
\end{itemize}

  \item[(49)] $\K[g,x,y]/(g^4-1,x^2,y^2)$, with $g\in\G(H)$, $x,y\in\Pp_{1,g}(H)$;
  \item[(50)]  $\K\langle g,x,y\rangle/(g^4-1,[g,x],[g,y],x^2,y^2,[x,y]-(1-g^2))$, with $g\in\G(H)$, $x,y\in\Pp_{1,g}(H)$;
  \item[(51)]  $\K\langle g,x,y\rangle/(g^4-1,[g,x]-g(1-g),[g,y],x^2-x,y^2,[x,y]+y)$, with $g\in\G(H)$, $x,y\in\Pp_{1,g}(H)$;
  \item[(52)]  $\K\langle g,x,y\rangle/(g^4-1,[g,x]-g(1-g),[g,y],x^2-x,y^2,[x,y]+y-(1-g^2))$, with $g\in\G(H)$, $x,y\in\Pp_{1,g}(H)$;

  \item[(53)]  $\K[g,x,y]/(g^4-1,x^2,y^2)$,  with $g\in\G(H)$, $x\in\Pp_{1,g}(H)$ and $y\in\Pp_{1,g^2}(H)$;
  \item[(54)] $\K[g,x,y]/(g^4-1,x^2-y,y^2)$,  with $g\in\G(H)$, $x\in\Pp_{1,g}(H)$ and $y\in\Pp_{1,g^2}(H)$;
  \item[(55)] $\K\langle g,x,y\rangle/(g^4-1,[g,x],[g,y],x^2,y^2,[x,y]-(1-g^3))$,  with $g\in\G(H)$, $x\in\Pp_{1,g}(H)$ and $y\in\Pp_{1,g^2}(H)$;
  \item[(56)] $\K\langle g,x,y\rangle/(g^4-1,[g,x]-g(1-g),[g,y],x^2-x,y^2,[x,y])$,  with $g\in\G(H)$, $x\in\Pp_{1,g}(H)$ and $y\in\Pp_{1,g^2}(H)$;
 \item[(57)] $\K\langle g,x,y\rangle/(g^4-1,[g,x]-g(1-g),[g,y],x^2-x-y,y^2,[x,y])$,  with $g\in\G(H)$, $x\in\Pp_{1,g}(H)$ and $y\in\Pp_{1,g^2}(H)$;
  \item[(58)] $\K\langle g,x,y\rangle/(g^4-1,[g,x]-g(1-g),[g,y],x^2-x,y^2,[x,y]-(1-g^3))$,
  with $g\in\G(H)$, $x\in\Pp_{1,g}(H)$ and $y\in\Pp_{1,g^2}(H)$;

  \item[(59)] $\K[g,x,y]/(g^4-1,x^2,y^2)$, with $g\in\G(H)$, $x\in\Pp_{1,g}(H)$ and $y\in\Pp_{1,g^3}(H)$;
  \item[(60)] $\K\langle g,x,y\rangle/(g^4-1,[g,x]-g(1-g),[g,y],x^2-x,y^2,[x,y]+y)$, with $g\in\G(H)$, $x\in\Pp_{1,g}(H)$ and $y\in\Pp_{1,g^3}(H)$;
  \item[(61)] $\K\langle g,x,y\rangle/(g^4-1,[g,x]-g(1-g),[g,y]-g(1-g^3),x^2-x,y^2-y,[x,y]+y-x)$,
  with $g\in\G(H)$, $x\in\Pp_{1,g}(H)$ and $y\in\Pp_{1,g^3}(H)$;

  \item[(62)] $\K[g,x,y]/(g^4-1,x^2,y^2)$,  with $g\in\G(H)$, $x,y\in\Pp_{1,g^2}(H)$;
  \item[(63)] $\K\langle g,x,y\rangle/(g^4-1,[g,x]-g(1-g^2),[g,y],x^2,y^2,[x,y])$,
  with $g\in\G(H)$, $x,y\in\Pp_{1,g^2}(H)$;

   \item[(64)] $\K\langle g,x,y\rangle/(g^4-1,[g,x],gy-(y+x)g,[x,y],x^2,y^2)$, with $g\in\G(H)$, $x,y\in\Pp(H)$;
  \item[(65)] $\K\langle g,x,y\rangle/(g^4-1,[g,x],gy-(y+x)g,[x,y],x^2-x,y^2)$, with $g\in\G(H)$, $x,y\in\Pp(H)$;
  \item[(66)] $\K\langle g,x,y\rangle/(g^4-1,[g,x],gy-(y+x)g,[x,y],x^2-y,y^2)$, with $g\in\G(H)$, $x,y\in\Pp(H)$;
  \item[(67)] $\K\langle g,x,y\rangle/(g^4-1,[g,x],gy-(y+x)g,[x,y],x^2-x,y^2-y)$, with $g\in\G(H)$, $x,y\in\Pp(H)$;
  \item[(68)] $\K\langle g,x,y\rangle/(g^4-1,[g,x],gy-(y+x)g,[x,y]-y,x^2-x,y^2)$
  with $g\in\G(H)$, $x,y\in\Pp(H)$;

  \item[(69)] $\K\langle g,x,y\rangle/(g^4-1,[g,x],gy-(y+x)g,[x,y],x^2,y^2)$, with $g\in\G(H)$, $x,y\in\Pp_{1,g^2}(H)$;
  \item[(70)] $\K\langle g,x,y\rangle/(g^4-1,[g,x]-g(1-g^2),gy-(y+x)g,[x,y],x^2,y^2)$,
  with $g\in\G(H)$, $x,y\in\Pp_{1,g^2}(H)$;

  \item[(71)] $\K[\Z_2\times \Z_2]\otimes\K[x,y]/(x^2,y^2)$,  with $x,y\in\Pp(H)$;
  \item[(72)] $\K[\Z_2\times \Z_2]\otimes\K[x,y]/(x^2-x,y^2)$,  with $x,y\in\Pp(H)$;
  \item[(73)] $\K[\Z_2\times \Z_2]\otimes\K[x,y]/(x^2-y,y^2)$,  with $x,y\in\Pp(H)$;
  \item[(74)] $\K[\Z_2\times \Z_2]\otimes\K[x,y]/(x^2-x,y^2-y)$,  with $x,y\in\Pp(H)$;
  \item[(75)] $\K[\Z_2\times \Z_2]\otimes\K\langle x,y\rangle/([x,y]-y,x^2-x,y^2)$,
  with $x,y\in\Pp(H)$;

  \item[(76)] $\K[g,h,x,y]/(g^2-1,h^2-1,x^2,y^2)$, with $g,h\in\G(H)$,  $x\in\Pp_{1,g}(H)$ and $y\in\Pp(H)$;
  \item[(77)] $\K[g,h,x,y]/(g^2-1,h^2-1,x^2-y,y^2)$, with $g,h\in\G(H)$,  $x\in\Pp_{1,g}(H)$ and $y\in\Pp(H)$;
  \item[(78)] $\K\langle g,h,x,y\rangle/(g^2-1,h^2-1,[g,h],[g,x],[h,x]-h(1-g),[g,y],[h,y],x^2,y^2,[x,y])$, with $g,h\in\G(H)$,  $x\in\Pp_{1,g}(H)$ and $y\in\Pp(H)$;
  \item[(79)] $\K\langle g,h,x,y\rangle/(g^2-1,h^2-1,[g,h],[g,x],[h,x]-h(1-g),[g,y],[h,y],x^2-y,y^2,[x,y])$,
  \item[(80)] $\K\langle g,h,x,y\rangle/(g^2-1,h^2-1,[g,h],[g,x],[h,x],[g,y],[h,y],x^2,y^2,[x,y]-(1-g))$, with $g,h\in\G(H)$,  $x\in\Pp_{1,g}(H)$ and $y\in\Pp(H)$;
  \item[(81)] $\K\langle g,h,x,y\rangle/(g^2-1,h^2-1,[g,h],[g,x],[h,x]-h(1-g),[g,y],[h,y],x^2,y^2,[x,y]-(1-g))$, with $g,h\in\G(H)$,  $x\in\Pp_{1,g}(H)$ and $y\in\Pp(H)$;
  \item[(82)] $\K[g,h,x,y]/(g^2-1,h^2-1,x^2,y^2-y)$, with $g,h\in\G(H)$,  $x\in\Pp_{1,g}(H)$ and $y\in\Pp(H)$;
  \item[(83)] $\K\langle g,h,x,y\rangle/(g^2-1,h^2-1,[g,h],[g,x],[h,x]-h(1-g),[g,y],[h,y],x^2,y^2-y,[x,y])$, with $g,h\in\G(H)$,  $x\in\Pp_{1,g}(H)$ and $y\in\Pp(H)$;
  \item[(84)] $\K\langle g,h,x,y\rangle/(g^2-1,h^2-1,[g,h],[g,x],[h,x],[g,y],[h,y],x^2-y,y^2-y,[x,y])$, with $g,h\in\G(H)$,  $x\in\Pp_{1,g}(H)$ and $y\in\Pp(H)$;
  \item[(85)] $\K\langle g,h,x,y\rangle/(g^2-1,h^2-1,[g,h],[g,x],[h,x]-h(1-g),[g,y],[h,y],x^2-y,y^2-y,[x,y])$, with $g,h\in\G(H)$,  $x\in\Pp_{1,g}(H)$ and $y\in\Pp(H)$;
  \item[(86)] $\K\langle g,h,x,y\rangle/(g^2-1,h^2-1,[g,h],[g,x],[h,x],[g,y],[h,y],x^2,y^2-y,[x,y]-x)$, with $g,h\in\G(H)$,  $x\in\Pp_{1,g}(H)$ and $y\in\Pp(H)$;
  \item[(87)] $\K\langle g,h,x,y\rangle/(g^2-1,h^2-1,[g,h],[g,x],[h,x],[g,y],[h,y],x^2-y,y^2-y,[x,y]-x)$, with $g,h\in\G(H)$,  $x\in\Pp_{1,g}(H)$ and $y\in\Pp(H)$;
  \item[(88)] $\HH_8(\lambda):= \K\langle g,h,x\rangle/(g^2-1,h^2-1,[g,h],[g,x]-g(1-g),[h,x]-\lambda h(1-g),x^2-x)\otimes\K[y]/(y^2)$, with $g,h\in\G(H)$,  $x\in\Pp_{1,g}(H)$ and $y\in\Pp(H)$;
  \item[(89)] $\HH_{9}(\lambda):=\K\langle g,h,x,y\rangle/(g^2-1,h^2-1,[g,h],[g,x]-g(1-g),[h,x]-\lambda h(1-g),[g,y],[h,y],x^2-x-y,y^2,[x,y])$, with $g,h\in\G(H)$,  $x\in\Pp_{1,g}(H)$ and $y\in\Pp(H)$;
  \item[(90)] $\HH_{10}(\lambda):=\K\langle g,h,x,y\rangle/(g^2-1,h^2-1,[g,h],[g,x]-g(1-g),[h,x]-\lambda h(1-g),[g,y],[h,y],x^2-x,y^2,[x,y]-(1-g))$, with $g,h\in\G(H)$,  $x\in\Pp_{1,g}(H)$ and $y\in\Pp(H)$;
  \item[(91)] $\HH_{11}(\lambda,\gamma):=\K\langle g,h,x,y\rangle/(g^2-1,h^2-1,[g,h],[g,x]-g(1-g),[h,x]-\lambda h(1-g),[g,y],[h,y],x^2-x-\gamma y,y^2-y,[x,y])$,
with $g,h\in\G(H)$,  $x\in\Pp_{1,g}(H)$ and $y\in\Pp(H)$;
\begin{itemize}
\item $\HH_n(\lambda)\cong\HH_n(\gamma)$ for $n\in\I_{8,10}$, if and only if, $\lambda=\gamma+i$ for $i\in\I_{0,1}$;
\item $\HH_{11}(\lambda,\mu)\cong\HH_{11}(\gamma,\nu)$ if and only if $\lambda=\gamma+i$ for $i\in\I_{0,1}$ and $\mu=\nu$;
\end{itemize}

  \item[(92)] $\K[g,h,x,y]/(g^2-1,h^2-1,x^2,y^2)$, with $g,h\in\G(H)$ and $x,y\in\Pp_{1,g}(H)$;
  \item[(93)] $\K\langle g,h,x,y\rangle/(g^2-1,h^2-1,[g,h],[g,x],[g,y],[h,x]-h(1-g),[h,y],x^2,y^2,[x,y])$,
  with $g,h\in\G(H)$ and $x,y\in\Pp_{1,g}(H)$;
  \item[(94)] $\HH_{12}(\lambda):=\K\langle g,h,x,y\rangle/(g^2-1,h^2-1,[g,h],[g,x]-g(1-g),[g,y],[h,x]-\lambda h(1-g),[h,y],x^2-x,y^2,[x,y]+y)$, with $g,h\in\G(H)$ and $x,y\in\Pp_{1,g}(H)$;
  \item[(95)] $\HH_{13}(\lambda):=\K\langle g,h,x,y\rangle/(g^2-1,h^2-1,[g,h],[g,x]-g(1-g),[g,y],[h,x]-\lambda h(1-g),[h,y]-h(1-g),x^2-x,y^2,[x,y]+y)$,
with $g,h\in\G(H)$ and $x,y\in\Pp_{1,g}(H)$; moreover,
\begin{itemize}
\item $\HH_{n}(\lambda)\cong\HH_{n}(\gamma)$ for $n\in\I_{12,13}$, if and only if, $\lambda=\gamma+i$ for $i\in\I_{0,1}$;
\end{itemize}

  \item[(96)] $\K[g,h,x,y]/(g^2-1,h^2-1,x^2,y^2)$, with $g,h\in\G(H)$, $x\in\Pp_{1,g}(H)$ and $y\in\Pp_{1,h}(H)$;
  \item[(97)] $\K\langle g,h,x,y\rangle/(g^2-1,h^2-1,[g,h],[g,x],[g,y],[h,x],[h,y],x^2,y^2,[x,y]-(1-gh))$, with $g,h\in\G(H)$, $x\in\Pp_{1,g}(H)$ and $y\in\Pp_{1,h}(H)$;
  \item[(98)] $\K\langle g,h,x,y\rangle/(g^2-1,h^2-1,[g,h],[g,x]-g(1-g),[g,y],[h,x],[h,y],x^2-x,y^2,[x,y])$, with $g,h\in\G(H)$, $x\in\Pp_{1,g}(H)$ and $y\in\Pp_{1,h}(H)$;
  \item[(99)] $\K\langle g,h,x,y\rangle/(g^2-1,h^2-1,[g,h],[g,x]-g(1-g),[g,y],[h,x]-h(1-g),[h,y],x^2-x,y^2,[x,y]+y)$, with $g,h\in\G(H)$, $x\in\Pp_{1,g}(H)$ and $y\in\Pp_{1,h}(H)$;
  \item[(100)] $\K\langle g,h,x,y\rangle/(g^2-1,h^2-1,[g,h],[g,x]-g(1-g),[g,y],[h,x],[h,y]-h(1-h),x^2-x,y^2-y,[x,y])$,  with $g,h\in\G(H)$, $x\in\Pp_{1,g}(H)$ and $y\in\Pp_{1,h}(H)$;
  \item[(101)] $\K\langle g,h,x,y\rangle/(g^2-1,h^2-1,[g,h],[g,x]-g(1-g),[g,y]-g(1-h),[h,x]-h(1-g),[h,y]-h(1-h),x^2-x,y^2-y,[x,y]-x+y)$, with $g,h\in\G(H)$, $x\in\Pp_{1,g}(H)$ and $y\in\Pp_{1,h}(H)$;

  \item[(102)] $\K\langle g,h,x,y\rangle/(g^2-1,h^2-1,[g,h],[g,x],gy-(y+x)g,[h,x],hy-(y+\lambda x)h,[x,y],x^2,y^2)$,  with $g\in\G(H)$, $x,y\in\Pp(H)$;
  \item[(103)] $\K\langle g,h,x,y\rangle/(g^2-1,h^2-1,[g,h],[g,x],gy-(y+x)g,[h,x],hy-(y+\lambda x)h,[x,y],x^2-x,y^2)$, with $g\in\G(H)$, $x,y\in\Pp(H)$;
  \item[(104)] $\K\langle g,h,x,y\rangle/(g^2-1,h^2-1,[g,h],[g,x],gy-(y+x)g,[h,x],hy-(y+\lambda x)h,[x,y],x^2-y,y^2)$, with $g\in\G(H)$, $x,y\in\Pp(H)$;
  \item[(105)] $\K\langle g,h,x,y\rangle/(g^2-1,h^2-1,[g,h],[g,x],gy-(y+x)g,[h,x],hy-(y+\lambda x)h,[x,y],x^2-x,y^2-y)$, with $g\in\G(H)$, $x,y\in\Pp(H)$;
  \item[(106)] $\K\langle g,h,x,y\rangle/(g^2-1,h^2-1,[g,h],[g,x],gy-(y+x)g,[h,x],hy-(y+\lambda x)h,[x,y]-y,x^2-x,y^2)$, $\lambda\in\K$, with $g\in\G(H)$, $x,y\in\Pp(H)$;

  \item[(107)] $\K\langle g,h,x,y\rangle/(g^2-1,h^2-1,[g,h],[g,x],[h,x],gy-(y+x)g,[h,y], [x,y],x^2,y^2)$, with $g\in\G(H)$, $x,y\in\Pp_{1,h}(H)$;
   \item[(108)] $\K\langle g,h,x,y\rangle/(g^2-1,h^2-1,[g,h],[g,x],[h,x],gy-(y+x)g,[h,y]-h(1-h), [x,y]-x,x^2,y^2-y)$,  with $g\in\G(H)$, $x,y\in\Pp_{1,h}(H)$;

  \item[(109)] $\K[\Z_2]\otimes\K[x,y,z]/(x^2,y^2,z^2)$, with $x,y,z\in\Pp(H)$;
  \item[(100)] $\K[\Z_2]\otimes\K[x,y,z]/(x^2-x,y^2-y,z^2-z)$, with $x,y,z\in\Pp(H)$;
  \item[(111)] $\K[\Z_2]\otimes\K[x,y,z]/(x^2-y,y^2-z,z^2)$, with $x,y,z\in\Pp(H)$;
  \item[(112)] $\K[\Z_2]\otimes\K[x,y,z]/(x^2,y^2-z,z^2)$, with $x,y,z\in\Pp(H)$;
  \item[(113)] $\K[\Z_2]\otimes\K[x,y,z]/(x^2,y^2,z^2-z)$, with $x,y,z\in\Pp(H)$;
  \item[(114)] $\K[\Z_2]\otimes\K[x,y,z]/(x^2,y^2-y,z^2-z)$, with $x,y,z\in\Pp(H)$;
  \item[(115)] $\K[\Z_2]\otimes\K[x,y,z]/(x^2-y,y^2,z^2-z)$, with $x,y,z\in\Pp(H)$;
  \item[(116)] $\K[\Z_2]\otimes\K\langle x,y,z\rangle/([x,y]-z,[x,z],[y,z],x^2,y^2,z^2)$, with $x,y,z\in\Pp(H)$;
  \item[(117)] $\K[\Z_2]\otimes\K\langle x,y,z\rangle/([x,y]-z,[x,z],[y,z],x^2,y^2,z^2-z)$, with $x,y,z\in\Pp(H)$;
  \item[(118)] $\K[\Z_2]\otimes\K\langle x,y,z\rangle/([x,y]-y,[x,z],[y,z],x^2-x,y^2,z^2)$, with $x,y,z\in\Pp(H)$;
  \item[(119)] $\K[\Z_2]\otimes\K\langle x,y,z\rangle/([x,y]-y,[x,z],[y,z],x^2-x,y^2-z,z^2)$, with $x,y,z\in\Pp(H)$;
  \item[(120)] $\K[\Z_2]\otimes\K\langle x,y,z\rangle/([x,y]-y,[x,z],[y,z],x^2-x,y^2,z^2-z)$, with $x,y,z\in\Pp(H)$;
  \item[(121)] $\K[\Z_2]\otimes\K\langle x,y,z\rangle/([x,y]-y,[x,z],[y,z],x^2-x,y^2-z,z^2-z)$, with $x,y,z\in\Pp(H)$;
  \item[(122)] $\K[\Z_2]\otimes\K\langle x,y,z\rangle/([x,y],[x,z]=x,[y,z]=y,x^2,y^2,z^2-z)$, with $x,y,z\in\Pp(H)$;

  \item[(123)] $\K[g,x,y,z]/(g^2-1,x^2,y^2,z^2)$, with $g\in\G(H)$ and $x,y,z\in\Pp_{1,g}(H)$;
  \item[(124)] $\K\langle g,x,y,z\rangle/(g^4-1,[g,x]-g(1-g),[g,y],[g,z],[x,y]=y,[x,z]=z,[y,z],x^2-x,y^2,z^2)$, with $g\in\G(H)$ and $x,y,z\in\Pp_{1,g}(H)$;

  \item[(125)] $\K\langle g,x,y,z\rangle/(g^2-1,[g,x],[g,y],[g,z],[x,y]-z,[x,z],[y,z],x^2,y^2,z^2)$, with $g\in\G(H)$, $x,y\in\Pp_{1,g}(H)$ and $z\in\Pp(H)$;
  \item[(126)] $\K\langle g,x,y,z\rangle/(g^2-1,[g,x],[g,y],[g,z],[x,y]-z,[x,z],[y,z],x^2,y^2,z^2-z)$,
  \item[(127)] $\K[g,x,y]/(g^2-1,x^2,y^2)\otimes\K[z]/(z^2)$, with $g\in\G(H)$, $x,y\in\Pp_{1,g}(H)$ and $z\in\Pp(H)$;
  \item[(128)] $\K\langle g,x,y,z\rangle/(g^2-1,[g,x],[g,y],[g,z],[x,y],[x,z]-(1-g),[y,z],x^2,y^2,z^2)$, with $g\in\G(H)$, $x,y\in\Pp_{1,g}(H)$ and $z\in\Pp(H)$;
  \item[(129)] $\K\langle g,x,y,z\rangle/(g^2-1,[g,x],[g,y],[g,z],[x,y],[x,z]-y,[y,z],x^2,y^2,z^2)$,
  \item[(130)] $\K[g,x,y]/(g^2-1,x^2,y^2)\otimes\K[z]/(z^2-z)$, with $g\in\G(H)$, $x,y\in\Pp_{1,g}(H)$ and $z\in\Pp(H)$;
  \item[(131)] $\K\langle g,x,y,z\rangle/(g^2-1,[g,x],[g,y],[g,z],[x,y],[x,z]-x,[y,z],x^2,y^2,z^2-z)$, with $g\in\G(H)$, $x,y\in\Pp_{1,g}(H)$ and $z\in\Pp(H)$;
  \item[(132)] $\K\langle g,x,y,z\rangle/(g^2-1,[g,x],[g,y],[g,z],[x,y],[x,z]-x,[y,z]-y,x^2,y^2,z^2-z)$, with $g\in\G(H)$, $x,y\in\Pp_{1,g}(H)$ and $z\in\Pp(H)$;
  \item[(133)] $\K[g,x,y,z]/(g^2-1,x^2-z,y^2,z^2)$, with $g\in\G(H)$, $x,y\in\Pp_{1,g}(H)$ and $z\in\Pp(H)$;
  \item[(134)] $\K[g,x,y,z]/(g^2-1,x^2-z,y^2,z^2-z)$, with $g\in\G(H)$, $x,y\in\Pp_{1,g}(H)$ and $z\in\Pp(H)$;
  \item[(135)] $\K\langle g,x,y,z\rangle/(g^2-1,[g,x],[g,y],[g,z],[x,y]-z,[y,z],[x,z],x^2-z,y^2,z^2)$, with $g\in\G(H)$, $x,y\in\Pp_{1,g}(H)$ and $z\in\Pp(H)$;
  \item[(136)]  $\K\langle g,x,y,z\rangle/(g^2-1,[g,x],[g,y],[g,z],[x,y]-z,[y,z],[x,z],x^2-z,y^2,z^2-z)$, with $g\in\G(H)$, $x,y\in\Pp_{1,g}(H)$ and $z\in\Pp(H)$;
  \item[(137)] $\K\langle g,x,y,z\rangle/(g^2-1,[g,x]-g(1-g),[g,y],[g,z],[x,y]-y-z,[x,z],[y,z],x^2-x,y^2,z^2)$, with $g\in\G(H)$, $x,y\in\Pp_{1,g}(H)$ and $z\in\Pp(H)$;
  \item[(138)] $\K\langle g,x,y,z\rangle/(g^2-1,[g,x]-g(1-g),[g,y],[g,z],[x,y]-y-z,[x,z],[y,z],x^2-x,y^2,z^2-z)$, with $g\in\G(H)$, $x,y\in\Pp_{1,g}(H)$ and $z\in\Pp(H)$;
  \item[(139)] $\K\langle g,x,y\rangle/(g^2-1,[g,x]-g(1-g),[g,y], [x,y]-y,x^2-x,y^2)\otimes\K[z]/(z^2)$, with $g\in\G(H)$, $x,y\in\Pp_{1,g}(H)$ and $z\in\Pp(H)$;
  \item[(140)] $\K\langle g,x,y,z\rangle/(g^2-1,[g,x]-g(1-g),[g,y],[g,z],[x,y]-y,[x,z],[y,z]-(1-g),x^2-x,y^2,z^2)$, with $g\in\G(H)$, $x,y\in\Pp_{1,g}(H)$ and $z\in\Pp(H)$;
  \item[(141)] $\K\langle g,x,y,z\rangle/(g^2-1,[g,x]-g(1-g),[g,y],[g,z],[x,y]-y,[x,z]-(1-g),[y,z],x^2-x,y^2,z^2)$, with $g\in\G(H)$, $x,y\in\Pp_{1,g}(H)$ and $z\in\Pp(H)$;
  \item[(142)] $\K\langle g,x,y,z\rangle/(g^2-1,[g,x]-g(1-g),[g,y],[g,z],[x,y]-y,[x,z]-y,[y,z],x^2-x,y^2,z^2)$, with $g\in\G(H)$, $x,y\in\Pp_{1,g}(H)$ and $z\in\Pp(H)$;
  \item[(143)] $\K\langle g,x,y,z\rangle/(g^2-1,[g,x]-g(1-g),[g,y],[g,z],[x,y]-y,[x,z],[y,z],x^2-x-z,y^2,z^2)$, with $g\in\G(H)$, $x,y\in\Pp_{1,g}(H)$ and $z\in\Pp(H)$;
  \item[(144)] $\K\langle g,x,y,z\rangle/(g^2-1,[g,x]-g(1-g),[g,y],[g,z],[x,y]-y,[x,z],[y,z]-y,x^2-x,y^2,z^2-z)$, with $g\in\G(H)$, $x,y\in\Pp_{1,g}(H)$ and $z\in\Pp(H)$;
  \item[(145)] $\HH_{14}(\lambda):=\K\langle g,x,y,z\rangle/(g^2-1,[g,x]-g(1-g),[g,y],[g,z],[x,y]-y,[x,z],[y,z],x^2-x-\lambda  z,y^2,z^2- z)$, with $g\in\G(H)$, $x,y\in\Pp_{1,g}(H)$ and $z\in\Pp(H)$;
  \item[(146)] $\K\langle g,x,y,z\rangle/(g^2-1,[g,x]-g(1-g),[g,y],[g,z],[x,y]-y,[x,z],[y,z],x^2-x,y^2-z,z^2)$, with $g\in\G(H)$, $x,y\in\Pp_{1,g}(H)$ and $z\in\Pp(H)$;
  \item[(147)] $\K\langle g,x,y,z\rangle/(g^2-1,[g,x]-g(1-g),[g,y],[g,z],[x,y]-y-z,[x,z],[y,z],x^2-x,y^2-z,z^2)$, with $g\in\G(H)$, $x,y\in\Pp_{1,g}(H)$ and $z\in\Pp(H)$;
  \item[(148)] $\HH_{15}(\lambda):=\K\langle g,x,y,z\rangle/(g^2-1,[g,x]-g(1-g),[g,y],[g,z],[x,y]-y-\lambda z,[x,z],[y,z],x^2-x,y^2-z,z^2-z)$,
with $g\in\G(H)$, $x,y\in\Pp_{1,g}(H)$ and $z\in\Pp(H)$; Moreover,
\begin{itemize}
\item $\HH_{14}(\lambda)\cong\HH_{14}(\gamma)$ or $\HH_{15}(\lambda)\cong\HH_{15}(\gamma)$, if and only if, $\lambda=\gamma$;
\end{itemize}

  \item[(149)] $\HH_{16}(\lambda):=\K\langle g,x,y,z\rangle/(g^2-1,[g,x],[g,y],[g,z],[x,y]-\lambda x,[x,z],[y,z]-z,x^p,y^2-y,z^2)$, with $g\in\G(H)$£¬ $x\in\Pp_{1,g}(H)$ and $y,z\in\Pp(H)$;
  \item[(150)] $\K\langle g,x,y,z\rangle/(g^2-1,[g,x],[g,y],[g,z],[x,y]-x,[x,z]-(1-g),[y,z]-z,x^2,y^2-y,z^2)$, with $g\in\G(H)$£¬ $x\in\Pp_{1,g}(H)$ and $y,z\in\Pp(H)$;
  \item[(151)] $\K\langle g,x\rangle/(g^2-1,[g,x]-g(1-g),x^2-x)\otimes \K\langle y,z\rangle/(y^2-y,z^2,[y,z]-z)$, with $g\in\G(H)$£¬ $x\in\Pp_{1,g}(H)$ and $y,z\in\Pp(H)$;
  \item[(152)] $\K[ g,x]/(g^2-1, x^2)\otimes \K[y,z]/(y^2-y,z^2-z)$, with $g\in\G(H)$£¬ $x\in\Pp_{1,g}(H)$ and $y,z\in\Pp(H)$;
  \item[(153)]  $\HH_{17}(\lambda):=\K[ g,x,y,z]/(g^2-1, x^2-y-\lambda z,y^2-y,z^2-z)$,
  \item[(154)]  $\K\langle g,x,y,z\rangle/(g^2-1,[g,x],[g,y],[g,z],[x,y]-x,[x,z],[y,z],x^2,y^2-y,z^2-z)$, with $g\in\G(H)$£¬ $x\in\Pp_{1,g}(H)$ and $y,z\in\Pp(H)$;
  \item[(155)]  $\K\langle g,x,y,z\rangle/(g^2-1,[g,x],[g,y],[g,z],[x,y]-x,[x,z],[y,z],x^2- z,y^2-y,z^2-z)$, with $g\in\G(H)$£¬ $x\in\Pp_{1,g}(H)$ and $y,z\in\Pp(H)$;
  \item[(156)]  $\HH_{18}(\lambda,\gamma):=\K\langle g,x,y,z\rangle/(g^2-1,[g,x]-g(1-g),[g,y],[g,z],[x,y],[x,z],[y,z],x^2-x-\lambda y-\gamma z,y^2-y,z^2-z)$, with $g\in\G(H)$£¬ $x\in\Pp_{1,g}(H)$ and $y,z\in\Pp(H)$;
  \item[(157)] $\K[g,x]/(g^2-1,x^2)\otimes\K[y,z]/(y^2-y,z^2)$, with $g\in\G(H)$£¬ $x\in\Pp_{1,g}(H)$ and $y,z\in\Pp(H)$;
  \item[(158] $\K[g,x,y,z]/(g^2-1,x^2-z,y^2-y,z^2)$, with $g\in\G(H)$£¬ $x\in\Pp_{1,g}(H)$ and $y,z\in\Pp(H)$;
  \item[(159)] $\K[g,x,y,z]/(g^2-1,x^2-y,y^2-y,z^2)$, with $g\in\G(H)$£¬ $x\in\Pp_{1,g}(H)$ and $y,z\in\Pp(H)$;
  \item[(160)] $\K[g,x,y,z]/(g^2-1,x^2-y-z,y^2-y,z^2)$, with $g\in\G(H)$£¬ $x\in\Pp_{1,g}(H)$ and $y,z\in\Pp(H)$;
  \item[(161)] $\K\langle g,x,y,z\rangle/(g^2-1,[g,x],[g,y],[g,z],[x,y],[x,z]-(1-g),[y,z],x^2, y^2-y,z^2)$, with $g\in\G(H)$£¬ $x\in\Pp_{1,g}(H)$ and $y,z\in\Pp(H)$;
  \item[(162)] $\K\langle g,x,y,z\rangle/(g^2-1,[g,x],[g,y],[g,z],[x,y],[x,z]-(1-g),[y,z],x^2- y, y^2-y,z^2)$, with $g\in\G(H)$£¬ $x\in\Pp_{1,g}(H)$ and $y,z\in\Pp(H)$;
  \item[(163)] $\K\langle g,x,y,z\rangle/(g^2-1,[g,x],[g,y],[g,z],[x,y]-x,[x,z],[y,z],x^2, y^2-y,z^2)$, with $g\in\G(H)$£¬ $x\in\Pp_{1,g}(H)$ and $y,z\in\Pp(H)$;
  \item[(164] $\K\langle g,x,y,z\rangle/(g^2-1,[g,x],[g,y],[g,z],[x,y]-x,[x,z],[y,z],x^2-z, y^2-y,z^2)$, with $g\in\G(H)$£¬ $x\in\Pp_{1,g}(H)$ and $y,z\in\Pp(H)$;
  \item[(165)]  $\HH_{19}(\lambda,i):=\K\langle g,x,y,z\rangle/(g^2-1,[g,x]-g(1-g),[g,y],[g,z],[x,y],[x,z],[y,z],x^2-x-\lambda y-iz,y^2-y,z^2)$, for $i\in\I_{0,1}$, with $g\in\G(H)$£¬ $x\in\Pp_{1,g}(H)$ and $y,z\in\Pp(H)$;
  \item[(166)]  $\HH_{20}(\lambda):=\K\langle g,x,y,z\rangle/(g^2-1,[g,x]-g(1-g),[g,y],[g,z],[x,y],[x,z]-(1-g),[y,z],x^2-x-\lambda y, y^2-y,z^2)$, with $g\in\G(H)$£¬ $x\in\Pp_{1,g}(H)$ and $y,z\in\Pp(H)$;
  \item[(167)] $\K[g,x]/(g^2-1,x^2)\otimes\K[y,z]/(y^2-z,z^2)$, with $g\in\G(H)$£¬ $x\in\Pp_{1,g}(H)$ and $y,z\in\Pp(H)$;
  \item[(168)] $\K[g,x,y,z]/(g^2-1,x^2-z,y^2-z,z^2)$, with $g\in\G(H)$£¬ $x\in\Pp_{1,g}(H)$ and $y,z\in\Pp(H)$;
  \item[(169)] $\K[g,x,y,z]/(g^2-1,x^2-y,y^2-z,z^2)$, with $g\in\G(H)$£¬ $x\in\Pp_{1,g}(H)$ and $y,z\in\Pp(H)$;
  \item[(170)] $\K\langle g,x,y,z\rangle/(g^2-1,[g,x],[g,y],[g,z],[x,y]-(1-g),[y,z],[x,z],x^2,y^2-z,z^2)$, with $g\in\G(H)$£¬ $x\in\Pp_{1,g}(H)$ and $y,z\in\Pp(H)$;
  \item[(171)] $\K\langle g,x,y,z\rangle/(g^2-1,[g,x],[g,y],[g,z],[x,y]-(1-g),[y,z],[x,z],x^2- z,y^2-z,z^2)$, with $g\in\G(H)$£¬ $x\in\Pp_{1,g}(H)$ and $y,z\in\Pp(H)$;
  \item[(172)] $\K\langle g,x\rangle/(g^2-1,[g,x]-g(1-g),x^2-x)\otimes\K[y,z]/(y^2-z,z^2)$, with $g\in\G(H)$£¬ $x\in\Pp_{1,g}(H)$ and $y,z\in\Pp(H)$;
  \item[(173)] $\K\langle g,x,y,z\rangle/(g^2-1,[g,x]-g(1-g),[g,y],[g,z],[x,y],[x,z],[y,z],x^2-x-z,y^2-z,z^2)$, with $g\in\G(H)$£¬ $x\in\Pp_{1,g}(H)$ and $y,z\in\Pp(H)$;
  \item[(174)] $\K\langle g,x,y,z\rangle/(g^2-1,[g,x]-g(1-g),[g,y],[g,z],[x,y],[x,z],[y,z],x^2-x-y,y^2-z,z^2)$,
  \item[(175)]  $\HH_{21}(\lambda):=\K\langle g,x,y,z\rangle/(g^2-1,gx-xg-g(1-g),[g,y],[g,z],[x,y]-(1-g),[x,z],[y,z],x^2-x-\lambda z,y^2-z,z^2)$, with $g\in\G(H)$£¬ $x\in\Pp_{1,g}(H)$ and $y,z\in\Pp(H)$;
  \item[(176)] $\K[g,x]/(g^2-1,x^2)\otimes\K[y,z]/(y^2,z^2)$, with $g\in\G(H)$£¬ $x\in\Pp_{1,g}(H)$ and $y,z\in\Pp(H)$;
  \item[(177)] $\K\langle g,x\rangle/(g^2-1,[g,x]-g(1-g),x^2-x)\otimes\K[y,z]/(y^2,z^2)$, with $g\in\G(H)$£¬ $x\in\Pp_{1,g}(H)$ and $y,z\in\Pp(H)$;
  \item[(178)] $\K[g,x,y,z]/(g^2-1,x^2-y, y^2,z^2)$, with $g\in\G(H)$£¬ $x\in\Pp_{1,g}(H)$ and $y,z\in\Pp(H)$;
  \item[(179)] $\K\langle g,x,y,z\rangle/(g^2-1,[g,x]-g(1-g),[g,y],[g,z],[x,y],[x,z],[y,z],x^2-x-y,y^2,z^2)$, with $g\in\G(H)$£¬ $x\in\Pp_{1,g}(H)$ and $y,z\in\Pp(H)$;
  \item[(180)] $\K\langle g,x,y,z\rangle/(g^2-1,[g,x],[g,y],[g,z],[x,y]-(1-g),[x,z],[y,z],x^2,y^2,z^2)$, with $g\in\G(H)$£¬ $x\in\Pp_{1,g}(H)$ and $y,z\in\Pp(H)$;
  \item[(181)] $\K\langle g,x,y,z\rangle/(g^2-1,[g,x],[g,y],[g,z],[x,y]-(1-g),[x,z],[y,z],x^2-z,y^2,z^2)$, with $g\in\G(H)$£¬ $x\in\Pp_{1,g}(H)$ and $y,z\in\Pp(H)$;
  \item[(182)] $\K\langle g,x,y,z\rangle/(g^2-1,[g,x]-g(1-g),[g,y],[g,z],[x,y]-(1-g),[x,z],[y,z],x^2-x,y^2,z^2)$, with $g\in\G(H)$£¬ $x\in\Pp_{1,g}(H)$ and $y,z\in\Pp(H)$;
  \item[(183)] $\K\langle g,x,y,z\rangle/(g^2-1,[g,x]-g(1-g),[g,y],[g,z],[x,y]-(1-g),[x,z],[y,z],x^2-x-z,y^2,z^2)$,
with $g\in\G(H)$£¬ $x\in\Pp_{1,g}(H)$ and $y,z\in\Pp(H)$; Moreover,
\begin{itemize}
\item $\HH_{16}(\lambda)\cong \HH_{16}(\gamma)$ for $\lambda,\gamma\in\K$, if and only if, $\lambda=\gamma$;
\item $\HH_{17}(\lambda)\cong \HH_{17}(\gamma)$ for $\lambda,\gamma\in\K$, if and only if, there exist $\alpha_1,\alpha_2,\beta_1,\beta_2\in\K$ satisfying $\alpha_i^2-\alpha_i=0=\beta_i^2-\beta_i$ for $i\in\I_{1,2}$ such that $(\alpha_1+\beta_1\lambda)\gamma=(\alpha_2+\beta_2\lambda)$ and $\alpha_1\beta_2-\alpha_2\beta_1\neq 0$;
\item $\HH_{18}(\lambda,\gamma)\cong \HH_{18}(\mu,\nu)$ if and only if, there exist $\alpha_i,\beta_i\in\K$ satisfying $\alpha_i^2-\alpha_i=0=\beta_i^2-\beta_i$ for $i\in\I_{1,2}$ such that $\alpha_1\beta_2-\alpha_2\beta_1\neq 0$ and $\lambda\alpha_1+\gamma\beta_1=\mu$,  $\lambda\alpha_2+\gamma\beta_2=\nu$;
\item $\HH_{19}(\lambda,i)\cong \HH_{19}(\gamma,j)$ if and only if $\lambda =\gamma$ and $i=j$;
 \item $\HH_{20}(\lambda)\cong \HH_{20}(\gamma)$ or $\HH_{21}(\lambda)= \HH_{21}(\gamma)$, if and only if, $\lambda=\gamma$;
\end{itemize}
 \item[(184)] $\K\langle g,x,y,z\rangle/(g^2-1,[g,x],[g,y],gz-(z+y)g,[x,y],[x,z],[y,z],x^2,y^2,z^2)$, with $g\in\G(H)$, $x,y\in\Pp(H)$;
  \item[(185)] $\K\langle g,x,y,z\rangle/(g^2-1,[g,x],[g,y],gz-(z+y)g,[x,y],[x,z],[y,z],x^2-x,y^2-y,z^2-z)$, with $g\in\G(H)$, $x,y\in\Pp(H)$;
  \item[(186)]$\K\langle g,x,y,z\rangle/(g^2-1,[g,x],[g,y],gz-(z+y)g,[x,y],[x,z],[y,z],x^2-y,y^2-z,z^2)$, with $g\in\G(H)$, $x,y\in\Pp(H)$;
  \item[(187)] $\K\langle g,x,y,z\rangle/(g^2-1,[g,x],[g,y],gz-(z+y)g,[x,y],[x,z],[y,z],x^2,y^2-z,z^2)$, with $g\in\G(H)$, $x,y\in\Pp(H)$;
  \item[(188)] $\K\langle g,x,y,z\rangle/(g^2-1,[g,x],[g,y],gz-(z+y)g,[x,y],[x,z],[y,z],x^2,y^2,z^2-z)$, with $g\in\G(H)$, $x,y\in\Pp(H)$;
  \item[(189)] $\K\langle g,x,y,z\rangle/(g^2-1,[g,x],[g,y],gz-(z+y)g,[x,y],[x,z],[y,z],x^2,y^2-y,z^2-z)$,
      with $g\in\G(H)$, $x,y\in\Pp(H)$;
  \item[(190)] $\K\langle g,x,y,z\rangle/(g^2-1,[g,x],[g,y],gz-(z+y)g,[x,y],[x,z],[y,z],x^2-y,y^2,z^2-z)$,
      with $g\in\G(H)$, $x,y\in\Pp(H)$;
  \item[(191)] $\K\langle g,x,y,z\rangle/(g^2-1,[g,x],[g,y],gz-(z+y)g,[x,y]-z,[x,z],[y,z],x^2,y^2,z^2)$, with $g\in\G(H)$, $x,y\in\Pp(H)$;
  \item[(192)] $\K\langle g,x,y,z\rangle/(g^2-1,[g,x],[g,y],gz-(z+y)g,[x,y]-z,[x,z],[y,z],x^2,y^2,z^2-z)$, with $g\in\G(H)$, $x,y\in\Pp(H)$;
  \item[(193)] $\K\langle g,x,y,z\rangle/(g^2-1,[g,x],[g,y],gz-(z+y)g,[x,y]-y,[x,z],[y,z],x^2-x,y^2,z^2)$, with $g\in\G(H)$, $x,y\in\Pp(H)$;
  \item[(194)] $\K\langle g,x,y,z\rangle/(g^2-1,[g,x],[g,y],gz-(z+y)g,[x,y]-y,[x,z],[y,z],x^2-x,y^2-z,z^2)$, with $g\in\G(H)$, $x,y\in\Pp(H)$;
  \item[(195)] $\K\langle g,x,y,z\rangle/(g^2-1,[g,x],[g,y],gz-(z+y)g,[x,y]-y,[x,z],[y,z],x^2-x,y^2,z^2-z)$, with $g\in\G(H)$, $x,y\in\Pp(H)$;
  \item[(196)] $\K\langle g,x,y,z\rangle/(g^2-1,[g,x],[g,y],gz-(z+y)g,[x,y]-y,[x,z],[y,z],x^2-x,y^2-z,z^2-z)$, with $g\in\G(H)$, $x,y\in\Pp(H)$;
  \item[(197)] $\K\langle g,x,y,z\rangle/(g^2-1,[g,x],[g,y],gz-(z+y)g,[x,y],[x,z]=x,[y,z]=y,x^2,y^2,z^2-z)$, with $g\in\G(H)$, $x,y\in\Pp(H)$.
\end{description}
\end{thm}

\begin{rmk}
By Theorem \ref{thm:16-diagram-Nichols-algebra}, there are 197 types of non-connected pointed Hopf algebras of dimension $16$ with $\Char\K=2$ whose diagrams are Nichols algebras. Up to isomorphism, there are infinitely many classes of such Hopf algebras.  In particular, we obtain infinitely many new examples of non-commutative   non-cocommutative pointed Hopf algebras.
\end{rmk}

Let $H$ be a non-trivial non-connected pointed Hopf algebra of dimension $16$. By Lemma \ref{lem:16-group-like}, $\G(H)$ is isomorphic to   $D_4$,   $Q_8$, $\Z_8$, $\Z_4\times \Z_2$, $\Z_2\times \Z_2\times \Z_2$, $\Z_4$, $\Z_2\times \Z_2$ or $\Z_2$. We will subsequently  prove Theorem \ref{thm:16-diagram-Nichols-algebra} by a case by case discussion. In what follows, $R$ is the diagram of $H$ and $V:=R(1)$. By assumption, $R\cong\BN(V)$.

\subsection{Coradical of dimension $8$}
Observe that  $\dim H_0=8$. Then $\dim R=2$. By Proposition \ref{pro-P-duality}, $\dim\BN(V)>2$ if $\dim V>1$. Hence $\dim V=1$ with a basis $\{x\}$ satisfying $c(x\otimes x)=x\otimes x$. Consequently, $R\cong\K[x]/(x^2)$.

\subsubsection{$\G(H)\cong D_4$.} Observe that  $\widehat{\G(H)}=\{\epsilon\}$ and $Z(D_4)=\{1,g^2\}$. Then by Remark \ref{rmk:dimV=1}, $x\in V_{g^{2\mu}}^{\epsilon}$ for $\mu\in\I_{0,1}$. Therefore,
\begin{gather*}
\gr H=\K\langle g,h,x\mid g^4=h^2=1,hg=g^3h, gx=xg,hx=xh,x^2=0\rangle,
\end{gather*}
with $g,h\in\G(H)$ and $x\in\Pp_{1,g^{2\mu}}(H)$ for $\mu\in\I_{0,1}$. Now we determine the liftings of $\gr H$.

By similar computations as before, we have
\begin{align*}
gx-xg=\lambda_1g(1-g^{2\mu}),\quad hx-xh=\lambda_2h(1-g^{2\mu}),\quad x^2=(\mu+1)\lambda_3x,
\end{align*}
for some $\lambda_1\in\I_{0,1}$, $\lambda_2,\lambda_2\in\K$.

If $\mu=0$, then $gx-xg=0=hx-xh$ and $x^2=\lambda_3x$ in $H$. By rescaling $x$, we can take $\lambda_3\in\I_{0,1}$, which gives two classes of $H$ described in $(1)$--$(2)$.   Clearly, they are  non-isomorphic.

If $\mu=1$, then $x^2=0$ in $H$. Applying the Diamond Lemma \cite{B} to show that $\dim H=16$, it suffices to show that the following ambiguities
\begin{align*}
(g^4)x=g^3(gx),\quad (h^2)x=h(hx),\quad (gh)x=g(hx),
\end{align*}
are resolvable with the order $x<h<g$. By Lemma \ref{pqlem1}, we have $[g^4,x]=0=[h^2,x]$ and hence the first two ambiguities are resolvable. Now we show that the ambiguity $(gh)x=g(hx)$ is resolvable:
\begin{align*}
g(hx)&=g(xh+\lambda_2h(1-g^2))=(gx)h+\lambda_2gh(1-g^2)=xhg^3+(\lambda_1+\lambda_2)hg^3(1-g^2),\\
&=(xh+\lambda_2h(1-g^2))g^3+\lambda_1hg^3(1-g^2)=(hx)g^3+\lambda_1hg^3(1-g^2)=(hg^3)x=(gh)x.
\end{align*}

If $\lambda_1=0$, then by rescaling $x$, we can take $\lambda_2\in\I_{0,1}$, which gives two classes of $H$ described in $(3)$--$(4)$.
If $\lambda_1=1$, then $H\cong\HH_1(\lambda_2)$ described in $(5)$.

Now we prove that $\HH_1(\lambda)\cong\HH_1(\gamma)$ for $\lambda,\gamma\in\K$, if and only if, $\lambda=\gamma+i$ for some $i\in\I_{0,1}$.

Observe that $\Aut(D_8)\cong D_8$ with generators $\psi_1,\psi_2$, where
\begin{align*}
\psi_1(g)=g,\quad \psi_1(h)=gh;\quad \psi_2(g)=g^{-1},\quad\psi_2(h)=h.
\end{align*}

Write $g^{\prime},h^{\prime},x^{\prime}$ to distinguish the generators of $\HH_1(\gamma)$. Suppose that $\phi:\HH_1(\lambda)\rightarrow \HH_1(\gamma)$ for $\lambda,\gamma\in\K$ is a Hopf algebra isomorphism. Then by Proposition \ref{pro:R11-4.3.3}, $\phi(\Pp_{1,g^2}(\HH_1(\lambda)))=\Pp_{1,(g^{\prime})^2}(\HH_1(\gamma))$ and $\phi|_{D_8}\in\Aut(D_8)$. Note that spaces of the skew-primitive elements of $\HH_1(\gamma)$ are trivial except $\Pp_{1,(g^{\prime})^2}(\HH_1(\gamma))=\K\{x^{\prime}\}\oplus \K\{1-(g^{\prime})^2\}$. Therefore,
\begin{align}\label{eq:HHHH-1}
\phi(g)\in\{g^{\prime},(g^{\prime})^3\},\quad \phi(h)=(g^{\prime})^ih^{\prime},\quad i\in\I_{0,3}
\end{align}
\begin{align}\label{eq:HHHH-2}
 \phi(x)=a(1-(g^{\prime})^2)+bx^{\prime},\quad \text{for some}\quad a,b\neq 0\in\K.
\end{align}
Applying $\phi$ to relation $gx-xg=g(1-g^2)$, then
\begin{align*}
\phi(gx-xg-g(1-g^2))&=\phi(g)\phi(x)-\phi(x)\phi(g)-\phi(g)(1-(g^{\prime})^2)\\
&=b\phi(g)x^{\prime}-bx^{\prime}\phi(g)-\phi(g)(1-(g^{\prime})^2)=(b-1)\phi(g)(1-(g^{\prime})^2)=0.
\end{align*}
Therefore, $b=1$. Then applying $\phi$ to the relations $hx-xh=\lambda h(1-g^2)$, then we have
\begin{align*}
\phi(h)x^{\prime}-x^{\prime}\phi(h)-\lambda\phi(h)(1-(g^{\prime})^2)=0.
\end{align*}

If $\phi(h)=(g^{\prime})^{2\mu}h^{\prime}$ for $\mu\in\I_{0,1}$, then $\phi(h)x^{\prime}-x^{\prime}\phi(h)=\gamma\phi(h)(1-(g^{\prime})^2)$ and hence $\gamma=\lambda$. If $\phi(h)=(g^{\prime})^{i}h^{\prime}$ for $i\in\{1,3\}$, then $\phi(h)x^{\prime}-x^{\prime}\phi(h)=(\gamma+1)\phi(h)(1-(g^{\prime})^2)$ and hence $\gamma+1=\lambda$. Consequently, we have
 $$\gamma=\lambda+i,\quad\text{for }i\in\I_{0,1}.$$

Conversely, for any $\lambda\in\K$, $i\in\I_{0,1}$, let $\psi:\HH_1(\lambda)\rightarrow \HH_1(\lambda+i)$ be the algebra map given by
\begin{align*}
\psi(g)=g^{\prime},\quad \psi(h)=(g^{\prime})^ih^{\prime},\quad \psi(x)=x^{\prime}+b(1-(g^{\prime})^2),\quad b\in\K.
\end{align*}
Observe that $\HH_1(\lambda)_1=\K\{g^ih^j,g^ih^jx\}_{i\in\I_{0,3},j\in\I_{0,2}}$.  It is easy to see that it is an epimorphism of Hopf algebras and $\psi|_{(\HH_1(\lambda))_1}$ is injective. By Proposition \ref{pro:R11-4.3.3}, $\psi$ is a Hopf algebra isomorphism.

Similarly, the Hopf algebras from items $(3)$--$(5)$ are pairwise non-isomorphic. Indeed,
there are not  elements $a, b, i\in\K$ such that that the morphism satisfying relations
\eqref{eq:HHHH-1} and \eqref{eq:HHHH-2} is an isomorphism.

\subsubsection{$\G(H)\cong Q_8$.}   Observe that $\widehat{Q_8}=\{\epsilon\}$ and $Z(Q_8)=\{1,g^2\}$. Then by Remark \ref{rmk:dimV=1}, $x\in V_{g^{2\mu}}^{\epsilon}$ for $\mu\in\I_{0,1}$. Therefore,
\begin{align*}
\gr H=\K\langle g,h,x\mid g^4=1,hg=g^3h, g^2=h^2, gx=xg,hx=xh,x^2=0\rangle,
\end{align*}
with $g,h\in\G(H)$ and $x\in\Pp_{1,g^{2\mu}}(H)$. Similar to the case $\G(H)\cong D_4$, the defining relations of $H$ are given by
\begin{gather*}
g^4=1,\quad hg=g^3h,\quad g^2=h^2,\\
gx-xg=\lambda_1g(1-g^{2\mu}),\quad hx-xh=\lambda_2h(1-g^{2\mu}),\quad x^2-\lambda_3x=0,
\end{gather*}
for some $\lambda_1\in\I_{0,1}$, $\lambda_2\in\K$ with the conditions $\lambda_3=0$ if $\mu=1$.

If $\mu=0$, then $gx-xg=0=hx-xh$ in $H$. Observe that $H$ is the tensor product Hopf algebra between $\K[Q_8]$ and $\K[x]/(x^2-\lambda_3x)$.  By rescaling $x$, we can take $\lambda_3\in\I_{0,1}$, which gives two classes of $H$ described in $(6)$--$(7)$.

If $\mu=1$, then it follows by a direct computation that the ambiguities $(g^4)x=g^3(gx)$, $(h^4)x=h^3(hx)$, $(gh)x=g(hx)$,
are resolvable with the order $x<h<g$ and hence $\dim H=16$.  If $\lambda_1=0$, then by rescaling $x$, we can take $\lambda_2\in\I_{0,1}$, which gives two classes of $H$ described in $(8)$--$(9)$. Indeed, if $\lambda_2=1$, then $H\cong\HH_2(0)$  by swapping $g$ and $h$. If $\lambda_1=1$, then $H\cong\HH_2(\lambda_2)$ described in $(9)$.

Now we prove that $\HH_2(\lambda)\cong\HH_2(\gamma)$ for $\lambda,\gamma\in\K $, if and only if, $\lambda=\gamma+i$ or $(\lambda-j)(\gamma-i)=1$ for $i,j\in\I_{0,1}$.

Observe that $\Aut(Q_8)\cong S_4$ with generators $\psi_1,\psi_2,\psi_3$ where
\begin{gather*}
\psi_1(g)=g^{-1},\quad \psi_1(h)=gh;\quad
\psi_2(g)=h,\quad \psi_2(h)=g;\quad
\psi_3(g)=gh,\quad \psi_3(h)=g^2h.
\end{gather*}

 Suppose that $\phi:\HH_2(\lambda)\rightarrow \HH_2(\gamma)$ for $\lambda,\gamma\in\K$ is a Hopf algebra isomorphism. Then $\phi|_{Q_8}:Q_8\rightarrow Q_8$ is an automorphism. Hence $\phi(g)\in\{g,g^3,h,g^2h,gh,g^3h\}$ and $\phi(h)\in\{g,g^3,h,g^2h,gh,g^3h\}-\{\phi(g),\phi(g^{-1})\}$. Write $g^{\prime},h^{\prime},x^{\prime}$ to distinguish the generators of $\HH_2(\gamma)$.
Since spaces of skew-primitive elements of $\HH_2(\gamma)$ are trivial except $\Pp_{1,(g^{\prime})^2}(\HH_2(\gamma))=\K\{x^{\prime}\}\oplus \K\{1-(g^{\prime})^2\}$, $\phi(x)=a(1-(g^{\prime})^2)+bx^{\prime}$ for some $a,b\neq 0\in\K$.

If $\phi(g)=(g^{\prime})^{2\mu}g^{\prime}$ for $\mu\in\I_{0,1}$, then $\phi(h)=(g^{\prime})^{2\nu}(g^{\prime})^ih^{\prime}$ for $i,\nu\in\I_{0,1}$. Applying $\phi$ to  the relations $gx-xg=g(1-g^2), hx-xh=\lambda h(1-g^2)$, we have
$$a=1,\quad \lambda=\gamma+i.$$

If $\phi(g)=(g^{\prime})^{2\mu}h^{\prime}$ for $\mu\in\I_{0,1}$, then $\phi(h)=(g^{\prime})^{2\nu}g^{\prime}(h^{\prime})^i$ for $i,\nu\in\I_{0,1}$. Applying $\phi$ to the relations $gx-xg=g(1-g^2), hx-xh=\lambda h(1-g^2)$, we have
$$a\gamma=1,\  a(1+i\gamma)=\lambda\quad \Rightarrow\quad (\lambda-i)\gamma=1.$$

If $\phi(g)=(g^{\prime})^{2\mu}g^{\prime}h^{\prime}$ for $\mu\in\I_{0,1}$, then $\phi(h)=(g^{\prime})^{2\nu}(g^{\prime})^i(h^{\prime})^j$ for $i,j,\nu\in\I_{0,1}$ satisfying $i+j=1$.
Applying $\phi$ to the relations $gx-xg=g(1-g^2), hx-xh=\lambda h(1-g^2)$, we have
$$   a(1+\gamma)=1,~a(i+j\gamma)=\lambda\quad \Rightarrow (\lambda-j)(\gamma+1)=1.$$

Conversely, if $\lambda=\gamma+i$ or $(\lambda-j)(\gamma-i)=1$ for $i,j\in\I_{0,1}$, then we can build an algebra map $\psi:\HH_2(\lambda)\rightarrow\HH_2(\gamma)$ in the form of $\phi$. It is easy to see that $\psi$ is a Hopf algebra epimorphism and $\psi|_{(\HH_2(\lambda))_1}$ is injective, which implies that $\psi$ is a Hopf algebra isomorphism.

Similarly, one can show that the Hopf algebras from items $(6)$--$(9)$ are pairwise non-isomorphic.

\subsubsection {$\G(H)\cong \Z_8$}  Then $\widehat{\Z_8}=\{\epsilon\}$ and $Z(\Z_8)=\Z_8:=\langle g\rangle$. Then by Remark \ref{rmk:dimV=1}, $x\in V_{g^{\mu}}^{\epsilon}$ for $\mu\in\I_{0,7}$. By changing the generator of $\Z_{8}$, we can take $\mu\in\{0,1,2,4\}$. Therefore,
\begin{gather*}
\gr H=\K\langle g, x\mid g^8=1,  gx=xg, x^2=0\rangle,
\end{gather*}
 with $g \in\G(H)$ and $x\in\Pp_{1,g^{\mu}}(H)$ for $\mu\in\{0,1,2,4\}$.  Then by a similar computation as before, we have
\begin{align*}
gx-xg=\lambda_1g(1-g^{\mu}),\quad \lambda_1\in\I_{0,1}.
\end{align*}
By induction, we have $[g^{\mu},x]=\mu\lambda_1g^{\mu}(1-g^{\mu})$. Then
\begin{gather*}
\Delta(x^2)=x^2\otimes 1+[g^{\mu},x]\otimes x+g^{2\mu}\otimes x^2=x^2\otimes 1+\mu\lambda_1g^{\mu}(1-g^{\mu})\otimes x+g^{2\mu}\otimes x^2,\\
\Delta(x^2-\mu\lambda_1x)=(x^2-\mu\lambda_1x)\otimes 1+g^{2\mu}\otimes(x^2-\mu\lambda_1x).
\end{gather*}

If $\mu=0$, then $gx-xg=0$ in $H$  and $\Pp(H)=\K\{x\}$. Hence $x^2=\lambda_2x$ for $\lambda_2\in\K$. Observe that $H\cong\K[\Z_8]\otimes\K[x]/(x^2-\lambda_2x)$. Then $\dim H=16$. By rescaling $x$, we can take $\lambda_2\in\I_{0,1}$, which gives two classes of $H$ described in $(10)$--$(11)$. Clearly, they are non-isomorphic.

If $\mu\neq 0$, then $\Pp_{1,g^{2\mu}}=\K\{1-g^{2\mu}\}$ and hence $x^2-\mu\lambda_1x=\lambda_3(1-g^{2\mu})$ for $\lambda_3\in\K$. Then we take $\lambda_3=0$ via the linear translation $x\mapsto x-a(1-g^{\mu})$ satisfying $a^2-\mu\lambda_1a=\lambda_3$. Indeed, it is easy to see that the linear translation is a Hopf algebra isomorphism. By Lemma \ref{pqlem1}, we have $[g,x^2]=0$, which implies that the ambiguity $(g^2)x=g(gx)$ is resolvable. By Proposition \ref{proJ},
\begin{align*}
[g,x^2]=[[g,x],x]=\lambda_1g(1-g^{\mu})-\lambda_1(\mu+1)g^{\mu+1}(1-g^{\mu}).
\end{align*}
Hence the ambiguity $g(x^2)=(gx)x$ imposes the condition $\lambda_1=0$ if $\mu=2$.
Then by Diamond Lemma, $\dim H=16$ with the condition: $\lambda_1=0$ if $\mu=2$.

If $\lambda_1=0$, then $H$ is the Hopf algebra described in $(12)$. If $\lambda_1=1$, then $\mu\in\{1,4\}$ and $H$ is the Hopf algebra described in $(13)$.    Obviously, the two Hopf algebras with $\mu=1$ and $\mu=4$ are non-isomorphic since they are not isomorphic as coalgebras.

\subsubsection{$\G(H)\cong \Z_4\times \Z_2=\langle g\rangle\times \langle h\rangle$.}  Then $\widehat{\Z_4\times \Z_2}=\{\epsilon\}$ and $Z(\Z_4\times \Z_2)=\Z_4\times \Z_2 $. Then by Remark \ref{rmk:dimV=1}, $x\in V_{g^{\mu}h^{\nu}}^{\epsilon}$ for $\mu\in\I_{0,3},\nu\in\I_{0,1}$.  Therefore,
\begin{gather*}
\gr H=\K\langle g,h, x\mid g^4=1,h^2=1,gh=gh,  gx=xg, x^2=0\rangle,
\end{gather*}
with $g \in\G(H)$ and $x\in\Pp_{1,g^{\mu}h^{\nu}}(H)$.

Observe that $\Aut(\Z_4\times \Z_2)\cong D_4$ with generators $\psi_1,\psi_2$, where
\begin{gather*}
\psi_1(g)=gh,\quad \psi_1(h)=g^2h;\quad \psi_2(g)=gh,\quad \psi_2(h)=h.
\end{gather*}
Then up to isomorphism,  we can take $(\mu,\nu)\in\{(0,0),(1,0),(2,0),(0,1)\}$. By similar computations as before, we have
\begin{align*}
gx-xg=\lambda_1g(1-g^{\mu}h^{\nu}),\quad hx-xh=\lambda_2h(1-g^{\mu}h^{\nu}),
\end{align*}
for some $\lambda_1,\lambda_2\in\K$. Then
\begin{align*}
\Delta(x^2)&=(x\otimes 1+g^{\mu}h^{\nu}\otimes x)^2=x^2\otimes 1+[g^{\mu}h^{\nu},x]\otimes x+g^{2\mu}\otimes x^2\\
&=x^2\otimes 1+(\mu\lambda_1+\nu\lambda_2)(g^{\mu}h^{\nu}-g^{2\mu})\otimes x+g^{2\mu}\otimes x^2.
\end{align*}
It is easy to see that $x^2-(\mu\lambda_1+\nu\lambda_2)x\in\Pp_{1,g^{2\mu}}(H)$.

If $(\mu,\nu)=(0,0)$, then $gx=xg,hx=xh$ in $H$ and $\Pp(H)=\K\{x\}$, which implies that $x^2=\lambda_3x$ for $\lambda_3\in\I_{0,1}$. In this case, $H\cong\K[\Z_4\times \Z_2]\otimes\K[x]/(x^2-\lambda_3x)$, which are described in $(14)$--$(15)$.

If $(\mu,\nu)\in\{(1,0),(2,0)\}$, then $\Pp_{1,g^{2\mu}}(H)=\K\{1-g^{2\mu}\}$ and hence $x^2-\mu\lambda_1x=\lambda_3(1-g^{2\mu})$ for some $\lambda_3\in\K$. We can take $\lambda_3=0$ via the linear translation $x\mapsto x-a(1-g^{\mu})$ satisfying $a^2-\mu\lambda_1\mu=\lambda_3$.
Similar to the case $\G(H)\cong \Z_8$, it follows by a direct computation that the ambiguities $(g^4)x=g^3(gx)$, $(h^2)x=h(hx)$, $g(x^2)=(gx)x$, $h(x^2)=(hx)x$ and $(gh)x=g(hx)$ are resolvable. Then by Diamond lemma, $\dim H=16$. By rescaling $x$, we can take $\lambda_1\in\I_{0,1}$. If $\lambda_1=0$, then by rescaling $x$, $\lambda_2\in\I_{0,1}$, which gives two classes of $H$ described in $(16)$--$(17)$. If $\lambda_1=1$, then $H\cong\HH_{3,\mu}(\lambda_2)$ described in $(18)$.  Obviously, $\HH_{3,1}(\lambda)$ and $\HH_{3,2}(\gamma)$ for any $\lambda,\gamma\in\K$ are non-isomorphic since their coalgebra structure are not isomorphic.

We claim that $\HH_{3,1}(\lambda)\cong\HH_{3,1}(\gamma)$, if and only if, $\lambda=\gamma$; $\HH_{3,2}(\lambda)\cong\HH_{3,2}(\gamma)$, if and only if, $\lambda=\gamma$ or $\lambda\gamma=\lambda+\gamma$.

Suppose that $\phi:\HH_{3,1}(\lambda)\rightarrow \HH_{3,1}(\gamma)$ for $\lambda,\gamma\in\K$ is a Hopf algebra isomorphism. Then $\phi|_{\Z_4\times \Z_2}:\Z_4\times \Z_2\rightarrow \Z_4\times \Z_2$ is an automorphism. Therefore, $\phi(g)\in\{g,g^3,gh,g^3h\}$ and $\phi(h)\in\{h,g^2h\}$.  Write $g^{\prime},h^{\prime},x^{\prime}$ to distinguish the generators of $\HH_{3,1}(\gamma)$. Since spaces of skew-primitive elements of $\HH_{3,1}(\gamma)$ are trivial except $\Pp_{1,g^{\prime}}(\HH_{3,1}(\gamma))=\K\{x^{\prime}\}\oplus \K\{1-g^{\prime}\}$, it follows that $$\phi(g)=g^{\prime},\quad\phi(x)=a(1- g^{\prime} )+bx^{\prime}$$ for some $a,b\neq 0\in\K$.
Applying $\phi$ to the relations $gx-xg=g(1-g)$ and $x^2-x=0$, then we have $b=1$. Observe that $\phi(h)\in\{h^{\prime},(g^{\prime})^2h^{\prime}\}$. Applying $\phi$ to the relations $hx-xh=\lambda h(1-g)$, then we have $\gamma=\lambda$. Similarly, we have $\HH_{3,2}(\lambda)\cong\HH_{3,2}(\gamma)$, if and only if, $\lambda=\gamma$ or $\lambda\gamma=\lambda+\gamma$.

If $(\mu,\nu)=(0,1)$, then $\Pp(H)=0$ and hence $x^2-\lambda_2x=0$. Then by rescaling $x$, $\lambda_2\in\I_{0,1}$. Similar to the last case, it follows by a direct computation that the ambiguities $(g^4)x=g^3(gx)$, $(h^2)x=h(hx)$, $g(x^2)=(gx)x$, $h(x^2)=(hx)x$ and $(gh)x=g(hx)$ are resolvable. Then by Diamond lemma, $\dim H=16$. If $\lambda_2=0$, then we can take $\lambda_1\in\I_{0,1}$, which gives two classes of $H$ described in $(19)$--$(20)$. If $\lambda_2=1$, then $H\cong\HH_4(\lambda_1)$ described in $(21)$. Similar to the last case, $\HH_4(\lambda)\cong\HH_4(\gamma)$, if and only if, $\lambda=\gamma+i$ for $i\in\I_{0,1}$.

\subsubsection{Case $\G(H)\cong \Z_2\times \Z_2\times \Z_2$.} Then $\widehat{\Z_2\times \Z_2\times \Z_2}=\{\epsilon\}$ and $Z(\Z_2\times \Z_2\times \Z_2)=\Z_2\times \Z_2\times \Z_2:=\langle g\rangle\times\langle h\rangle\times\langle k\rangle$. Then by Remark \ref{rmk:dimV=1}, $x\in V_{g^{\mu}h^{\nu}k^{\iota}}^{\epsilon}$ for $\mu,\nu,\iota\in\I_{0,1}$. Therefore,
\begin{align*}
 \gr H=\K\langle g,h,k, x\mid g^2=1,h^2=1,k^2=1,  gx=xg,hx=xh,kx=xk, x^2=0\rangle,
\end{align*}
 with $g,h,k\in\G(H)$ and $x\in\Pp_{1,g^{\mu}h^{\nu}k^{\iota}}(H)$. Then by a similar computation as before, we have
\begin{gather*}
gx-xg=\lambda_1g(1-g^{\mu}h^{\nu}k^{\iota}),\quad hx-xh=\lambda_2h(1-g^{\mu}h^{\nu}k^{\iota}),\quad kx-xk=\lambda_3k(1-g^{\mu}h^{\nu}k^{\iota}),\\
x^2-(\mu\lambda_1+\nu\lambda_2+\iota\lambda_3)x\in\Pp(H).
\end{gather*}
for some $\lambda_1,\lambda_2,\lambda_3\in\K$.  Observe that $\Z_2\times \Z_2\times \Z_2$ is 2-torsion. Then we can take $(\mu,\nu,\iota)=(0,0,0),(1,0,0)$.

If $(\mu,\nu,\iota)=(0,0,0)$, then $gx-xg=hx-xh=kx-xk=0$ in $H$ and $\Pp(H)=\K\{x\}$, which implies that $x^2=\lambda_4x$. By rescaling $x$, $\lambda_4\in\I_{0,1}$. Then $H\cong \K[\Z_2\times \Z_2\times \Z_2]\otimes\K[x]/(x^2-\lambda_4x)$, which gives two classes of $H$ described in $(22)$--$(23)$.

If $(\mu,\nu,\iota)=(1,0,0)$, then $\Pp(H)=0$ and hence $x^2-\lambda_1x=0$ in $H$. It follows by a direct computation that the ambiguities $(a^2)b=a(ab)$ and $(ab)c=a(bc)$ for $a,b,c\in\{g,h,k,x\}$ are resolvable. By Diamond lemma, $\dim H=16$. By rescaling $x$, we can take $\lambda_1\in\I_{0,1}$.

If $\lambda_1=0$, then we can take $\lambda_2\in\I_{0,1}$ by rescaling $x$. If $\lambda_2=0$, then we can also take $\lambda_3\in\I_{0,1}$, which gives two classes of $H$ described in $(24)$ and $(25)$. In fact, if $\lambda_3=1$, then $H\cong\HH_5(0)$.  If $\lambda_2=1$, then $H\cong\HH_5(\lambda_3)$. If $\lambda_1=1$, then $H\cong\HH_{6}(\lambda_2,\lambda_3)$ described in $(26)$.

We claim that $\HH_5(\lambda)\cong\HH_5(\gamma)$, if and only if,
\begin{align*}
\lambda\gamma=\lambda+\gamma,\quad\text{or }(1+\lambda)\gamma=1, \quad\text{or } \lambda=\gamma+i,\quad \text{or } 1+i\gamma=\lambda\gamma,\quad i\in\I_{0,1}.
\end{align*}

Suppose that $\phi:\HH_5(\lambda)\rightarrow \HH_5(\gamma)$ for $\lambda,\gamma\in\K$ is a Hopf algebra isomorphism. Then $\phi|_{\Z_2\times \Z_2\times \Z_2}:\Z_2\times \Z_2\times \Z_2\rightarrow \Z_2\times \Z_2\times \Z_2$ is an automorphism.   Write $g^{\prime},h^{\prime},x^{\prime}$ to distinguish the generators of $\HH_5(\gamma)$.
Since spaces of skew-primitive elements of $\HH_5(\gamma)$ are trivial except $\Pp_{1, g^{\prime} }(\HH_5(\gamma))=\K\{x^{\prime}\}\oplus \K\{1- g^{\prime} \}$, it follows that
$$\phi(g)=g^{\prime},\quad \phi(x)=a(1- g^{\prime} )+bx^{\prime}$$ for some $a,b\neq 0\in\K$. Let $\phi(h)=(g^{\prime})^p(h^{\prime})^q(k^{\prime})^r$ for $p,q,r\in\I_{0,1}$. Then applying $\phi$ to the relation $hx-xh=h(1-g)$, we have
$$(q+r\gamma)b=1.$$

Let $\phi(k)=(g^{\prime})^{\mu}(h^{\prime})^{\nu}(k^{\prime})^{\iota}$ for $\mu,\nu,\iota\in\I_{0,1}$. Then applying $\phi$ to the relation $kx-xk=\lambda k(1-g)$, we have
$$(\nu+\gamma\iota)b=\lambda.$$
Observe that $\phi|_{\G(\HH_{5}(\lambda))}$ is an isomorphism if and only if $q\iota+r\nu=1$. Hence by a case by case discussion, we have
\begin{align*}
\lambda\gamma=\lambda+\gamma,\quad\text{or }(1+\lambda)\gamma=1, \quad\text{or } \lambda=\gamma+i,\quad \text{or } 1+i\gamma=\lambda\gamma,\quad i\in\I_{0,1}.
\end{align*}

Conversely, if $\lambda\gamma=\lambda+\gamma$, then let $\psi:\HH_5(\lambda)\rightarrow\HH_5(\gamma)$ be the algebra given by
\begin{align*}
\psi(g)=g^{\prime},\quad \psi(h)=h^{\prime}k^{\prime},\quad \psi(k)=k^{\prime},\quad \psi(x)=(1-\lambda) x^{\prime};
\end{align*}
if $(1+\gamma)\lambda=1$, then let $\psi:\HH_5(\lambda)\rightarrow\HH_5(\gamma)$ be the algebra given by
\begin{align*}
\psi(g)=g^{\prime},\quad \psi(h)=h^{\prime}k^{\prime},\quad \psi(k)=h^{\prime},\quad \psi(x)=\lambda x^{\prime};
\end{align*}
if $i+\gamma=\lambda$ for $i\in\I_{0,1}$, then let $\psi:\HH_5(\lambda)\rightarrow\HH_5(\gamma)$ be the algebra given by
\begin{align*}
\psi(g)=g^{\prime},\quad \psi(h)=h^{\prime},\quad \psi(k)=(h^{\prime})^ik^{\prime},\quad \psi(x)= x^{\prime};
\end{align*}
if $1+i\gamma=\lambda\gamma$ for $i\in\I_{0,1}$, then let $\psi:\HH_5(\lambda)\rightarrow\HH_5(\gamma)$ be the algebra given by
\begin{align*}
\psi(g)=g^{\prime},\quad \psi(h)=k^{\prime},\quad \psi(k)= h^{\prime} (k^{\prime})^i,\quad \psi(x)= \gamma^{-1}x^{\prime}.
\end{align*}
It follows by a direct computation that $\psi$ is a well-defined Hopf algebra epimorphism. Observe that $\psi|_{\Pp_{1,g}(H_5(\lambda))}$ is injective. Then $\psi$ is a Hopf algebra isomorphism.

We claim that $\HH_6(\lambda_1,\lambda_2)\cong\HH_6(\gamma_1,\gamma_2)$, if and only if, there exists $q,r,\nu,\iota\in\I_{0,1}$ such that
\begin{align}\label{eq:HH6-condition-iso}
q\iota+r\nu=1,\quad q\gamma_1+r\gamma_2 =\lambda_1,\quad \nu\gamma_1+\iota\gamma_2=\lambda_2.
\end{align}

Suppose that $\phi:\HH_6(\lambda_1,\lambda_2)\rightarrow \HH_6(\gamma_1,\gamma_2)$ for $\lambda_1,\lambda_2,\gamma_1,\gamma_2\in\K$ is a Hopf algebra isomorphism. Similar to the last case, we have
$$\phi(g)=g^{\prime},\quad \phi(x)=a(1- g^{\prime} )+bx^{\prime}$$ for some $a,b\neq 0\in\K$. Applying $\phi$ to the relations $gx-xg=g(1-g), x^2-x=0$, we have $b=1$.

Let $\phi(h)=(g^{\prime})^p(h^{\prime})^q(k^{\prime})^r$ and $\phi(k)=(g^{\prime})^{\mu}(h^{\prime})^{\nu}(k^{\prime})^{\iota}$ for $\mu,\nu,\iota\in\I_{0,1}$, $p,q,r\in\I_{0,1}$. Observe that $q\iota+r\nu=1$ since $\phi$ is an isomorphism. Then applying $\phi$ to the relations $hx-xh=\lambda_1h(1-g)$ and $kx-xk=\lambda_2k(1-g)$, we have
$$q\gamma_1+r\gamma_2 =\lambda_1,\quad \nu\gamma_1+\iota\gamma_2=\lambda_2.$$

Conversely, if there exist $q,r,\nu,\iota$ satisfying conditions \eqref{eq:HH6-condition-iso}, then let $\psi:\HH_6(\lambda_1,\lambda_2)\rightarrow\HH_6(\gamma_1,\gamma_2)$ be the algebra defined by
\begin{align*}
\psi(g)=g^{\prime},\quad \psi(h)= (h^{\prime})^q(k^{\prime})^r,\quad \phi(k)= (h^{\prime})^{\nu}(k^{\prime})^{\iota},\quad \psi(x)= x^{\prime}.
\end{align*}
It follows by a direct computation that $\psi$ is a well-defined Hopf algebra epimorphism. Observe that $\psi|_{\Pp_{1,g}(H_6(\lambda_1,\lambda_2))}$ is injective. Then $\psi$ is a Hopf algebra isomorphism.

\subsection{Coradical of dimension $4$}
In this case, $\G(H)\cong \Z_4$ or $\Z_2\times \Z_2$. Then $\dim R=4$. Then by Proposition \ref{pro-P-duality}, $\dim V\leq 2$. If $\dim V=1$, then there is an element $x\in V$ such that $c(x\otimes x)=x\otimes x$, which implies that $\dim R=\dim\BN(\K\{x\})=2$, a contradiction. Therefore, $\dim V=2$. By Remark \ref{rmk-11-1--1}, $V$ is of diagonal type and hence $R\cong\K[x,y]/(x^2,y^2)$.

\subsubsection{$\G(H)\cong \Z_4:=\langle g\rangle$.} Then by Lemma \ref{lem:cyclic-groups-dimV=2}, $V\cong M_{i,1}\oplus M_{j,1}$ for $i,j\in\I_{0,3}$ or $M_{k,2}$ for $k\in\{0,2\}$.


Assume that $V\cong M_{i,1}\oplus M_{j,1}$ for $i,j\in\I_{0,3}$, that is,  $x\in V_{g^i}^{\epsilon},~y\in V_{g^j}^{\epsilon}$. Then
\begin{align*}
\gr H:=\K\langle g,x,y\mid g^4=1,gx=xg,gy=yg,x^2=0, y^2=0,xy-yx=0\rangle,
\end{align*}
with $g\in\G(H)$, $x\in\Pp_{1,g^i}(H)$ and $y\in\Pp_{1,g^j}(H)$. Observe that $\Aut(\Z_4)\cong \Z_2$. Up to isomorphism, we can take
$$
(i,j)\in\{(0,0),(0,1),(0,2),(1,1),(1,2),(1,3),(2,2)\}.
$$
By similar computations as before, we have
\begin{gather*}
gx-xg=\lambda_1g(1-g^i),\quad gy-yg=\lambda_2g(1-g^j),\\
x^2-i\lambda_1x\in\Pp_{1,g^{2i}}(H),\quad y^2-j\lambda_2y\in\Pp_{1,g^{2j}}(H),\\
xy-yx+\lambda_1jy-\lambda_2ix\in\Pp_{1,g^{i+j}}(H).
\end{gather*}
for $\lambda_1,\lambda_2\in\I_{0,1}$.

Assume that $(i,j)=(0,0)$. Then $gx=xg$, $gy=yg$ in $H$ and $\Pp(H)=\K\{x,y\}$. Then
\begin{align*}
x^2=\mu_1x+\mu_2y,\quad y^2=\mu_3x+\mu_4y,\quad xy-yx=\mu_5x+\mu_6y,
\end{align*}
for some $\mu_1,\mu_2,\cdots,\mu_6\in\K$. In this case, $H\cong\K[\Z_4]\otimes U(\Pp(H))$, where $U(\Pp(H))$ is the restricted universal enveloping algebra. Then by \cite[Theorem 7.4]{W1}, we obtain five classes of $H$ described in $(27)$--$(31)$.

Assume that $(i,j)=(0,1)$. Then $\Pp(H)=\K\{x\}$, $\Pp_{1,g}(H)=\K\{1-g,y\}$ and $\Pp_{1,g^2}(H)=\K\{1-g^2\}$. Hence
$$x^2=\mu_1x, \quad y^2-\lambda_2y=\mu_2(1-g^2),\quad xy-yx=\mu_3y+\mu_4(1-g),$$ for some $\mu_1,\mu_2,\mu_3,\mu_4\in\K$. We can take $\mu_1\in\I_{0,1}$ and  $\mu_2=0$ by rescaling $x,y$ and via the linear translation $y:=y-a(1-g)$ satisfying $a^2-\lambda_2a=\mu_2$. Then it follows by a direct computation that
\begin{align*}
[x,[x,y]]&=\mu_3[x,y]=\mu_3^2y+\mu_3\mu_4(1-g),\quad [x^2,y]=[\mu_1x,y]=\mu_1\mu_3y+\mu_1\mu_4(1-g),\\
[[x,y],y]&=-\mu_4[g,y]=-\mu_4\lambda_2g(1-g),\quad [x,y^2]=\lambda_2[x,y]=\lambda_2\mu_3y+\lambda_2\mu_4(1-g).
\end{align*}
By Proposition \ref{proJ}, $[x,[x,y]]=[x^2,y]$ and $[[x,y],y]=[x,y^2]$, which implies that
\begin{align*}
(\mu_1-\mu_3)\mu_3=0,\quad (\mu_1-\mu_3)\mu_4=0,\quad \lambda_2\mu_3=0,\quad \lambda_2\mu_4=0.
\end{align*}
Then it is easy to verify that the ambiguities $(g^4)x=g^3(gx)$, $(g^4)y=g^3(gx)$, $(x^2)y=x(xy)$, $(xy)y=x(y^2)$, $(gx)y=g(xy)$, $(x^2)x=x(x^2)$ and $(y^2)y=y(y^2)$ are resolvable. By Diamond lemma, $\dim H=16$.

If $\lambda_2=0=\mu_1$, then $\mu_3=0$ and we can take $\mu_4\in\I_{0,1}$ by rescaling $x$, which gives two classes of $H$ described in $(32)$ and $(33)$.

If $\lambda_2=0=\mu_1-1$, then $\mu_3^2=\mu_3$ and $\mu_4=\mu_3\mu_4$ and hence we can take $\mu_3\in\I_{0,1}$ by rescaling $x$. If $\mu_3=0$, then $\mu_4=0$, which gives one class of $H$ described in $(34)$. If $\mu_3=1$, then we can take $\mu_4=0$ via the linear translation $y\mapsto y-\mu_4(1-g)$, which gives one class of $H$ described in $(35)$.

If $\lambda_2=1$, then $\mu_3=0=\mu_4$, which gives two classes of $H$ described in $(36)$--$(37)$.

Assume that $(i,j)=(0,2)$. Then $\Pp(H)=\K\{x\}$, $\Pp_{1,g^2}(H)=\K\{1-g^2,y\}$. Hence
\begin{align*}
x^2=\mu_1x,\quad y^2=\mu_2x,\quad xy-yx=\mu_3y+\mu_4(1-g^2).
\end{align*}
From $[x,[x,y]]=[x^2,y]$,  $[[x,y],y]=[x,y^2]$, $(x^2)x=x(x^2)$ and $(y^2)y=y(y^2)$, we have
\begin{align*}
(\mu_1-\mu_3)\mu_3=0,\quad (\mu_1-\mu_3)\mu_4=0,\quad \mu_2\mu_3=0=\mu_2\mu_4.
\end{align*}
Then it is easy to verify that the ambiguities $(g^4)x=g^3(gx)$, $(g^4)y=g^3(gy)$, $(x^2)y=x(xy)$, $(xy)y=x(y^2)$, $(gx)y=g(xy)$ are resolvable. By Diamond lemma, $\dim H=16$.

By rescaling $x,y$, $\lambda_2,\mu_1\in\I_{0,1}$.

If $\mu_1=0$, then $\mu_3=0$ and $\mu_2\mu_4=0$. If $\mu_2=0$, then we can take $\mu_4\in\I_{0,1}$ by rescaling $x$. If $\mu_4=0$, then we can take $\mu_2\in\I_{0,1}$. Therefore, $(\mu_2,\mu_4)$ admits three possibilities and $\lambda_2\in\I_{0,1}$, which gives six classes of $H$ described in $(38)$--$(43)$.

If $\mu_1=1$, then $\mu_3^2=\mu_3$ and $\mu_4=\mu_3\mu_4$, which implies that $\mu_3\in\I_{0,1}$ by rescaling $x$.
\begin{itemize}

\item If $\mu_3=0$, then $\mu_4=0$, which impies that $xy-yx=0$ in $H$.
If $\lambda_2=0$, then by rescaling $y$, we can take $\mu_2\in\I_{0,1}$, which gives two classes of $H$ described in $(44)$--$(45)$. If $\lambda_2=1$, then $H\cong\HH_7(\mu_2)$ described in $(46)$.

\item If $\mu_3=1$, then $\mu_2=0$, that is, $y^2=0$ in $H$. Hence we can take $\mu_4=0$ via the linear translation $y\mapsto y-\mu_4(1-g^2)$. Indeed, it is easy to see that the translation is a well-defined Hopf algebra isomorphism. Therefore, we obtain two classes of $H$ described in $(47)$--$(48)$.
\end{itemize}

Now we claim that $\HH_7(\lambda)\cong\HH_7(\gamma)$, if and only if, $\lambda=\gamma$.

Suppose that $\phi:\HH_7(\lambda)\rightarrow \HH_7(\gamma)$ for $\lambda,\gamma\in\K$ is a Hopf algebra isomorphism.  Write $g^{\prime},x^{\prime},y^{\prime}$ to distinguish the generators of $\HH_7(\gamma)$.
Observe that spaces of skew-primitive elements of $\HH_7(\gamma)$ are trivial except $\Pp_{1,(g^{\prime})^2}(\HH_7(\gamma))=\K\{y^{\prime}\}\oplus \K\{1-(g^{\prime})^2\}$ and $\Pp(\HH_7(\gamma))=\K\{x^{\prime}\}$. Then
\begin{align*}
\phi(g)=g^{\prime\,\pm 1},\quad \phi(x)=\alpha x^{\prime},\quad \phi(y)=a(1-(g^{\prime})^2)+by^{\prime}
\end{align*}
for some $\alpha\neq 0, a,b\neq 0\in\K$. Applying $\phi$ to the relation $x^2-x=0$, we have $\alpha=1$. Applying $\phi$ to the relation $gy-yg=g(1-g^2)$, we have $b=1$. Then applying $\phi$ to the relation $y^2-\lambda x=0$, we have
\begin{align*}
\phi(y^2-\lambda x)=(y^{\prime})^2-\lambda x^{\prime}=(\gamma-\lambda)x^{\prime}=0\quad \Rightarrow\quad \gamma=\lambda.
\end{align*}

Assume that $(i,j)=(1,1)$. Then $\Pp_{1,g}(H)=\K\{1-g,x,y\}$ and $\Pp_{1,g^2}=\K\{1-g^2\}$. Hence
\begin{align*}
x^2-\lambda_1x=\mu_1(1-g^2),\quad y^2-\lambda_2y=\mu_2(1-g^2),\quad xy-yx+\lambda_1y-\lambda_2x=\mu_3(1-g^2),
\end{align*}
for $\mu_1,\mu_2,\mu_3\in\K$. It follows by a direct computation that all ambiguities are resolvable and hence by the Diamond lemma, $\dim H=16$. We can take $\mu_1=0=\mu_2$ via the linear translation $x\mapsto x-a(1-g)$, $y\mapsto y-b(1-g)$ satisfying $a^2-\lambda_1a=\mu_2$ and $b^2-\lambda_2b=\mu_3$. If $\lambda_1=0$ or $\lambda_2=0$, then we can take $\mu_3\in\I_{0,1}$ by rescaling $x$ or $y$.

If $\lambda_1=0=\lambda_2$, then $\mu_3\in\I_{0,1}$, which gives two classes of $H$ described in $(49)$--$(50)$. If $\lambda_1-1=0=\lambda_2$, then $\mu_3\in\I_{0,1}$, which gives two classes of $H$ described in $(51)$--$(52)$. If $\lambda_1=0=\lambda_2-1$, then $\mu_3\in\I_{0,1}$, which gives two classes of $H$ described in $(51)$--$(52)$ by swapping $x$ and $y$. If $\lambda_1=\lambda_2=1$, then $H$ is isomorphic to one of the Hopf algebras described in $(51)$--$(52)$. Indeed, in this case, consider the translation $y\mapsto y+x+a(1-g)$ satisfying $a^2=\mu_3$, it is easy to see that $H$ is isomorphic to the Hopf algebras defined by
\begin{align*}
\K\langle g,x,y\mid g^4=1, [g,x]=g(1-g), [g,y]=0, x^2=x,y^2=0,[x,y]=y+(a+\mu_3)(1-g^2)\rangle.
\end{align*}
If $a+\mu_3=0$, then $H$ is isomorphic to the Hopf algebra described in $(51)$. If $a+\mu_3\neq 0$, then by rescaling $y$, $H$ is isomorphic to the Hopf algebra described in $(52)$.

Assume that $(i,j)=(1,2)$. Then $\Pp(H)=0$, $\Pp_{1,g}(H)=\{1-g,x\}$, $\Pp_{1,g^2}(H)=\{1-g^2,y\}$ and $\Pp_{1,g^3}(H)=\K\{1-g^3\}$. Hence,
\begin{align*}
x^2-\lambda_1x=\mu_1y+\mu_2(1-g^2),\quad y^2=0,\quad xy-yx-\lambda_2x=\mu_3(1-g^3),
\end{align*}
for some $\mu_1,\mu_2,\mu_3\in\K$. The verification of the ambiguities $(a^2)b=a(ab)$ and $(ab)b=a(b^2)$ for all $a,b\in\{g,x,y\}$ and $(gx)y=g(xy)$ amount to the conditions
\begin{align*}
\mu_1\lambda_2=0=\mu_1\mu_3,\quad \lambda_2=0.
\end{align*}
Then by Diamond lemma, $\dim H=16$. We can take $\mu_2=0$ via the linear translation $x:=x-a(1-g)$ satisfying $a^2-\lambda_1a=\mu_2$ and take $\mu_3\in\I_{0,1}$ by rescaling $y$.

If $\lambda_1=0$, then we can take $\mu_1\in\I_{0,1}$ by rescaling $x$,  which gives three classes of $H$ described in $(53)$--$(55)$. If $\lambda_1-1=0=\mu_3$, then by rescaling $y$, we can take $\mu_1\in\I_{0,1}$, which gives two classes of $H$ described in $(56)$--$(57)$. If $\lambda_1-1=0=\mu_3-1$, then $\mu_1=0$, which gives one class of $H$ described in $(58)$.

Assume that $(i,j)=(1,3)$. Then $\Pp(H)=0$, $\Pp_{1,g}(H)=\K\{1-g,x\}$, $\Pp_{1,g^2}(H)=\K\{1-g^2\}$ and $\Pp_{1,g^3}(H)=\K\{1-g^3,y\}$. Hence
\begin{align*}
x^2-\lambda_1x=\mu_1(1-g^2),\quad y^2-\lambda_2y=\mu_2(1-g^2),\quad [x,y]+\lambda_1y-\lambda_2x=0,
\end{align*}
for some $\mu_1,\mu_2\in\K$. It follows by a direct computation that all ambiguities are resolvable and hence by Diamond lemma, $\dim H=16$. Then we can take $\mu_1=0=\mu_2$ via the linear translation $x\mapsto x-a(1-g),~y\mapsto y-b(1-g^3)$ satisfying $a^2-a\lambda_1=\mu_1,b^2-b\lambda_2=\mu_2$. Therefore, the structure of $H$ depends on $\lambda_1,\lambda_2\in\I_{0,1}$, denoted by $H(\lambda_1,\lambda_2)$.

We claim that $H(0,1)\cong H(1,0)$. Indeed, consider the the algebra map $\phi: H(0,1)\rightarrow H(1,0)$ given by $\phi(g)=g^3$, $\phi(x)=y$ and $\phi(y)=x$. It follows by a direct computation that $\phi$ is a Hopf algebra morphism. Obviously, $\phi$ is an epimorphism and $\phi|_{(H(0,1))_1}$ is injective. Therefore, $\phi$ is an isomorphism. It is easy to see that $H(0,0)$, $H(1,0)$ and $H(1,1)$ are pairwise non-isomorphic. Therefore, we obtain three classes of $H$ described in $(59)$--$(61)$.

Assume that $(i,j)=(2,2)$. Then $\Pp(H)=0$. Hence
\begin{align*}
x^2=0,\quad y^2=0,\quad xy-yx=0.
\end{align*}
Then it is easy to see that all ambiguities are resolvable and hence by the Diamond lemma, $\dim H=16$. Similar to the last case, we obtain two classes of $H$ described in $(62)$--$(63)$.

Assume that $V\cong  M_{k,2}$ for $k\in\{0,2\}$. Then
\begin{gather*}
\gr H:=\K\langle g,x,y\mid g^4-1,gx=xg,gy=(y+x)g,xy-yx,x^2,y^2\rangle;
\end{gather*}
  with  $ g\in\G(\gr H), x,y\in\Pp_{1,g^{2k}}(\gr H)$ for $k\in\I_{0,1}$. By similar computations as before, we have
\begin{gather*}
gx-xg=\lambda_1(g-g^{2k+1}),\quad gy-(y+x)g=\lambda_2(g-g^{2k+1}),\quad xy-yx,x^2,y^2\in\Pp(H),
\end{gather*}
for some $\lambda_1,\lambda_2\in\K$.

If $k=0$, then $\Pp(H)=\K\{x,y\}$, which implies that
\begin{align*}
x^2=\alpha_1x+\alpha_2y,\quad y^2=\alpha_3x+\alpha_4y,\quad xy-yx=\alpha_4x+\alpha_6y;
\end{align*}
for some $\alpha_1,\cdots,\alpha_6\in\K$.
Observe that $\Pp(H)$ is a two-dimensional restricted Lie algebra and $H\cong\K[\Z_4]\sharp U(\Pp(H)$, where $U(\Pp(H))$ is the restricted universal enveloping algebra. Then by \cite[Theorem 7.4]{W1}, we obtain five classes of $H$ described in $(64)$--$(68)$.

If $k=1$, then $\Pp(H)=0$ and hence the defining relations of $H$ are
\begin{gather*}
gx-xg=\lambda_1(g-g^{3}),\quad gy-(y+x)g=\lambda_2(g-g^{3}),\quad xy-yx=x^2=y^2=0.
\end{gather*}
The verification of the ambiguities $(a^2)b=a(ab)$ and $(ab)b=a(b^2)$ for all $a,b\in\{g,x,y\}$ and $(gx)y=g(xy)$ gives no conditions.
Then by Diamond lemma, $\dim H=16$. We write $H(\lambda_1,\lambda_2):=H$ for convenience.

\textbf{Cliam:} $H(\lambda_1,\lambda_2)\cong H(\gamma_1,\gamma_2)$, if and only if, there exist $\alpha_1,\alpha_2\neq 0,\beta_2\in\K$ such that $\alpha_2\gamma_1=\lambda_1$ and $\beta_2\gamma_1-\alpha_1+\alpha_2\gamma_2-\lambda_2=0$.

Suppose that $\phi: H(\lambda_1,\lambda_2)\rightarrow H(\gamma_1,\gamma_2)$ for $\lambda_1,\lambda_2,\gamma_1,\gamma_2\in\K$ is a Hopf algebra isomorphism.  Write $g^{\prime},x^{\prime},y^{\prime}$ to distinguish the generators of $H(\gamma_1,\gamma_2)$.
  Then
\begin{align*}
\phi(g)=g^{\prime\,\pm 1},\quad \phi(x)=\alpha_1(1-(g^{\prime})^2)+\alpha_2x^{\prime}+\alpha_3y^{\prime},\quad \phi(y)=\beta_1(1-(g^{\prime})^2)+\beta_2x^{\prime}+\beta_3y^{\prime}
\end{align*}
for some $\alpha_1,\alpha_2,\alpha_3,\beta_1,\beta_2,\beta_3\in\K$. Applying $\phi$ to the relation $gx-xg=\lambda_1g(1-g^2)$, we have $\alpha_3=0=\alpha_2\gamma_1-\gamma_1$. Then applying $\phi$ to the relation $gy-(y+x)g=\lambda_2g(1-g^2)$, we have
\begin{align*}
\beta_3=\alpha_2,\quad \beta_2\gamma_1-\alpha_1+\gamma_2\beta_3-\lambda_2=0.
\end{align*}
Then it is easy to check that $\phi$ is a well-defined bialgebra map. Since $\phi$ is an isomorphism, it follows that $\alpha_2\neq 0$. Consequently, the claim follows.

By rescaling $x$, we can take $\lambda_1\in\I_{0,1}$. Then from the last claim, we have $H(\lambda_1,0)\cong H(\lambda_1,\lambda_2)$ for $\lambda_1\in\I_{0,1}$ and $H(0,0)\not\cong H(1,0)$. Consequently, we obtain two classes of $H$ described in $(69)$--$(70)$.
\subsubsection{$\G(H)\cong \Z_2\times \Z_2:=\langle g\rangle\times\langle h\rangle$.} If $V$ is a decomposable object in ${}_{\Z_2\times \Z_2}^{\Z_2\times \Z_2}\mathcal{YD}$, then $V:=\K\{x,y\}$ must be the sum of two one-dimensional objects in ${}_{\Z_2\times \Z_2}^{\Z_2\times \Z_2}\mathcal{YD}$ such that $x\in V_{g^ih^j}^{\epsilon},~y\in V_{g^{\mu}h^{\mu}}^{\epsilon}$ for $i,j,\mu,\nu\in\I_{0,1}$. If $V$ is an indecomposable object in  ${}_{\Z_2\times \Z_2}^{\Z_2\times \Z_2}\mathcal{YD}$, then by \cite{Ba} and Theorem \ref{thm:indecomposable-object-YD-over-groups}, $V:=\K\{x,y\}\in{}_{\Z_2\times \Z_2}^{\Z_2\times \Z_2}\mathcal{YD}$ by
\begin{gather*}
g\cdot x=x,\quad g\cdot y=y+x,\quad h\cdot x=x,\quad h\cdot y=y+\lambda x,\quad \lambda\in\K;\\
\delta(x)=g^kh^l\otimes x,\quad \delta(y)=g^kh^l\otimes y,\quad\text{for some }k,l\in\I_{0,1}.
\end{gather*}
We claim that $(k,l,\lambda)\in\{(0,0,\lambda),(0,1,0),(1,1,1)\}$; otherwise, $V$ is of Jordan type, a contradication.

Assume that $V$ is a decomposable object in ${}_{\Z_2\times \Z_2}^{\Z_2\times \Z_2}\mathcal{YD}$. Then $x\in V_{g^ih^j}^{\epsilon},~y\in V_{g^{\mu}h^{\mu}}^{\epsilon}$ for $i,j,\mu,\nu\in\I_{0,1}$. Without loss of generality, we may assume that $x,y\in V_{1}$, $x\in V_{g},y\in V_{g^i}$ for $i\in\I_{0,1}$ or $x\in V_{g},y\in V_{h}$.

Assume that $x,y\in V_{1}^{\epsilon}$. Then $H\cong\K[\Z_2\times \Z_2]\otimes U(\Pp(H))$, where $U(\Pp(H))$ is the restricted universal enveloping algebra of $\Pp(H)$. Then by \cite[Theorem 7.4]{W1}, we obtain five classes of $H$ described in $(71)$--$(75)$.

Assume that $x\in V_{g}^{\epsilon}, y\in V_{1}^{\epsilon}$. Then by Lemma \ref{lem:p4-x1yu}, the defining relations of $H$ are
\begin{gather*}
g^2=1,\quad h^{2}=1,\quad gx-xg=\lambda_1g(1-g), \quad gy-yg=0,\\
 hx-xh=\lambda_3h(1-g),\quad hy-yh=0,\\
x^2-\lambda_1x=\mu_1y,\quad y^2=\mu_2y,\quad xy-yx=\mu_3x+\mu_4(1-g),
\end{gather*}
for $\lambda_1\in\I_{0,1},\lambda_3,\mu_1,\cdots,\mu_4\in\K$ with the conditions
\begin{align*}
\mu_1\mu_3=0=\mu_1\mu_4,\quad \mu_2\mu_3=\mu_3^2,\quad \mu_2\mu_4=\mu_3\mu_4,\quad \lambda_1\mu_3=0=\lambda_3\mu_3.
\end{align*}
By rescaling $y$, we can take $\mu_2\in\I_{0,1}$.

If $\lambda_1=0=\mu_2$, then $\mu_3=0=\mu_1\mu_4$ and we can take $\lambda_3,\mu_4\in\I_{0,1}$ by rescaling $x,y$. If $\mu_4=0$, then by rescaling $y$, $\mu_1\in\I_{0,1}$. If $\mu_4\neq 0$, then $\mu_1=0$ and we can take $\mu_4=1$ by rescaling $y$. Therefore, we obtain six classes of $H$ described in $(76)$--$(81)$.

If $\lambda_1=0=\mu_2-1$, then $\mu_3=\mu_3^2$, $(\mu_3-1)\mu_4=0$ and $\mu_1\mu_4=0=\lambda_3\mu_3$. We can take $\mu_3\in\I_{0,1}$ by rescaling $y$. If $\mu_3=0$, then $\mu_4=0$ and we can take $\lambda_3\in\I_{0,1}$ by rescaling $x$, which gives four classes of $H$ described in $(82)$--$(85)$.  If $\mu_3=1$, then $\lambda_3=0$ and we can take $\mu_1\in\I_{0,1}$ by rescaling $x$. If $\mu_1=0$, then we can take $\mu_4=0$ via the linear translation $x:=x+\mu_4(1-g)$, which gives one class of $H$ described in $(86)$. If $\mu_1=1$, then $\mu_4=0$, which gives one class of $H$ described in $(87)$.

If $\lambda_1-1=0=\mu_2$, then $\mu_3=0$ and $\mu_1\mu_4=0$. If $\mu_1=0=\mu_4$, then $H\cong \HH_{8}(\lambda_3)$ described in $(88)$. If $\mu_1\neq 0$, then $\mu_4=0$ and we can take $\mu_1=1$ by rescaling $y$, which implies that $H\cong\HH_{9}(\lambda_3)$ described in $(89)$. If $\mu_1=0$ and $\mu_4\neq 0$, then by rescaling $y$, $\mu_4=1$, which implies that $H\cong\HH_{10}(\lambda_3)$ described in $(90)$.

If $\lambda_1=\mu_2=1$, then $\mu_3=0=\mu_4$ and hence $H\cong\HH_{11}(\lambda_3,\mu_1)$ described in $(91)$.

\textbf{Claim:} $\HH_n(\lambda)\cong\HH_n(\gamma)$ for $n\in\I_{8,10}$, if and only if, $\lambda=\gamma+i$ for $i\in\I_{0,1}$; $\HH_{11}(\lambda,\mu)\cong\HH_{11}(\gamma,\nu)$ if and only if $\lambda=\gamma+i$ for $i\in\I_{0,1}$ and $\mu=\nu$.

Suppose that $\phi:\HH_8(\lambda)\rightarrow \HH_8(\gamma)$ for $\lambda,\gamma\in\K$ is a Hopf algebra isomorphism. Then $\phi|_{\Z_2\times \Z_2}:\Z_2\times \Z_2\rightarrow \Z_2\times \Z_2$ is an automorphism.   Write $g^{\prime},h^{\prime},x^{\prime}$ to distinguish the generators of $\HH_8(\gamma)$.
Since spaces of skew-primitive elements of $\HH_8(\gamma)$ are trivial except $\Pp_{1, g^{\prime} }(\HH_8(\gamma))=\K\{x^{\prime}\}\oplus \K\{1- g^{\prime} \}$ and $\Pp(\HH_8(\gamma))=\K\{y^{\prime}\}$, it follows that
$$\phi(g)=g^{\prime},\quad \phi(h)=(g^{\prime})^ih^{\prime},\quad \phi(x)=a(1- g^{\prime} )+bx^{\prime},\quad \phi(y)=cy^{\prime}$$ for some $a,b\neq 0,c\neq 0\in\K$ and $i\in\I_{0,1}$. Then applying $\phi$ to the relations $gx-xg=g(1-g)$ and $hx-xh=\lambda h(1-g)$, we have
$$b=1\quad b(i+\gamma)=\lambda\quad \Rightarrow\quad i+\gamma =\lambda .$$

Conversely, if $\lambda=\gamma+i$ for $i\in\I_{0,1}$, then consider the algebra map $\psi:\HH_{8}(\lambda)\rightarrow\HH_8(\gamma), g\rightarrow g, h\rightarrow g^ih, x\rightarrow x,y\rightarrow y$. It is easy to see that $\psi$ is a Hopf algebra epimorphism and $\psi|_{\HH_8(\lambda)_1}$ is injective. Therefore, $\HH_8(\lambda)\cong\HH_8(\gamma)$.

Similarly, $\HH_n(\lambda)\cong\HH_n(\gamma)$ for $n\in\I_{9,10}$, if and only if, $\lambda=\gamma+i$ for $i\in\I_{0,1}$; $\HH_{11}(\lambda,\mu)\cong\HH_{11}(\gamma,\nu)$ if and only if $\lambda=\gamma+i$ for $i\in\I_{0,1}$ and $\mu=\nu$.

Assume that $x,y\in V_{g}^{\epsilon}$. Then by Lemma \ref{lem:p4-x1yu}, the defining relations of $H$ are
\begin{gather*}
g^2=1,\quad h^{2}=1,\quad  gx-xg=\lambda_1g(1-g), \quad gy-yg=\lambda_2g(1-g ),\\
hx-xh=\lambda_3h(1-g),\quad hy-yh=\lambda_4h(1-g ),\\
x^2-\lambda_1x=0,\quad y^2- \lambda_2y=0,\quad xy-yx+ \lambda_1y-\lambda_2x=0.
\end{gather*}
for $\lambda_1,\lambda_2\in\I_{0,1},\lambda_3,\cdots,\lambda_5\in\K$.

If $\lambda_1=0=\lambda_2$, then we can take $\lambda_3,\lambda_4\in\I_{0,1}$ by rescaling $x,y$, which gives two classes of $H$ described in $(92)$--$(93)$. Let $H:=H(\lambda_3,\lambda_4)$ for convenience. Indeed, $H(1,0)\cong H(0,1)$ by swapping $x$ and $y$; $H(1,0)\cong H(1,1)$ via the Hopf algebra isomorphism $\phi:H(1,0)\rightarrow H(1,1)$ defined by
$$\phi(g)=g,\quad \phi(h)=h,\quad \phi(x)=x,\quad \phi(y)=x+y.$$
Moreover, $H(0,0)$ and $H(1,0)$ are not isomorphic since $H(0,0)$ is commutative while $H(1,0)$ is not commutative.

If $\lambda_1-1=0=\lambda_2$, then we can take $\lambda_4\in\I_{0,1}$ by rescaling $y$. If $\lambda_4=0$, then $H\cong\HH_{12}(\lambda_3)$ described in $(94)$. If $\lambda_4=1$, then $H\cong\HH_{13}(\lambda_3)$ described in $(95)$.

If $\lambda_1=0=\lambda_2-1$, then $H$ is isomorphic to one of the Hopf algebras described in $(94)$--$(95)$ by swapping $x$ and $y$.

If $\lambda_1=\lambda_2=1$, then $H$ is isomorphic to one of the Hopf algebras described in $(94)$--$(95)$. Indeed, consider the translation $y\mapsto x+y$, it is easy to see that $H$ is isomorphic to the Hopf algebra defined by
\begin{gather*}
g^2=1,\quad h^{2}=1,\quad  gx-xg=g(1-g), \quad gy-yg=0,\quad
hx-xh=\lambda_3h(1-g),\\ hy-yh=(\lambda_3+\lambda_4)h(1-g ),\quad
x^2-x=0,\quad y^2=0,\quad xy-yx+y=0.
\end{gather*}
If $\lambda_3+\lambda_4=0$, then $H$ is isomorphic to the Hopf algebra described in $(94)$. If $\lambda_3+\lambda_4\neq 0$, then by rescaling $y$, $H$ is isomorphic to the Hopf algebra described in $(95)$.

\textbf{Claim:} $\HH_{n}(\lambda)\cong\HH_{n}(\gamma)$ for $n\in\I_{12,13}$, if and only if, $\lambda=\gamma+i$ for $i\in\I_{0,1}$.

Assume that $x\in V_{g}^{\epsilon}, y\in V_{h}^{\epsilon}$. Then by Lemma \ref{lem:p4-xg1yhu}, the defining relations of $H$ are
\begin{gather*}
gx-xg=\lambda_1g(1-g),\quad hx-xh=\lambda_2h(1-g), \quad x^2-\lambda_1x=0\\
gy-yg=\lambda_3g(1-h ), \quad hy-yh=\lambda_4h(1-h ),\quad y^2- \lambda_4y=0,\\
 xy-yx-\lambda_3x+ \lambda_2y=\lambda_5(1-gh ).
\end{gather*}
for some $\lambda_1,\lambda_4\in\I_{0,1}$, $\lambda_2,\lambda_3,\lambda_5\in\K$. The verifications of $(a^2)b=a(ab),a(b^2)=(ab)b$ for $a,b\in\{g,h,x,y\}$ and $a(xy)=(ax)y$ for $a\in\{g,h\}$ amounts to the conditions
\begin{align*}
(\lambda_1+\lambda_2)\lambda_3=(\lambda_1+\lambda_2)\lambda_2=(\lambda_1+\lambda_2)\lambda_5=0,\\
(\lambda_3-\lambda_4)\lambda_3=(\lambda_3-\lambda_4)\lambda_2=(\lambda_3-\lambda_4)\lambda_5=0.
\end{align*}
Then by the Diamond lemma, $\dim H=16$.

If $\lambda_1=0=\lambda_4$, then $\lambda_2=0=\lambda_3$ and hence we can take $\lambda_5\in\I_{0,1}$ by rescaling $x$, which gives two classes of $H$ described in $(96)$--$(97)$.
%

If $\lambda_1-1=0=\lambda_4$, then $\lambda_2^2=\lambda_2$, $\lambda_3=0$ and $(\lambda_2-1)\lambda_5=0$. Hence we can take $\lambda_2,\lambda_5\in\I_{0,1}$ by rescaling $x,y$. If $\lambda_2=0$, then $\lambda_5=0$, which gives one class of $H$ described in $(98)$. If $\lambda_2=1$, then we can take $\lambda_5=0$ via the linear translation $y\mapsto y-\lambda_5(1-h)$, which gives one class of $H$ described in $(99)$.
%

If $\lambda_1=0=\lambda_4-1$, then we obtain two classes of $H$ described in in $(98)$--$(99)$  via the linear translation $g\mapsto h, h\mapsto g, x\mapsto y,y\mapsto x$.

If $\lambda_1=1=\lambda_4$, then $\lambda_2=\lambda_3\in\I_{0,1}$ and $(1+\lambda_2)\lambda_5=0$. If $\lambda_2=0=\lambda_3$, then $\lambda_5=0$, which gives one class of $H$ described in $(100)$. If $\lambda_2=\lambda_3=1$, then we can take $\lambda_5=0$ via the linear translation $y\mapsto y-\lambda_5(1-h)$, which gives one class of $H$ described in $(101)$.

Assume that $V$ is an indecomposable object in  ${}_{\Z_2\times \Z_2}^{\Z_2\times \Z_2}\mathcal{YD}$. Then $ \gr H=\K\langle g,h,x,y\rangle$, subject to the relations
\begin{gather*}
  g^2=h^2=x^2=y^2=1,[g,x]=[h,x]=[g,h]=0,gy=(y+x)g,  hy=(y+\lambda x)h,
\end{gather*}
with $g,h\in\G(\gr H),x,y\in\Pp_{g^kh^l}(\gr H)$, where $(k,l,\lambda)\in\{(0,0,\lambda),(0,1,0),(1,1,1)\}$. It is easy to see that $\gr H$ with $(k,l,\lambda)\in\{(0,1,0),(1,1,1)\}$ are isomorphic. Hence we can take $(k,l,\lambda)\in\{(0,0,\lambda),(0,1,0)\}$.  By similar computations as before, we have
\begin{gather*}
gx-xg=\lambda_1g(1-g^kh^l),\quad gy-(y+x)g=\lambda_2g(1-g^kh^l);\\
hx-xh=\lambda_3h(1-g^kh^l),\quad hy-(y+\lambda x)h=\lambda_4h(1-g^kh^l).
\end{gather*}

If $(k,l,\lambda)=(0,0,\lambda)$, then $\Pp(H)=\K\{x,y\}$ and $x^2,y^2,[x,y]\in\Pp(H)$. Hence $H\cong\K[\Z_4]\sharp U(\Pp(H))$, where $U(\Pp(H))$ is the restricted universal enveloping algebra of   $\Pp(H)$. Then by \cite[Theorem 7.4]{W1}, we obtain five classes of $H$ described in $(102)$--$(106)$.

If $(k,l,\lambda)=(0,1,0)$, then it follows by a direct computation that $x^2-\lambda_3x,y^2-\lambda_4y,xy-yx-\lambda_4x+\lambda_3y\in\Pp(H)$. Therefore, the defining relations of $H$ are
\begin{gather*}
g^2=h^2=1,\quad gh=hg,\quad gx-xg=\lambda_1g(1-h),\quad gy-(y+x)g=\lambda_2g(1-h),\\
hx-xh=\lambda_3h(1-h),\quad hy-yh=\lambda_4h(1-h),\\ xy-yx-\lambda_4x+\lambda_3y=0,\quad  x^2-\lambda_3x=0,\quad y^2-\lambda_4y=0.
\end{gather*}
The verifications of $(a^2)b=a(ab),(ab)b=a(b^2)$ for $a,b\in\{g,h,x,y\}$ and $a(xy)=(ax)y$ for $a\in\{g,h\}$ amounts to the conditions
\begin{gather*}
\lambda_1=0=\lambda_3.
\end{gather*}
By Diamond Lemma, $\dim H=16$. We can take $\lambda_2=0$ via the linear translation $x\mapsto x+\lambda_2(1-h)$ and take $\lambda_4\in\I_{0,1}$ by rescaling $x,y$, which gives two classes of $H$ described in $(107)$--$(108)$.
\subsection{Coradical $\K[\Z_2]$}
Then by Lemma \ref{lem:cyclic-groups-dimV=3}, $V\cong M_{i,1}\oplus M_{j,1}\oplus M_{k,1}$ for $i,j,k\in\I_{0,p-1}$  or $ M_{0,1}\oplus M_{0,2}$ and hence $\BN(V)\cong\K[x,y,z]/(x^p,y^p,z^p)$.

 Assume that $V\cong M_{i,1}\oplus M_{j,1}\oplus M_{k,1}$ for $i,j,k\in\I_{0,1}$. Then
\begin{align*}
\gr H=\K\langle g,x,y,z\mid g^2=1,[g,x]=[g,y]=[g,z]=x^2=y^2=z^2=[x,y]=[x,z]=[y,z]=0\rangle,
\end{align*}
with $g\in\G(H)$, $x\in\Pp_{1,g^i}(H)$, $y\in\Pp_{1,g^j}(H)$ and $z\in\Pp_{1,g^k}(H)$.
Up to isomorphism, we may assume that $(i,j,k)=(0,0,0)$, $(1,1,1)$, $(1,1,0)$ and $(1,0,0)$.

Assume that $(i,j,k)=(0,0,0)$.  Then $H\cong\K[\Z_2]\otimes U(\Pp(H))$, where $U(\Pp(H))$ is the restricted universal enveloping algebra of $\Pp(H)$. Then by \cite[Theorem 1.4]{NWW1}, we obtain fourteen classes of $H$ described in $(109)$--$(122)$.

Assume that $(i,j,k)=(1,1,1)$. Then by Lemma \ref{lem:p4-x1y1z1}, the defining relations of $H$ are
\begin{gather*}
g^2=1,\quad gx-xg=\lambda_1g(1-g),\quad  gy-yg=\lambda_2g(1-g),\quad gz-zg=\lambda_3g(1-g),\\
x^2-\lambda_1x=0,\quad y^2-\lambda_2y=0,\quad z^2-\lambda_3z=0,\quad
xy-yx-\lambda_2x+\lambda_1y=0,\\ xz-zx-\lambda_3x+\lambda_1z=0,\quad  yz-zy-\lambda_3y+\lambda_2z=0.
\end{gather*}
for $\lambda_1,\lambda_2,\lambda_3\in\I_{0,2}$. Let $H(\lambda_1,\lambda_2,\lambda_3):=H$ for convenience. We claim that $H(1,0,0)\cong H(1,1,0)$. Indeed, consider the algebra map $\phi:H(1,0,0)\rightarrow H(1,1,0), g\rightarrow g,x\rightarrow x,y\rightarrow x+y, z\rightarrow z$. Then it is easy to see $\phi$ is a Hopf algebra epimorphism and $\phi|_{(H(1,0,0))_1}$ is injective, which implies that the claim follows. Similarly, $H(1,1,1)\cong H(1,1,0)\cong H(1,0,0)$. Observe that $H(0,0,0)$ is commutative and $H(1,0,0)$ is not commutative. Hence $H\cong H(0,0,0)$ or $H(1,0,0)$ described in $(123)$ or $(124)$.

Assume that $(i,j,k)=(1,1,0)$. Then by Lemma \ref{lem:p4-x1y1z0}, the defining relations of $H$ are
\begin{gather*}
g^2=1,\quad gx-xg=\lambda_1g(1-g),\quad gy-yg=\lambda_2g(1-g),\quad gz-zg=0,\\
x^2-\lambda_1x=\lambda_3z,\quad y^2-\lambda_2y=\lambda_4z,\quad z^2=\lambda_5z,\\
xz-zx=\gamma_1x+\gamma_2y+\gamma_3(1-g),\quad yz-zy=\gamma_4x+\gamma_5y+\gamma_6(1-g),\\
xy-yx-\lambda_2x+\lambda_1y=\lambda_6z.
\end{gather*}
for $\lambda_1,\lambda_2,\lambda_5\in\I_{0,1}$ and $\lambda_3,\lambda_4,\lambda_6,\gamma_1,\cdots,\gamma_6\in\K$ with the  conditions  given by \eqref{eq:x1y1z0-1}--\eqref{eq:x1y1z0-7}.

Suppose that $\lambda_1=0=\lambda_2$. Then by rescaling $x,y$, we can take $\lambda_3,\lambda_4\in\I_{0,1}$.

If $\lambda_3=0=\lambda_4$, then  $\lambda_6\gamma_i=0$ for all $i\in\I_{0,1}$ and by rescaling $x$, we can take $\lambda_6\in\I_{0,1}$.

If $\lambda_6=1$, then $\gamma_i=0$ for all $i\in\I_{1,6}$, that is, $[x,z]=0=[y,z]$ in $H$. Then $H$ depends on $\lambda_6\in\I_{0,1}$, that is, $H$ is isomorphic to one of the Hopf algebras described in $(125)$--$(126)$.

If $\lambda_6=0=\lambda_5$, then $\gamma_1^2=\gamma_2\gamma_4=\gamma_5^2$, $\gamma_5\gamma_6=\gamma_3\gamma_4$, $\gamma_1\gamma_3=\gamma_2\gamma_6$, $(\gamma_1-\gamma_5)\gamma_2=0=(\gamma_1-\gamma_5)\gamma_4$ and by rescaling $x,y$, we can take $\gamma_2,\gamma_4\in\I_{0,1}$. If $\gamma_2=0=\gamma_4$, then $\gamma_1=0=\gamma_5$ and we can take $\gamma_3,\gamma_6\in\I_{0,1}$. Let $H(\gamma_3,\gamma_6):=H$ for convenience. It is easy to see that $H(0,1)\cong H(1,0)$ by swapping $x$ and $y$ and $H(1,1)\cong H(1,0)$ via the linear translation $y\mapsto y+x$. Observe that $H(0,0)$ is commutative while $H(1,0)$ is not commutative. Therefore, $H$ is isomorphic to one of the Hopf algebras described in $(127)$--$(128)$.
If $\gamma_2-1=0=\gamma_4$, then $\gamma_1=\gamma_5=\gamma_6=0$ and hence we can take $\gamma_3=0$ via the linear translation $y\mapsto y+\gamma_3(1+g)$, which gives one class of $H$ described in $(129)$.
If $\gamma_2=0=\gamma_4-1$, then $H$ is isomorphic to the Hopf algebra described in $(129)$ by swapping $x$ and $y$. If $\gamma_2=1=\gamma_4$,  then $H$ is isomorphic to the Hopf algebra described in $(129)$ via the linear translation $y\mapsto y+x$.

If $\lambda_6=0=\lambda_5-1$, then $(1-\gamma_1)\gamma_1=\gamma_2\gamma_4=(1-\gamma_5)\gamma_5$, $(1+\gamma_1+\gamma_5)\gamma_2=0=(1+\gamma_1+\gamma_5)\gamma_4$, $(1-\gamma_1)\gamma_3=\gamma_2\gamma_6$, $(1-\gamma_5)\gamma_6=\gamma_3\gamma_4$. If $\gamma_2=0=\gamma_4$, then $\gamma_1,\gamma_5\in\I_{0,1}$, $(1-\gamma_1)\gamma_3=0=(1-\gamma_5)\gamma_6$. Moreover, we can take $\gamma_3=0=\gamma_6$. Indeed, if $\gamma_1=0$ or $\gamma_5=0$, then $\gamma_3=0$ or $\gamma_6=0$; if $\gamma_1=1$ or $\gamma_5=1$, then we can take $\gamma_3=0$ or $\gamma_6=0$ via the linear translation $x:=x+\gamma_3(1-g)$ or $y:=y+\gamma_6(1-g)$. Observe that the Hopf algebras with $\gamma_1-1=0=\gamma_5$ and $\gamma_1=0=\gamma_5-1$ are isomorphic by swapping $x$ and $y$. Then $H$ is isomorphic to one of the Hopf algebras described in $(130)$--$(132)$.
If $\gamma_2-1=0=\gamma_4$, then $\gamma_1,\gamma_5\in\I_{0,1}$, $\gamma_1+\gamma_5=1$, $(1-\gamma_1)\gamma_3=\gamma_6$, $(1-\gamma_5)\gamma_6=0$. If $\gamma_1=1$, then $\gamma_5=0=\gamma_6$ and hence $H$ is isomorphic to the Hopf algebra described in $(131)$ via the linear translation $x\mapsto x+y+\gamma_3(1-g)$. If $\gamma_1=0$, then $\gamma_5=1$, $\gamma_3=\gamma_6$ and hence $H$ is isomorphic to the Hopf algebra described in $(131)$ via the linear translation $x\mapsto y+\gamma_3(1-g),y\mapsto x+y+\gamma_3(1-g)$. Similarly, if $\gamma_2=\gamma_4-1$ or $\gamma_2=1=\gamma_4$, $H$ is isomorphic to the Hopf algebra described in $(131)$.

If $\lambda_3-1=0=\lambda_4$, then $\gamma_i=0$ for all $i\in\I_{1,6}$ and hence $H$ is isomorphic to one of the Hopf algebras described in $(133)$--$(136)$.
If $\lambda_3=0=\lambda_4-1$ or $\lambda_3=1=\lambda_4$, then similar to the last case, $H$ is isomorphic to one of the Hopf algebra described in $(133)$--$(136)$.

Suppose that $\lambda_1-1=0=\lambda_2$. Then $\gamma_1=0=\gamma_4$ and by rescaling $y$, we can take $\lambda_4\in\I_{0,1}$.

If $\lambda_4=0$, then $\lambda_6\gamma_i=0$ for all $i\in\I_{1,6}-\{3\}$ and by rescaling $y$, we can take $\lambda_6\in\I_{0,1}$. Observe that $\gamma_3=\lambda_3\gamma_6$ and $\lambda_3\gamma_3=0$.  If $\lambda_6=1$, then $\gamma_i=0$ for all $i\in\I_{1,6}$ and we can take $\lambda_3=0$ via the linear translation $x\mapsto x-\lambda_3 y$. Therefore, we obatin two classes of $H$ described in $(137)$--$(138)$.

If $\lambda_6=0=\lambda_5$, then $\lambda_3\gamma_i=0$ for all $i\in\I_{1,6}$. If $\lambda_3=0$, then $\gamma_5=0$,  $\gamma_2\gamma_6=0$ and by rescaling $y,z$, we can take $\gamma_2,\gamma_6\in\I_{0,1}$. If $\gamma_2=0$, then  we can take $\gamma_3\in\I_{0,1}$. Let $H(\gamma_3,\gamma_6):=H$ for convenience. Then it is easy to see that $H(1,1)\cong H(0,1)$ via the linear translation $x\mapsto x+y$. Therefore, $H$ is isomorphic to one of the Hopf algebras described in $(139)$--$(141)$.
If $\gamma_2=1$, then $\gamma_6=0$ and hence we can take $\gamma_3=0$ via the linear translation $y\mapsto y+\gamma_3(1-g)$, which gives one class of $H$ described in $(142)$.
If $\lambda_3\neq 0$, then $\gamma_i=0$ for all $i\in\I_{1,6}$ and we can take $\lambda_3=1$ by rescaling $z$, which gives two classes of $H$ described in $(143)$.

If $\lambda_6=0=\lambda_5-1$, then $\lambda_3\gamma_5=0$, $(1-\gamma_5)\gamma_5=0$, $(1+\gamma_5)\gamma_2=0$, $\gamma_3=\gamma_2\gamma_6$, $(1-\gamma_5)\gamma_6=0$.
If $\gamma_5=1$, then $\lambda_3=0$, $\gamma_3=\gamma_2\gamma_6$ and we can take $\gamma_6=0=\gamma_3$ via the linear translation $y\mapsto y+\gamma_6(1-g)$. Indeed, if $\gamma_2=0$, then $\gamma_3=0$; if $\gamma_2\neq 0$, then $\gamma_3=\gamma_2\gamma_6$ and hence the translation is well-defined. Then $H$ is isomorphic to the Hopf algebra described as follows:
\begin{itemize}
\item $\K\langle g,x,y,z\rangle/(g^2-1,[g,x]-g(1-g),[g,y],[g,z],[x,y]-y,[x,z]-\gamma_2y,[y,z]-y,x^2-x,y^2,z^2-z)$.
\end{itemize}
We can take $\gamma_2=0$ via the linear translation $x\mapsto x+\gamma_2y$. Indeed, it follows by a direct computation that the translation is a well-defined Hopf algebra isomorphism. Therefore, $H$ is isomorphic to the Hopf algebra described in $(144)$.
If $\gamma_5=0$, then $\gamma_2=0=\gamma_3=\gamma_6$, and hence $H\cong\HH_{14}(\lambda_3)$ described in $(145)$.

If $\lambda_4=1$, then $\gamma_i=0$ for $i\in\I_{1,6}$. If $\lambda_5=0=\lambda_6$, then  we can take $\lambda_3=0$ via the linear translation $x\mapsto x+\alpha y$ satisfying $\alpha^2=\lambda_3$, which gives one class of $H$ described in $(146)$.
If $\lambda_5=0$ and $\lambda_6\neq 0$, then by rescaling $y,z$, we can take $\lambda_6=1$. Moreover, we can take $\lambda_3=0$ via the linear translation $x\mapsto x+\alpha y$ satisfying $\alpha^2+\alpha=\lambda_3$, which gives one class of $H$ described in $(147)$.
If $\lambda_5=1$, then  we can take $\lambda_3=0$ via the linear translation $x\mapsto x+\alpha y$ satisfying $\alpha^2+\lambda_6\alpha=\lambda_3$ and hence $H\cong\HH_{15}(\lambda_6)$ described in $(148)$.

Suppose that $\lambda_1=0=\lambda_2-1$ or $\lambda_1=1=\lambda_2$. Then it can be reduced to the case $\lambda_1-1=0=\lambda_2$ by swapping $x$ and $y$ or  via the linear translation $y\mapsto x+y$, respectively.

\textbf{Claim:} $\HH_{14}(\lambda)\cong\HH_{14}(\gamma)$ or $\HH_{15}(\lambda)\cong\HH_{15}(\gamma)$, if and only if, $\lambda=\gamma$.

Suppose that $\phi:\HH_{15}(\lambda)\rightarrow \HH_{15}(\gamma)$ for $\lambda,\gamma\in\K$ is a Hopf algebra isomorphism. Then $\phi|_{\Z_2}:\Z_2 \rightarrow \Z_2$ is an automorphism.   Write $g^{\prime},x^{\prime},y^{\prime},z^{\prime}$ to distinguish the generators of $\HH_{15}(\gamma)$.
Since spaces of skew-primitive elements of $\HH_{15}(\gamma)$ are trivial except $\Pp_{1, g^{\prime} }(\HH_8(\gamma))=\K\{x^{\prime},y^{\prime}\}\oplus \K\{1- g^{\prime} \}$ and $\Pp(\HH_{15}(\gamma))=\K\{z^{\prime}\}$, it follows that
$$\phi(g)=g^{\prime}, \quad \phi(x)=\alpha_1x^{\prime}+\alpha_2y^{\prime}+\alpha_3(1-g^{\prime}),\quad\phi(y)=\beta_1x^{\prime}+\beta_2y^{\prime}+\beta_3(1-g^{\prime}),\quad \phi(z)=kz^{\prime}$$ for some $\alpha_i,\beta_i, k\in\K$ and $i\in\I_{1,3}$. Then applying $\phi$ to the relations $gx-xg=g(1-g)$, $z^2=z$, $x^2=x$ and $[g,y]=0$, we have
$$\alpha_1=0,\quad k=1,\quad \alpha_2^2+\alpha_2\gamma=0,\quad \beta_1=0.$$
Then applying $\phi$ to the relation $[x,y]-y-\lambda z$, we have
\begin{align*}
\lambda=\gamma.
\end{align*}
 Conversely, it is easy to see that   $\HH_{15}(\lambda)\cong\HH_{15}(\gamma)$ if  $\lambda=\gamma$. Similarly, $\HH_{14}(\lambda)\cong\HH_{14}(\gamma)$  if and only if  $\lambda=\gamma$.

Assume that $(i,j,k)=(1,0,0)$. Then by Theorem \ref{thm:p4-x1y0z0}, $H$ is isomorphic to one of the Hopf algebras described in $(149)$--$(183)$.

Assume that $V\cong   M_{0,1}\oplus M_{0,2}$. Then $\gr H=\K\langle g,x,y,z\rangle$, subject to the relations
\begin{gather*}
g^2=x^2=y^2=z^2=1,\quad [g,x]=[g,y]=[x,y]=[x,z]=[y,z]=0,\quad gz-(z+y)g=0,
\end{gather*}
with $g\in\G(\gr H),x,y,z\in\Pp(\gr H)$. It follows by a direct computation that
\begin{gather*}
gx-xg=0,\quad gy-yg=0,\quad gz-(z+y)g=0;\\
x^2,y^2,z^2,[x,y],[x,z],[y,z]\in\Pp(H).
\end{gather*}
Then $H\cong\K[\Z_2]\sharp U(\Pp(H))$, where $U(\Pp(H))$ is the restricted universal enveloping algebra of $\Pp(H)$. Then by \cite[Theorem 1.4]{NWW1}, we obtain fourteen classes of $H$ described in  $(184)$--$(197)$.

\vskip10pt \centerline{\bf ACKNOWLEDGMENT}

\vskip10pt The essential part of this article was written during the visit of the author to University of Padova supported by China Scholarship Council. The author is partially supported by the NSFC (Grant No. 11771142,11926353). The author would like to thank his  supervisors Profs. G. Carnovale, N. Hu and Prof. G. A. Garcia so much for the  help and  encouragement. The author thanks the referee for careful reading and helpful comments.

\end{document}